\documentclass[a4paper, 11pt]{article}

\usepackage{amsmath}
\usepackage{amsthm}
\usepackage{amsfonts}
\usepackage{amssymb}
\usepackage{amscd}
\usepackage[dvips]{graphicx, color}
\newtheorem{thm}{Theorem}[section]
\newtheorem{cor}[thm]{Corollary}
\newtheorem{defn}[thm]{Definition}
\newtheorem{lem}[thm]{Lemma}
\newtheorem{prop}[thm]{Proposition}
\newtheorem{rem}[thm]{Remark}
\newtheorem{exam}[thm]{Example}

\numberwithin{equation}{section}

\begin{document}
\title{\bf \Large A Bridge between Conway-Coxeter Friezes and Rational Tangles through the Kauffman Bracket Polynomials}
\author{
{Takeyoshi Kogiso}\\
{\footnotesize  Department of Mathematics, Josai University, }\\
{\footnotesize  1-1, Keyakidai Sakado, Saitama, 350-0295, Japan}\\
{\footnotesize  E-mail address: kogiso@josai.ac.jp}\\  
\\ 
{Michihisa Wakui}\\ 
{\footnotesize  Department of Mathematics, Faculty of Engineering Science,}\\ 
{\footnotesize  Kansai University, Suita-shi, Osaka 564-8680, Japan}\\
{\footnotesize  E-mail address: wakui@kansai-u.ac.jp}\\ 
}
\date{}

\maketitle 
\begin{abstract}
In the present paper, we build a bridge between Conway-Coxeter friezes and rational tangles through the Kauffman bracket polynomials.  
One can compute a Kauffman bracket polynomials attached to rational links by using Conway-Coxeter friezes. 
As an application one can give a complete invariant on Conway-Coxeter friezes of zigzag-type.
\end{abstract}

\baselineskip 14pt 

\section{Introduction}
\par 
A Conway-Coxeter frieze (abbreviated by CCF) was introduced  in 1973 as some frieze pattern, which is given by arranging positive integers under the ``unimodular" rule related to elements of $\text{SL}(2,{\Bbb Z})$, and it was classified via triangulated polygons \cite{CoCo1, CoCo2}. 
CCF's are related to cluster algebras, that are introduced by Fomin and Zelevinsky in \cite{FZ1,FZ2} at the first decades of the $21$st century. 
Cluster algebras are a class of commutative algebras which was shown to be connected to various areas of mathematics, including Lie theory, Poisson geometry, Teichm\"{u}ller theory, mathematical physics and representation theory of algebras. 
A cluster algebra is generated by cluster variables, obtained recursively by a combinatorial process known as mutation starting from cluster variables. 
In order to compute cluster variables, one may use CCF's. 
Various relationships are known between CCF's and cluster algebras of type $A_n$ (\cite{FZ1,FZ2}).
\par 
On the other hand, the Kauffman bracket polynomial is very important tool for studying knot theory and mathematical physics. 
It is a regular isotopy invariant of link diagrams, and the Jones polynomial of links can be recursively computed from it \cite{Kau-Topology}. 
The Kauffman bracket polynomial is naturally extended to that of tangles, which are $1$-dimensional topological objects introduced by Conway \cite{Conway} to classify knots and links in combinatorial and algebraic. 
As an important class of tangles there is a set of rational tangles, that are completely classified by rational numbers.  
\par 
The purpose of the paper is to build a bridge between CCF's and rational tangles through the Kauffman bracket polynomials.  
To do this we need a formula found by Shuji Yamada in 1995 for computing the Kauffman bracket  polynomials of rational tangles by using triangles 
associated with Farey neighbors of rational numbers \cite{Yamada-Proceeding}.  
Unfortunately, it seems that a proof of the formula is not elsewhere. 
In this paper, so after proving of his formula rewritten as our version, 
we show that the Kauffman bracket polynomial can be defined for CCF's of zigzag-type. 
Here, ``zigzag-type" means that CCF's are generated by words consisting of $L$ and $R$ (see Subsection~\ref{subsection2-2} for detail). 
Various properties of the Kauffman bracket polynomial of CCF's are described and shown. 
Furthermore, combining the Conway's classification result of rational tangles with our consideration on the Kauffman bracket polynomial of CCF's, 
a complete invariant of CCF's of zigzag-type is given. 
\par 
We expect for this method to define and compute  Kauffman  bracket polynomials for frieze patterns associated to cluster algebras of other type, and  also expect that our method would be useful to study analytic number theory.
\par 
Kyungyong Lee and Ralf Schiffler \cite{LS} give a very interesting formula to express Jones polynomials for 
$2$-bridge knots as specializations for cluster variables. 
Our method of determining Kauffman bracket polynomials of the $2$-bridge 
links are consistent with the Jones polynomials computed in \cite{LS}.
We will write another paper on the detail and  proof of the facts.
On the other hand, Wataru Nagai and Yuji Terashima \cite{NT} imply that cluster variables know both Alexander polynomials and Jones polynomials 
for $2$-bridge knots through different specializations by using ancestor triangles which is modified Yamada's ancestor triangles. 
\par 
This paper is organized as follows. 
In Section 2, after recalling the definition of Kauffman bracket polynomials of tangle diagrams, 
based on Conway's classification result of rational tangles, we explain the method of construction of the rational tangle diagram $T(\frac{p}{q})$ corresponding to an irreducible fraction $\frac{p}{q}$ via its continued fraction expression. 
We also explain the definition of Yamada's ancestor triangles, and 
give the proof of a reformulated version of Yamada's formula \cite{Yamada-Proceeding} for computing the Kauffman bracket of $T(\frac{p}{q})$ by ancestor triangles. 
\par 
In Section 3, Conway-Coxeter friezes of zigzag-type are discussed, and it is shown that for such a Conway-Coxeter frieze $\Gamma$ the Kauffman bracket polynomial $\langle \Gamma \rangle$ can be defined. 
We give a recursive formula for computing $\langle \Gamma \rangle$ by $LR$ words, and derive various properties of $\langle \Gamma \rangle$. 
\par 
In Section 4, a complete invariant of Conway-Coxeter friezes of zigzag-type is given by using the Kauffman bracket polynomial. 
\par 
In the final section we explain that the main results in Section 3 can be recognized in the viewpoint of  deleting ``trigonometric-curves" in a Conway-Coxeter frieze.

\section{Kauffman bracket polynomials corresponding to fractions}
\par 
There are many known formulae to compute Kauffman bracket polynomials of rational tangles and Jones polynomial of rational links \cite{Kanenobu, Nakabo1, Nakabo2, LLS, QY-AQ}. 
In 1995 Shuji Yamada also found a unique formula to compute such polynomials by using Farey neighbors up to units \cite{Yamada-Proceeding}.  
His formula is great useful to examine relationship with Conway-Coxeter friezes and Markov triples \cite{KW1}. 
In this section we explain his idea and introduce a modified formula with no ambiguity, which is given in Theorems~\ref{w1-8} and \ref{w1-10}. 
Therefore, one can explicitly compute the value of the Kauffman bracket of a given rational tangle diagram by Yamada's formula. 

\par \smallskip 
\subsection{Kauffman brackets of rational tangles} 
In this subsection, we review definitions and properties of rational tangles and the Kauffman bracket polynomials for them.  
For the basic theory of knots and links, we refer the reader to \cite{Adams, Cromwell} for example. 
\par 
First let us recall the recursive definition of the Kauffman bracket polynomial for link diagrams. 
Let $\Lambda $ be the Laurent polynomial ring $\mathbb{Z}[A, A^{-1}]$. 
For each link diagram $D$, the Kauffman bracket polynomial $\langle D\rangle \in \Lambda $ is computed by applying the following rules repeatedly. 
\begin{enumerate}\itemindent=1cm 
\item[(KB1)] $\langle \ \raisebox{-0.15cm}{\includegraphics[width=0.5cm]{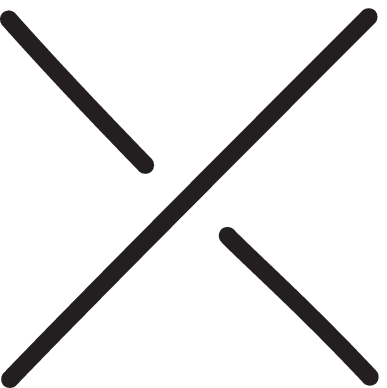}}\ \rangle 
=A\langle \ \raisebox{-0.15cm}{\includegraphics[width=0.5cm]{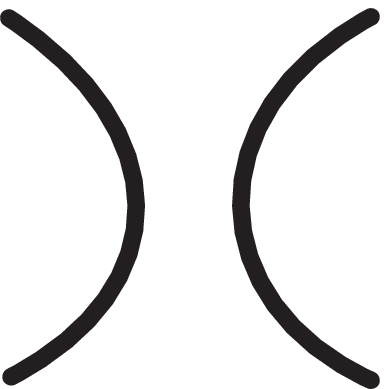}} \ \rangle 
+A^{-1}\langle \ \raisebox{-0.15cm}{\includegraphics[width=0.5cm]{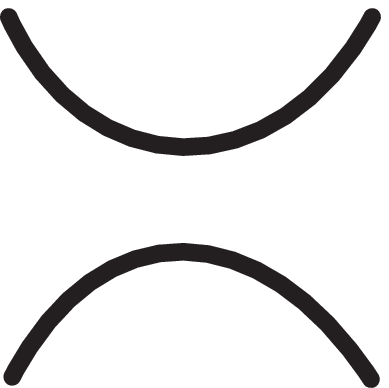}} \ \rangle $ 
\item[(KB2)] $\langle \ D\coprod \raisebox{-0.1cm}{\includegraphics[width=0.4cm]{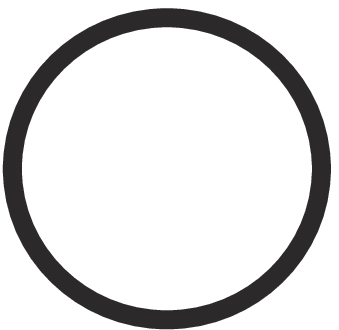}}\ \rangle =\delta \langle  D\rangle $,  where $\delta =-A^2-A^{-2}$.
\item[(KB3)] $\langle \ \raisebox{-0.1cm}{\includegraphics[width=0.4cm]{unoriented_circle.eps}}\ \rangle =1$. 
\end{enumerate}

\par 
Here, although three illustrations in (KB1) only express insides in a small ball,  
the equation (KB1) should be viewed as an equation between Kauffman brackets for three link diagrams such that they are different only in the small ball and identical on the exterior. 
In (KB2), $D\coprod \raisebox{-0.1cm}{\includegraphics[width=0.4cm]{unoriented_circle.eps}}$ expresses a link diagram such that $D$ and a simple closed curve are splittable by a planar isotopy. 
\par 
The Kauffman bracket $\langle D\rangle $ is a regular isotopy invariant of $D$, that is, it is invariant under Reidemeister moves II and III. 
On the other hand, under Reidemeister move I it behaves as 
$$\langle \ \raisebox{-0.15cm}{\includegraphics[height=0.5cm]{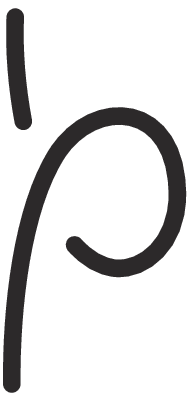}}\ \rangle 
=-A^3\langle \ \raisebox{-0.15cm}{\includegraphics[height=0.5cm]{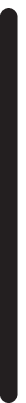}}\ \rangle 
,\qquad 
\langle \ \raisebox{-0.15cm}{\includegraphics[height=0.5cm]{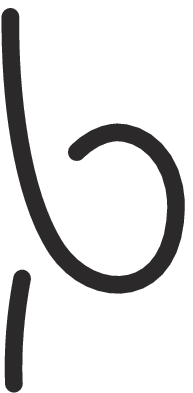}}\ \rangle 
=-A^{-3}\langle \ \raisebox{-0.15cm}{\includegraphics[height=0.5cm]{unoriented_straight_line.eps}}\ \rangle .$$
However, it is well-known by Kauffman \cite{Kau-Topology} that after taking a suitable normalization and changing the parameter 
we have an ambient isotopy invariant of oriented links, which coincides with Jones polynomial. 

\par 
The notion of a tangle (diagram) was introduced by Conway \cite{Conway} to make a list of knots and links in the order of fewer crossing numbers. 
By a tangle we mean a compact $1$-dimensional manifold properly embedded in the standard $3$-ball $B^3$ 
which intersects with the boundary of $B^3$ only at the specific four points $(0 , -\frac{1}{\sqrt{2}}, \frac{1}{\sqrt{2}}), 
(0 , -\frac{1}{\sqrt{2}}, -\frac{1}{\sqrt{2}}),  (0 , \frac{1}{\sqrt{2}}, \frac{1}{\sqrt{2}}), (0 , \frac{1}{\sqrt{2}}, -\frac{1}{\sqrt{2}})$.
Two tangles are called equivalent if one is deformed other by an isotopy of $B^3$ fixing the boundary at pointwise. 
As the same in the case of knots and links, any tangle can be represented by a tangle diagram, 
which obtained by projecting the tangle to the standard disk $D^2$ via the map $\pi : \mathbb{R}^3 \longrightarrow \mathbb{R}^2, \pi (x,y,z)=(y,z)$ as depicted in Figure~\ref{figw1}. 

\begin{figure}[htbp]
\centering 
\includegraphics[height=2cm]{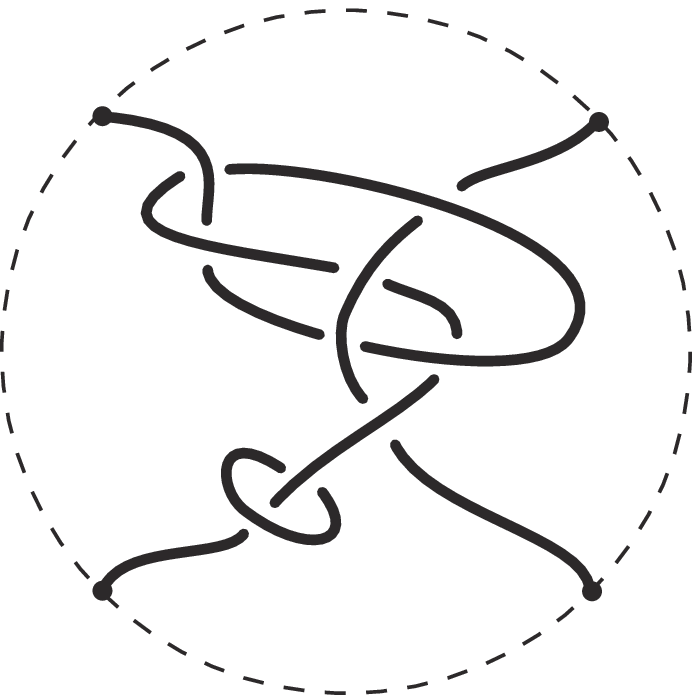}
\caption{}\label{figw1}
\end{figure}

\par 
The Kauffman bracket is also defined for a tangle diagram and can be computed by the same rules (KB1) and (KB2). 
However, in the tangle case, it takes a value in rank $2$ free $\Lambda$-module $\Lambda ^2:=\Lambda [0]+\Lambda [\infty ]$, 
where  $[0]$ and $[\infty ]$ are the two tangle diagrams depicted as in Figure~\ref{figw2}. 

\begin{figure}[htbp]
\centering
\includegraphics[width=2cm]{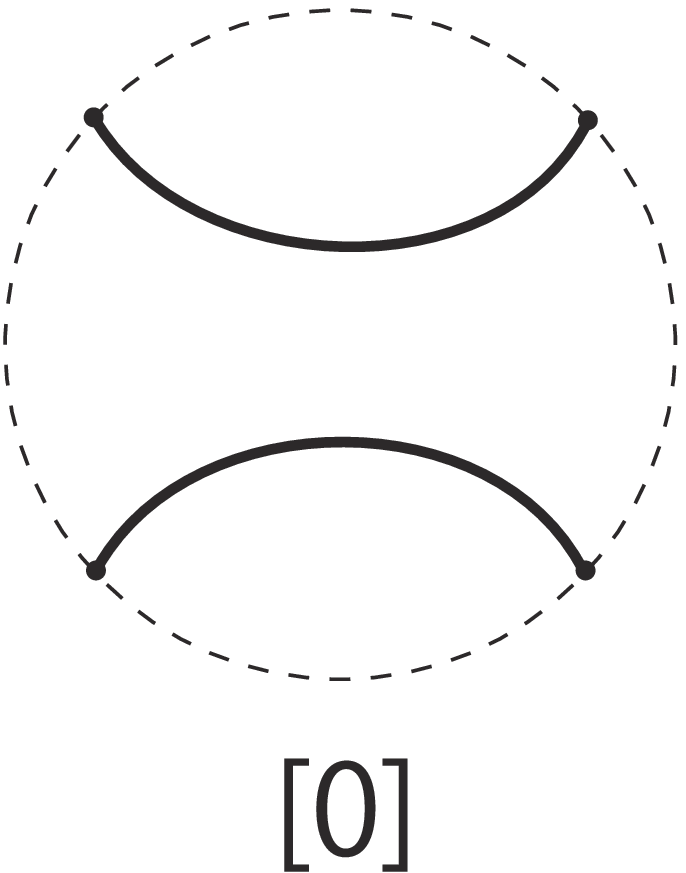}\hspace{2cm} 
\includegraphics[width=2cm]{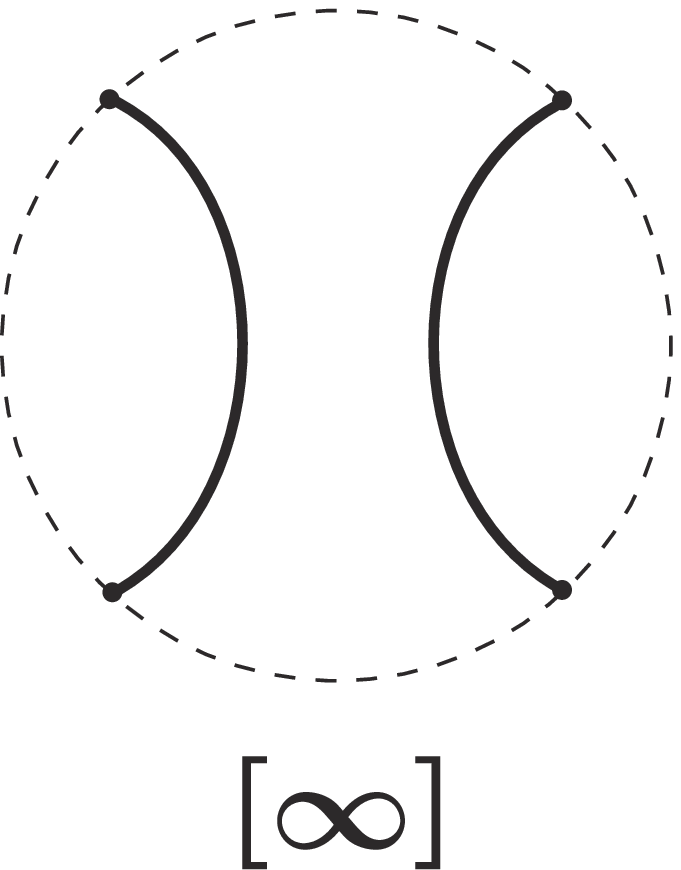}
\caption{}\label{figw2}
\end{figure}

\par 
For a tangle diagram $T$, we denote the Kauffman bracket polynomial by $\langle T\rangle $, and 
define $n_T, d_T\in \Lambda $ by 
\begin{equation}
\langle T\rangle =n_T[\infty ]+d_T[0], 
\end{equation}
which are also regular isotopy invariants. 

\par 
Let $T$ be a tangle diagram. 
By closing endpoints of $T$, two link diagrams $N(T)$ and $D(T)$ are obtained as in Figure~\ref{figw3}. 
$N(T)$ and $D(T)$ are called the numerator and the denominator of $T$, respectively. 

\begin{figure}[hbtp]
\centering
\includegraphics[height=5cm]{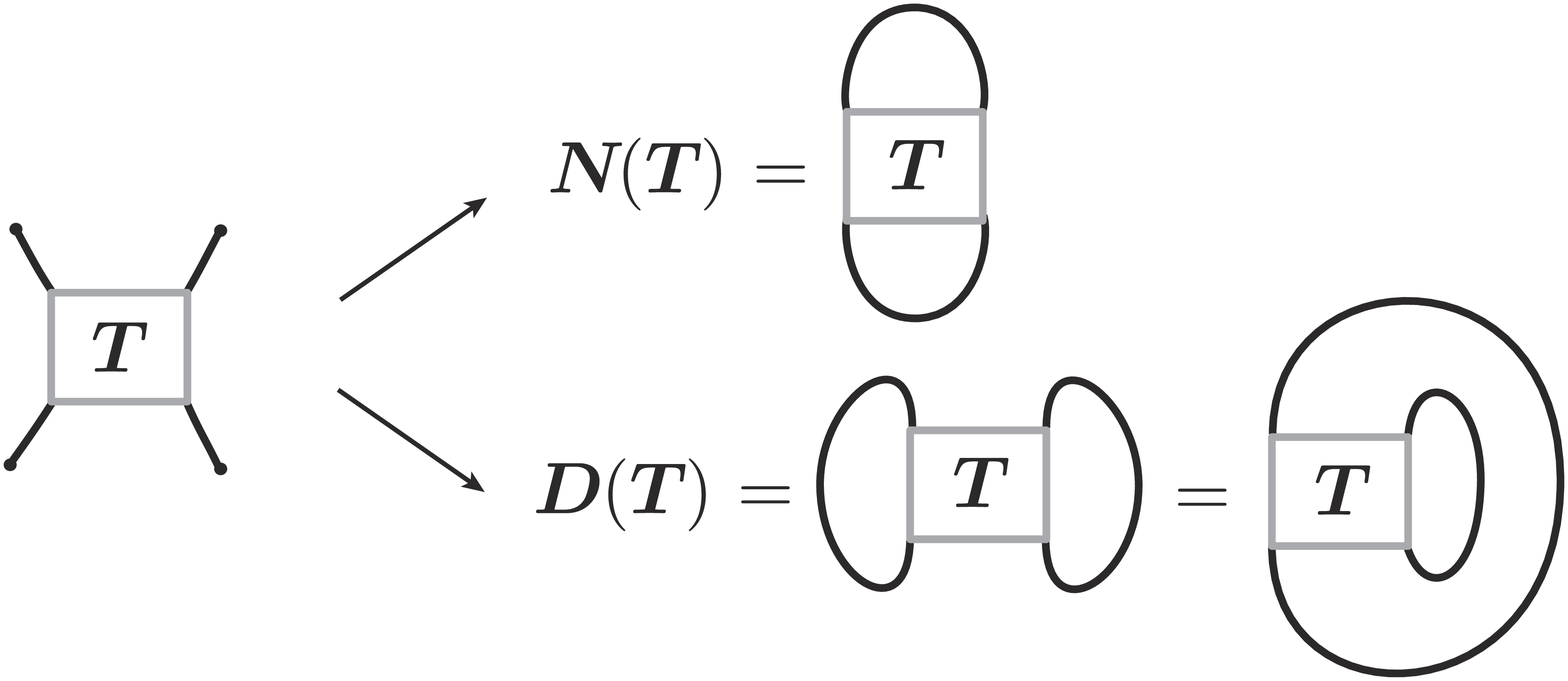}
\caption{}\label{figw3}
\end{figure}

The Kauffman bracket polynomials of $N(T)$ and $D(T)$ can be computed from that of $T$ as follows. 
\begin{align}
\langle N(T)\rangle &=n_T+d_T\delta , \label{eqw1-2}\\ 
\langle D(T)\rangle &=n_T\delta +d_T. \label{eqw1-3}
\end{align}

Let us consider the two maps $v:\Lambda ^2 \longrightarrow \Lambda $ and $(-)^{\text{rot}}: \Lambda ^2 \longrightarrow \Lambda ^2$ 
defined by
\begin{align}
v(a[\infty ]+b[0])&=a\delta +b, \label{eqw1-4}\\ 
(a[\infty ]+b[0])^{\textrm{rot}}&=b[\infty ]+a[0]
\end{align}
for $a, b\in \Lambda $. 
Then, the equations \eqref{eqw1-2} and \eqref{eqw1-3} can be rewritten as
\begin{align}
\langle N(T)\rangle &=v\bigl(\langle T\rangle ^{\textrm{rot}}\bigr) , \label{eqw1-6}\\ 
 \langle D(T)\rangle &=v\bigl(\langle T\rangle \bigr) . \label{eqw1-7}
\end{align}

\par \medskip 
There are three operations on tangle diagrams, which are called the mirror image, rotation and inversion denoted by $-T, T^{\textrm{rot}}, T^{\textrm{in}}$ for a tangle diagram $T$, respectively.
$-T$ and $T^{\textrm{rot}}$ are given by
\begin{align*}
-T&:=(\text{the tangle diagram obtained from $T$ by changing all crossings}),  \\ 
T^{\textrm{rot}}&:=(\text{the tangle diagram obtained from $T$ by rotating $90^{\circ}$ in } \\ 
& \hspace{0.8cm} \text{counterclockwise direction}), 
\end{align*}
\noindent 
and $T^{\textrm{in}}$ is given by the composition $(-T)^{\text{rot}}=-(T^{\text{rot}})$, that is pictorially represented as 

$$T^{\textrm{in}}\ =\ \ \raisebox{-1.2cm}{\includegraphics[height=2.5cm]{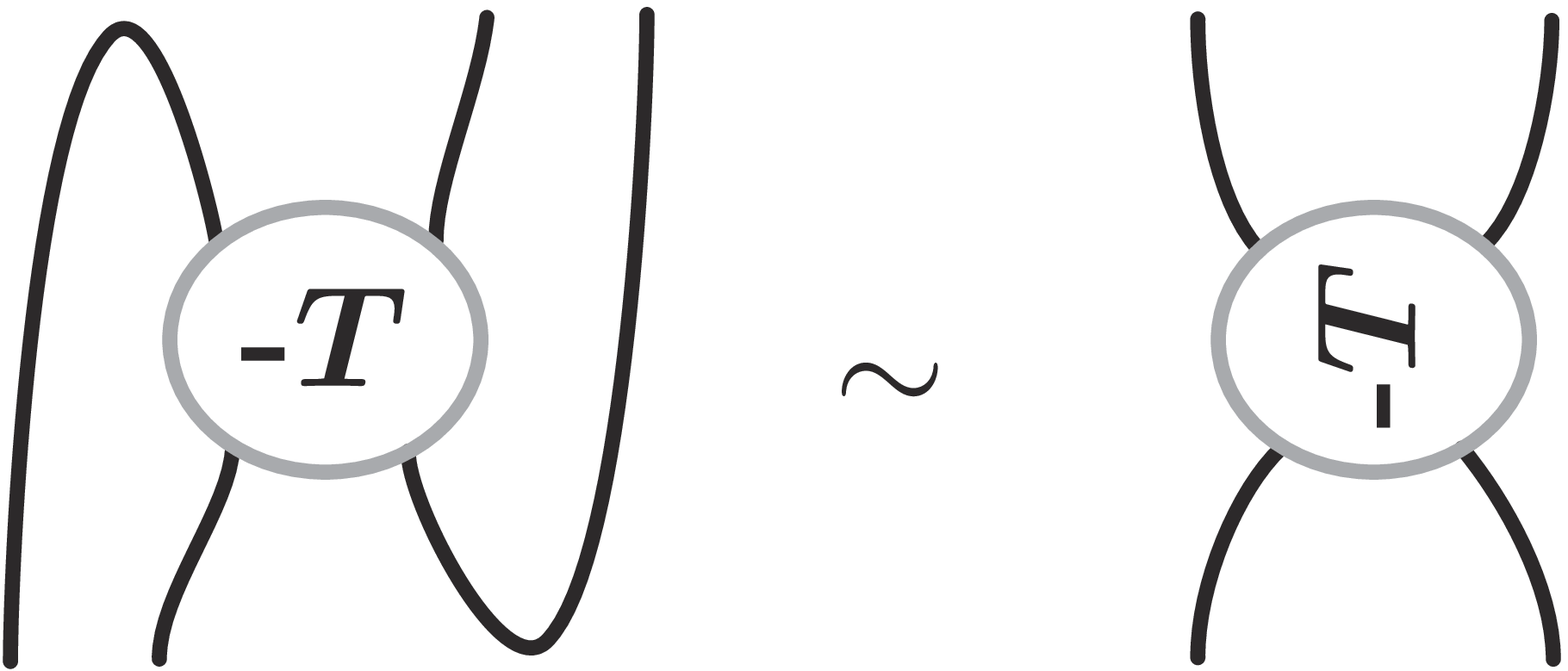}}.$$

Let 
$\overline{(-)}$ be the $\mathbb{Z}$-linear map  from  $\Lambda $ to $\Lambda $ defined by 
$$\overline{A}=A^{-1},\quad \overline{A^{-1}}=A,$$
and define a map ${(-)}^{\text{in}} : \Lambda ^2 \longrightarrow \Lambda ^2$ by 
\begin{equation}
(a[\infty ]+b[0])^{\textrm{in}}=\overline{b}[\infty ]+\overline{a}[0].
\end{equation}

\par \medskip 
\begin{lem}\label{w1-1}
For $\langle T\rangle =n_T[\infty ]+d_T[0]$, 
\begin{enumerate}
\item[$(1)$] $\langle -T\rangle =\overline{n_T}[\infty ]+\overline{d_T}[0]$.
\item[$(2)$] $\langle T^{\textrm{rot}}\rangle =d_T[\infty ]+n_T[0]\ (=\langle T\rangle ^{\textrm{rot}})$.
\item[$(3)$] $\langle T^{\textrm{in}}\rangle =\overline{d_T}[\infty ]+\overline{n_T}[0]\ (=\langle T\rangle ^{\textrm{in}})$. 
\end{enumerate}
\end{lem} 

\par \medskip 
For an integer $n$, we denote by $[n]$ and $\frac{1}{[n]}$ the tangle diagrams depicted in Figure~\ref{figw4}. 

\begin{figure}[htbp]
\centering
\includegraphics[height=7cm]{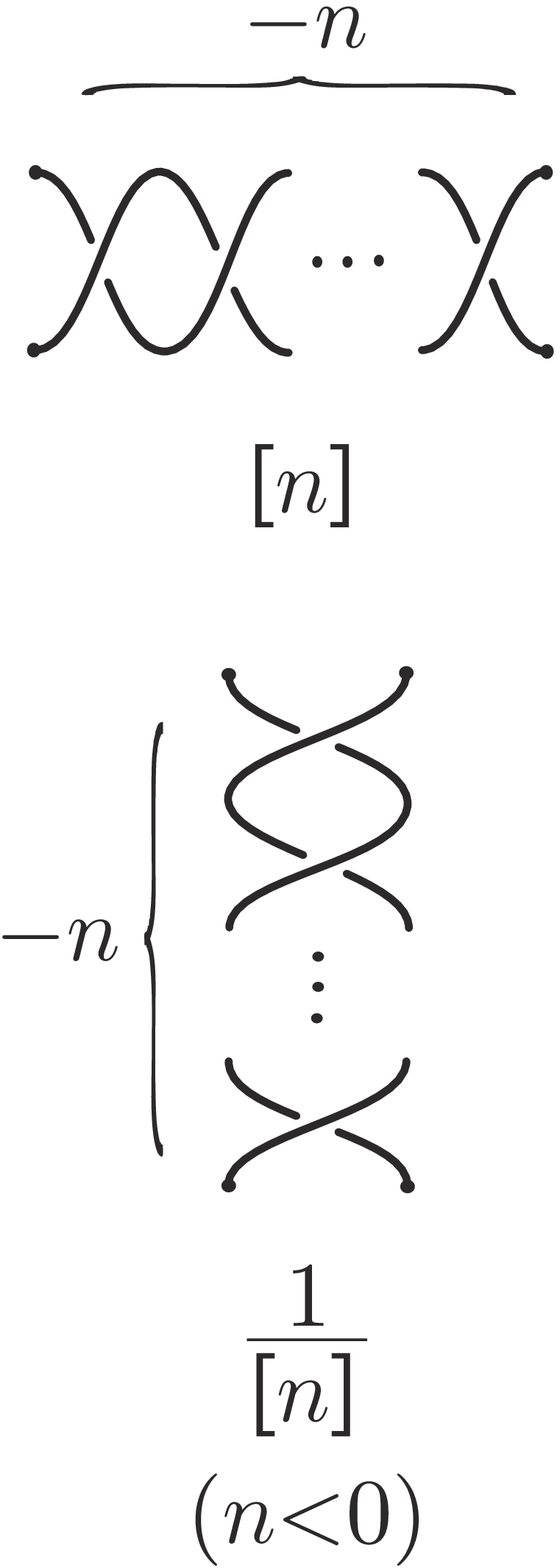}\hspace{1cm}   
 \includegraphics[height=7cm]{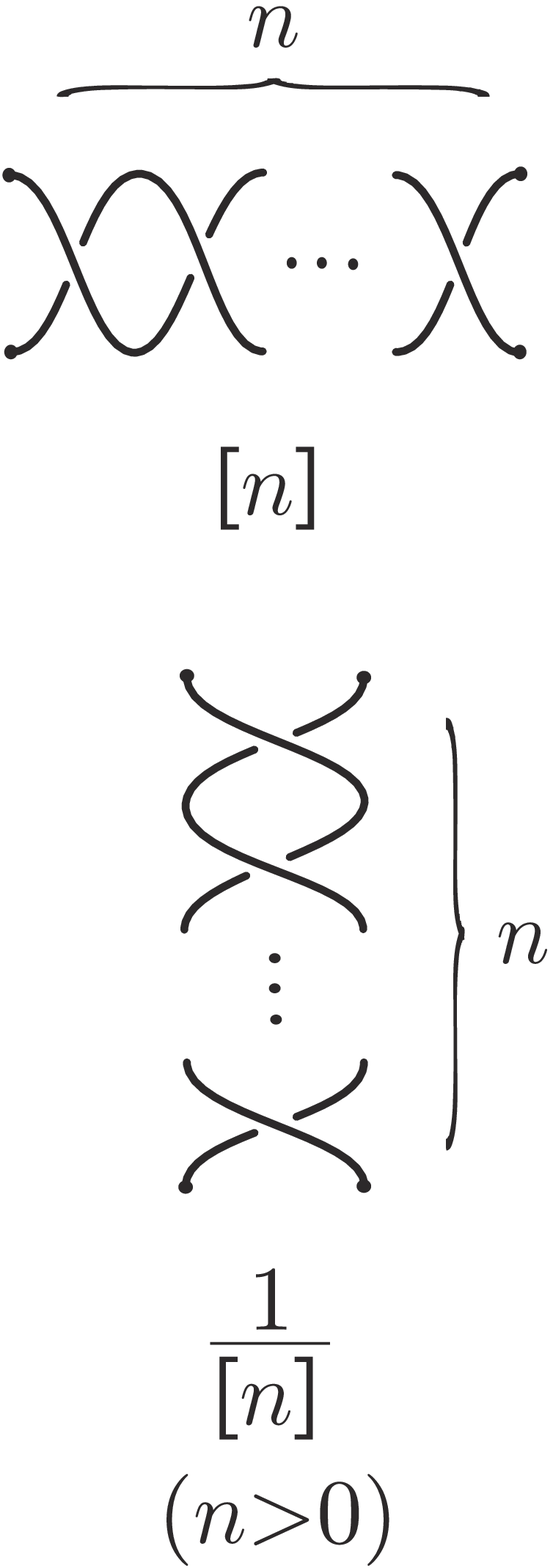}
\caption{}\label{figw4}
\end{figure}

The tangle diagrams $[n]$ and $\frac{1}{[n]}$ are called an integer tangle and a vertical tangle, respectively. 
It may be regarded as $\frac{1}{[0]}=[\infty ],\ \frac{1}{[\infty ]}=[0]$. 

\par \smallskip 
For tangle diagrams  $T, U$ we have tangle diagrams $T\bowtie U$ and $T\ast U$ defined by 
$$T\bowtie U\ :=\  \raisebox{-1cm}{\includegraphics[height=2cm]{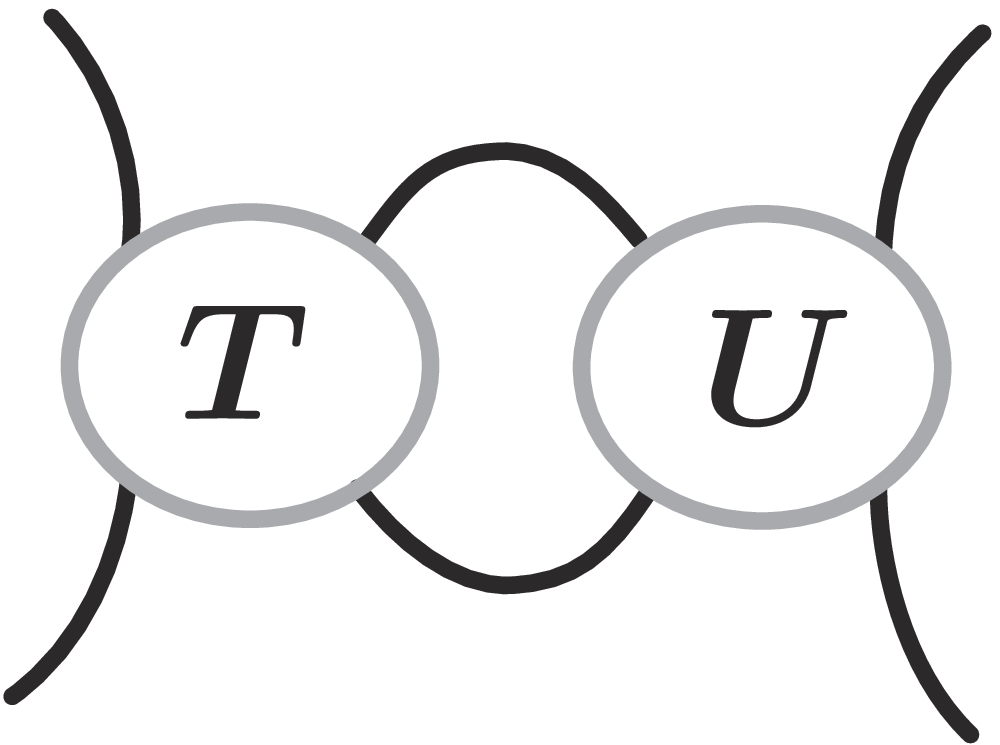}} \hspace{0.3cm} ,\qquad 
T\ast U\ :=\  \raisebox{-1.3cm}{\includegraphics[width=1.8cm]{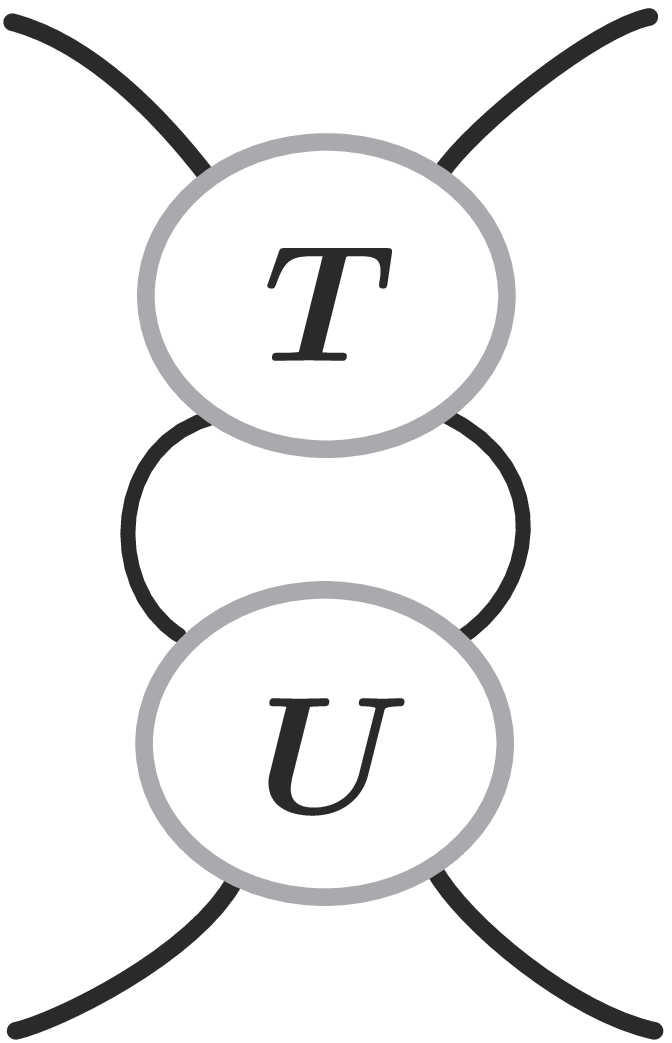}}\ \ .
$$

The tangle diagrams $T\bowtie U$ and $T\ast U$ are called the sum and the product of $T, U$. 
By definition the following equations hold. 

\begin{equation}\label{eqw1-9}
\begin{aligned}
d_{T\bowtie U}&=d_Td_U, &
n_{T\bowtie U}&=n_Td_U+d_Tn_U+n_Tn_U\delta ,\\ 
d_{T\ast U}&=d_Tn_U+n_Td_U+d_Td_U\delta , & 
n_{T\ast U}&=n_Tn_U. 
\end{aligned}
\end{equation}

\par \medskip 
\begin{lem}\label{w1-2}
For a positive integer $n$, 
\begin{enumerate}
\item[$(1)$] $\langle [n] \rangle =A^{n-2}\sum_{k=0}^{n-1}(-A^{-4})^k\ [\infty ] + A^n\ [0]$ \ $\Bigl(= A^{n-2}\frac{1-(-A^{-4})^n}{1+A^{-4}}\ [\infty ] + A^n\ [0] \Bigr)$.
\item[$(2)$] $\langle \frac{1}{[n]} \rangle = A^{-n}\ [\infty ] + A^{-n+2}\sum_{k=0}^{n-1}(-A^{4})^k\ [0]$ \ $\Bigl(=A^{-n}\ [\infty ] + A^{-n+2}\frac{1-(-A^{4})^n}{1+A^4}\ [0] \Bigr)$.
\item[$(3)$] $v\bigl( \langle [n] \rangle \bigr) =(-A^{-3})^n$.  
\end{enumerate}
\end{lem}

\par \medskip 
By \eqref{eqw1-9} and Lemma~\ref{w1-2} we have: 

\par \medskip 
\begin{lem}\label{w1-3}
Let $T$ be a tangle diagram, and $n$ be a positive integer. Then 
\begin{align*}
\Bigl\langle T \bowtie \frac{1}{[n]} \Bigr\rangle 
&= (A^{-n+2}\sum\limits_{k=0}^{n-1}(-A^4)^kn_T+\delta A^{-n}n_T+A^{-n}d_T)[\infty ]+A^{-n+2}\sum\limits_{k=0}^{n-1}(-A^4)^kd_T[0] \\ 
&=A^{-n}v\bigl( \langle T\rangle \bigr) [\infty ]+d\Bigl(\frac{1}{n}\Bigr) \langle T\rangle , 
\end{align*}
where $d(\frac{1}{n}):=d_{\frac{1}{[n]}}$.  \qed 
\end{lem}

\par \medskip 
Conway \cite{Conway} also introduced the notion of a rational tangle, 
and showed that there is a one-to-one correspondence between the set of rational tangles up to isotopy and $\mathbb{Q}\cup \{ \infty \}$.
The correspondence is given through continued fractions as follows. 
\par 
Let $\frac{p}{q}$ be an irreducible fraction, and 
expand it as a continued fraction  
\begin{equation}\label{eqw1-10}
\begin{aligned}
\frac{p}{q}=a_0+\dfrac{1}{\vbox to 18pt{ }a_1+\dfrac{1}{\vbox to 18pt{ }a_2+\dfrac{1}{\vbox to 18pt{ }\ddots +\dfrac{1}{\vbox to 18pt{ }a_{n-1}+\dfrac{1}{a_n}}}}}, 
\end{aligned}
\end{equation}

\noindent 
where $a_0\in \mathbb{Z},\ a_1, a_2, \ldots , a_n\in \mathbb{N}$. 
We denote the right-hand side of \eqref{eqw1-10} by $[a_0; a_1, a_2, \ldots , a_n]$. 
Note that the expansion \eqref{eqw1-10} is unique if the parity of $n$ is specified. 
We choose $n$ as an even integer, and define a tangle diagram $T(\frac{p}{q})$ by \vspace{-0.2cm} 
\begin{equation}\label{eqw1-11}
T\Bigl( \frac{p}{q}\Bigr) :=\ \raisebox{-1.5cm}{\includegraphics[height=3cm]{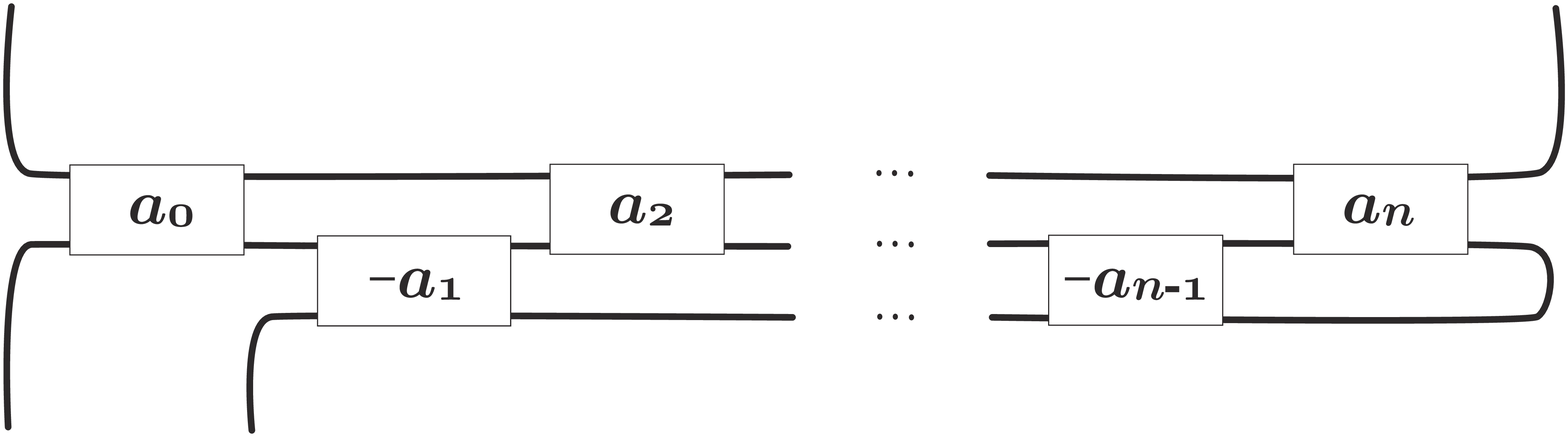}}, 
\end{equation}

\noindent 
where \  $\raisebox{-0.15cm}{\includegraphics[height=0.5cm]{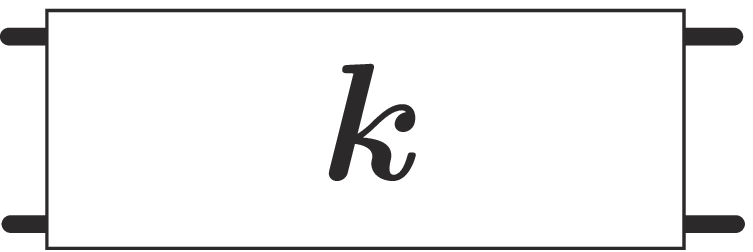}} \ =\ [k]$ for each $k\in \mathbb{Z}$. 

\par \bigskip 
Following Kauffman and Lambropoulou \cite{KL3} we denote 
the tangle diagram depicted in the right-hand side of \eqref{eqw1-11} by $[[a_0], [a_1], [a_2], \ldots , [a_n]]$. 
A tangle diagram which is regular isotopic to such a diagram is called a rational tangle diagram. 
For example, integer and vertical tangle diagrams are rational. 
By Conway's classification result \cite{Conway} of rational tangles, 
the isotopy class of $T( \frac{p}{q})$ is uniquely determined by the irreducible fraction $\frac{p}{q}$ (see also \cite[Chapter 8]{Cromwell}, \cite{BS}, \cite{GK2} and \cite{KL3}). 

\par \medskip 
\begin{exam}\label{w1-4} 
Let $m, n$ be positive integers. 
\begin{enumerate}
\item[$(1)$] By Lemma~\ref{w1-2}(1) and (2)
\begin{align*}
\Bigl\langle T\Bigl( \frac{1}{n}\Bigr)\Bigr\rangle \ 
&=\Bigl\langle\frac{1}{[n]}\Bigr\rangle \ 
=A^{-n}\Bigl\{ [\infty ]+A^{2}\sum\limits_{k=0}^{n-1}(-A^4)^k[0] \Bigr\} , \\ 
\Bigl\langle T\Bigl( \dfrac{n}{1}\Bigr)\Bigr\rangle \ 
&=\langle [n]\rangle \ 
=A^n\Bigr\{ A^{-2}\sum\limits_{k=0}^{n-1}(-A^{-4})^{k}[\infty ]+[0]\Bigr\}  . 
\end{align*}
\item[$(2)$] By Lemmas~\ref{w1-2}(3) and \ref{w1-3} 
$$
\Bigl\langle T\Bigl( m+\frac{1}{n}\Bigr)\Bigr\rangle \ 
\ =\ \Bigl\langle [m]\bowtie \frac{1}{[n]} \Bigr\rangle \ =A^{-n}(-A^{-3})^m[\infty ]+d\Bigl( \frac{1}{n}\Bigr)\langle [m]\rangle . 
$$
\end{enumerate}
\end{exam} 

\par \smallskip 
\subsection{Farey sums of fractions and the Stern-Brocot tree}\label{subsection1-3}
\par 
From now on, an irreducible fraction $\frac{p}{q}$ is always assumed to be $q\geq 0$, and if $q=0$, then $p=1$. 
\par 
Two irreducible fractions $\frac{p}{q}$ and $\frac{r}{s}$ are said to be \textit{Farey neighbors} if 
they satisfy $|ps-qr|=1$. 
Then the fraction
\begin{equation}\label{eq_Fareysum}
\frac{p}{q}\sharp \frac{r}{s}:=\frac{p+r}{q+s}
\end{equation}
is irreducible, and both $\frac{p}{q}, \frac{p}{q}\sharp \frac{r}{s}$ and $\frac{p}{q}\sharp \frac{r}{s}, \frac{r}{s}$ are Farey neighbors, again. 
The fraction \eqref{eq_Fareysum} is called the \textit{Farey sum} of $\frac{p}{q}$ and $\frac{r}{s}$. 
In this paper Farey neighbors $\frac{p}{q}$ and $\frac{r}{s}$ are always assumed to be arranged in ascending order, that is, $qr-ps=1$. 

\par \medskip 
\begin{lem}\label{w1-5}
\begin{enumerate}
\item[$(1)$] Let $\frac{p}{q}$ and $\frac{r}{s}$ be Farey neighbors that are not $\infty$. 
Then, 
$\frac{p}{q}\sharp \frac{r}{s}$ is the unique fraction that the absolute values of the numerator and the denominator are minimum between the numerators and the denominators of the irreducible fractions in the open interval $(\frac{p}{q}, \frac{r}{s})$, respectively. 
\item[$(2)$] Non-negative rational numbers are generated by operating $\sharp$ from $\frac{0}{1}$ and $\frac{1}{0}$. 
\item[$(3)$] For any nonzero rational number $\alpha $, there is a unique pair $(\frac{p}{q}, \frac{r}{s})$ of Farey neighbors which satisfies $\alpha =\frac{p}{q}\sharp \frac{r}{s}$. 
The pair $(\frac{p}{q}, \frac{r}{s})$ is called the parents of $\alpha $, and 
$\alpha $ is called the mediant of $(\frac{p}{q}, \frac{r}{s})$. 
\end{enumerate}
\end{lem}

For a proof of the above lemma see \cite[Theorem 3.9]{Aigner} or \cite[Lemma 3.5]{KW1}. 
\par 
For study on (positive) rational numbers the Stern-Brocot tree \cite{Brocot, Stern} is useful.  
This tree is an infinite binary tree in which the vertices correspond to the positive rational numbers, and the root of the tree is the rational number $\frac{1}{1}$ depicted as in Figure~\ref{figw5}. 
Two rational numbers are combined by an edge if they have a relation as the nearest parent and its mediant. 
By Lemma~\ref{w1-5} (3), all positive rational numbers appear in the Stern-Brocot tree as vertices. 

\begin{figure}[htbp]
\centering 
\includegraphics[width=10cm]{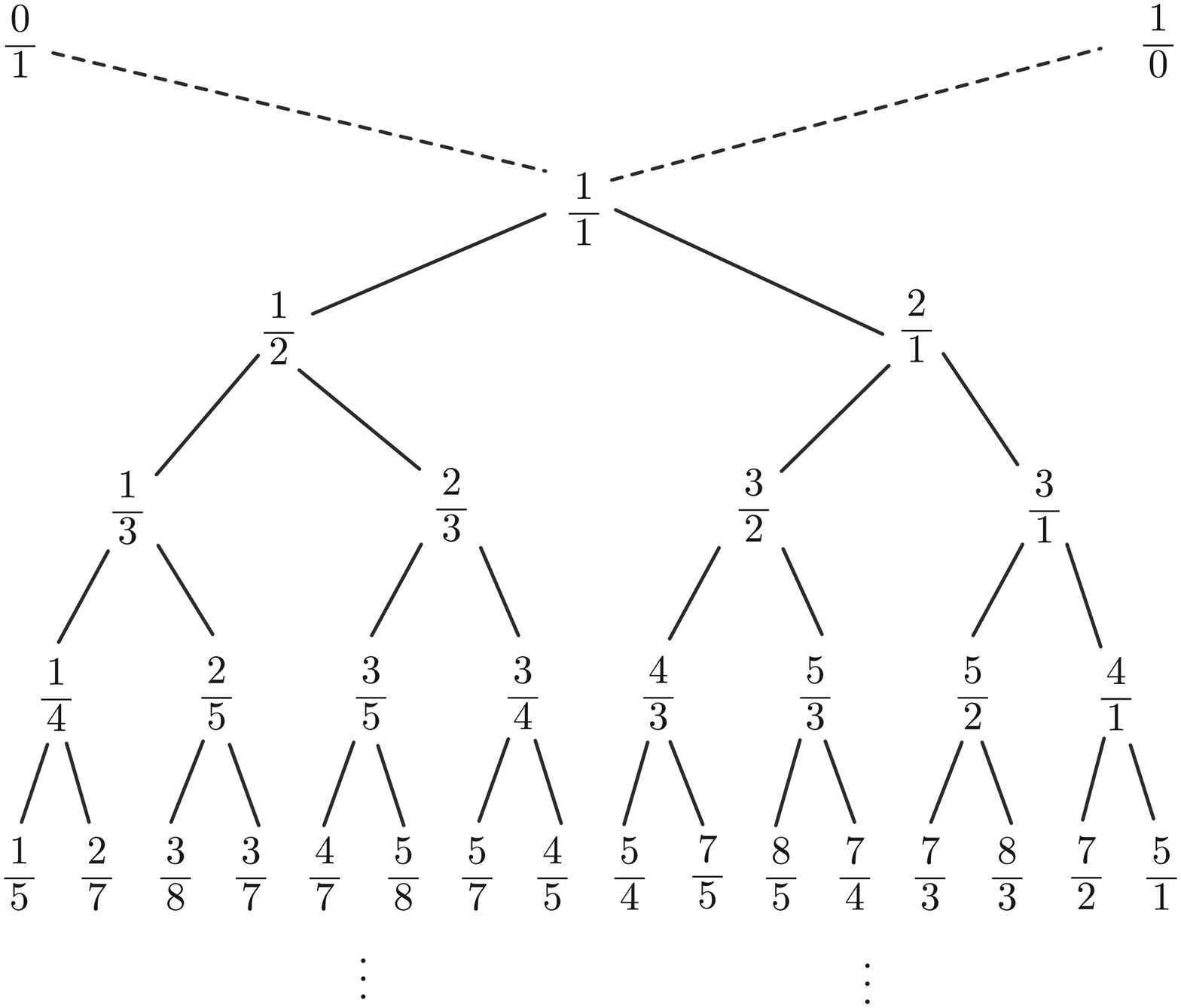}
\caption{}\label{figw5}
\end{figure}

\par 
The Kauffman bracket polynomial of the rational tangle diagram $T(\alpha )$ corresponding to a rational number $\alpha \in \mathbb{Q}\cap (0,1)$ can be computed by using the extended graph of the Stern-Brocot tree by adding $\frac{0}{1}$ up to powers $\pm A^{\pm 1}$ as follows. 
\begin{enumerate}
\item[(i)] Start from $\alpha $, and combine by edges upward successively rational numbers, which have relations as parents and mediants 
(In the Stern-Brocot tree we only have to add edges between disconnected pairs of a parent and its mediant). 
\item[(ii)] Repeat the above procedure until the path arrives at $\frac{0}{1}$ or $\frac{1}{1}$. 
\item[(iii)] Set $-t$ or $-t^{-1}$ on the obtained edges based on the following rule: $-t$ is set for an edge on the left oblique line of a triangle, and $-t^{-1}$ is set on the right. 
\item[(iv)] Take the product of all  $-t^{\pm 1}$ on the path from $\alpha $ to $\frac{0}{1}$ or $\frac{1}{1}$, and take the sum of them running over such all paths. 
\end{enumerate}

We see that the obtained Laurent polynomial of $t$ coincides with the Kauffman bracket polynomial $\langle T(\alpha )\rangle $ up to powers $\pm A^{\pm 1}$, provided that $t=A^4$ is substituted. This result will be shown in Subsection \ref{subsection2-3}. 

\begin{figure}[htbp]
\includegraphics[width=15.5cm]{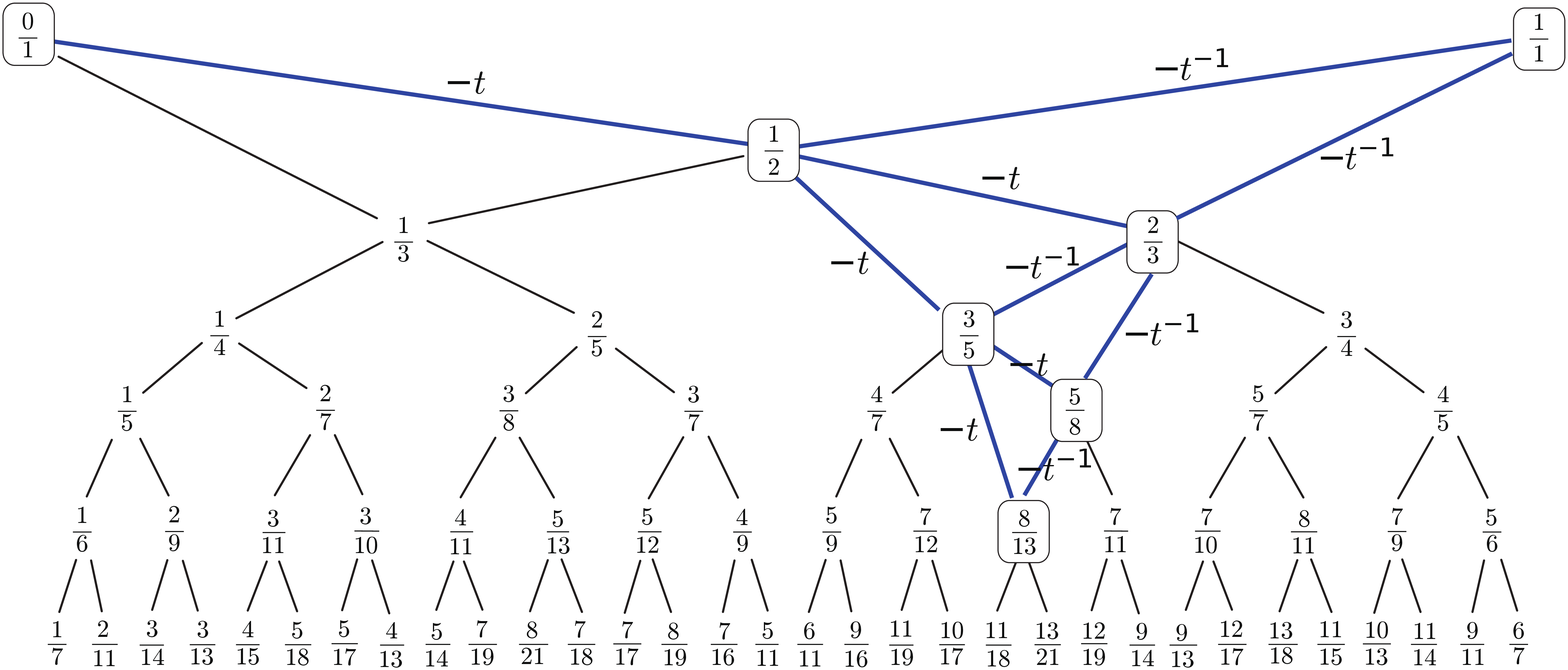}
\caption{}\label{figw6}
\end{figure}

\par \smallskip 
\subsection{Yamada's reconstruction of the Kauffman brackets for rational tangles} 
\par 
Yamada \cite{Yamada-Proceeding} introduced the concept of an ancestor triangle of a rational number to compute the Kauffman bracket polynomial for rational knots and links from Farey neighbors.  
The ancestor triangle naturally appears in steps of the process which is mentioned in the previous subsection. 

\par \medskip 
\begin{defn}[{\bf S.Yamada \cite{Yamada-Proceeding}}]
Let  $\alpha $ be a nonzero rational number. 
If $\alpha >0$, then we have a triangle whose vertices are $\alpha , \frac{0}{1}(=0), \frac{1}{0}(=\infty )$ by tracing the roots of parents from $\alpha$.  
The triangle with decomposition into small triangles whose vertices are Farey neighbors and their Farey sums is said to be the ancestor triangle of $\alpha$. 
These small triangles are called fundamental triangles. 
When we draw the Yamada's ancestor triangle of $\alpha$ in the plane, the vertices corresponding to $\frac{0}{1}, \frac{1}{0}$ put on the negative and positive parts in the $x$-axis, respectively, and the vertex corresponding to $\alpha $ put on the negative part in the $y$-axis.  
\end{defn}

\par 
Let us explain Yamada's formulation of the Kauffman bracket polynomial for rational tangles based on ancestor triangles. 
Let $\mathbb{Q}_+$ be the set of all positive rational numbers. 
For each $\alpha \in \mathbb{Q}_+-\{ 0\} $, we 
denote by $r(\alpha )$ and $l(\alpha )$ the numbers  
of edges on the right and left oblique sides of the Yamada's ancestor triangle  
of $\alpha $, respectively. 
We also set for $\ast =0, \infty$ 
$$P_{\ast}(\alpha ):=\{ \text{the descending paths from $\ast$ to $\alpha$}\} ,$$  
and for $\gamma \in P_{\infty }(\alpha )\cup P_{0}(\alpha )$ denote by $w_R(\gamma )$ and $w_L(\gamma )$ 
the numbers of fundamental triangles in the right-hand and left-hand sides of $\gamma$, respectively. 
Let us define a map $\phi : \mathbb{Q}_+\cup \{ \infty \} \longrightarrow \Lambda^2$ by 
$\phi (\infty )=A^6[\infty ],\ \phi (0)=[0]$ and
\begin{equation}\label{eqw1-13}
\begin{aligned}
\phi (\alpha )&=\Bigl( A^6(-A^4)^{-r(\alpha )}\sum\limits_{\gamma \in P_{\infty}(\alpha )} (-A^4)^{w_R(\gamma )}\Bigr) [\infty ] \\ 
&\qquad 
+\Bigl( (-A^4)^{l(\alpha )}\sum\limits_{\gamma \in P_{0}(\alpha )} (-A^4)^{-w_L(\gamma )}\Bigr) [0] 
\end{aligned}
\end{equation}
for $\alpha \in \mathbb{Q}_+-\{ 0\} $. 

\par \smallskip 
\begin{exam}
If $\alpha =\frac{7}{4}$, then 
$\alpha =\frac{5}{3}\sharp \frac{2}{1},  
\frac{5}{3}=\frac{3}{2}\sharp \frac{2}{1},   
\frac{3}{2}=\frac{1}{1}\sharp \frac{2}{1}$,  
$\frac{2}{1}=\frac{1}{1}\sharp \frac{1}{0},  
\frac{1}{1}=\frac{0}{1}\sharp \frac{1}{0}$. 
Thus, the  Yamada's ancestor triangle of $\frac{7}{4}$ is given 
by the left-hand side in Figure~\ref{figw7}. 

\begin{figure}[htbp]
\centering 
\includegraphics[height=3cm]{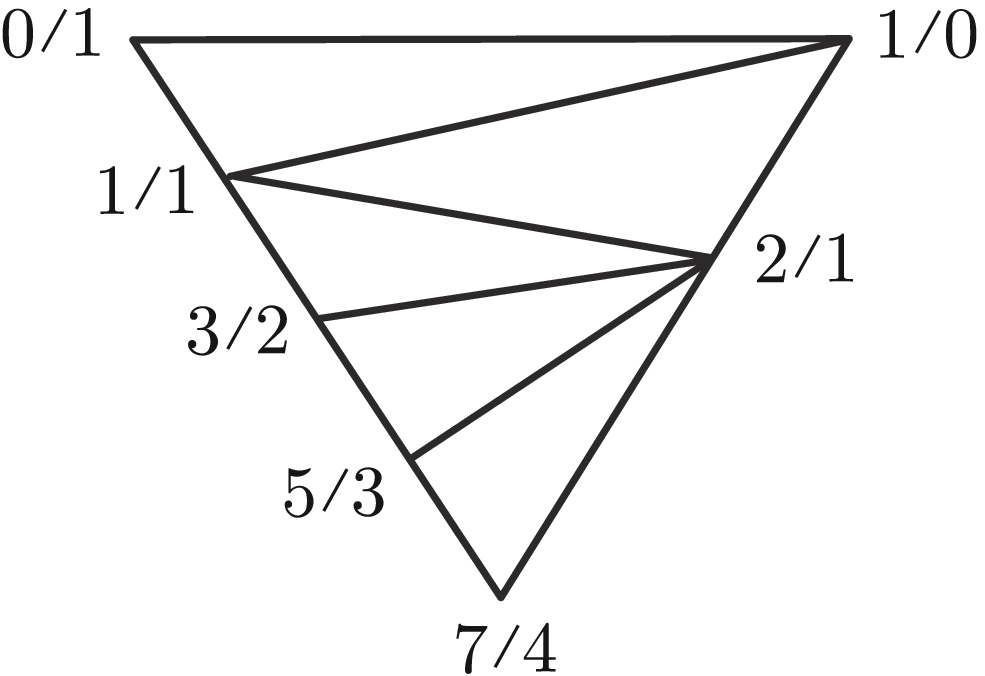}\hspace{2cm}  
\includegraphics[height=3cm]{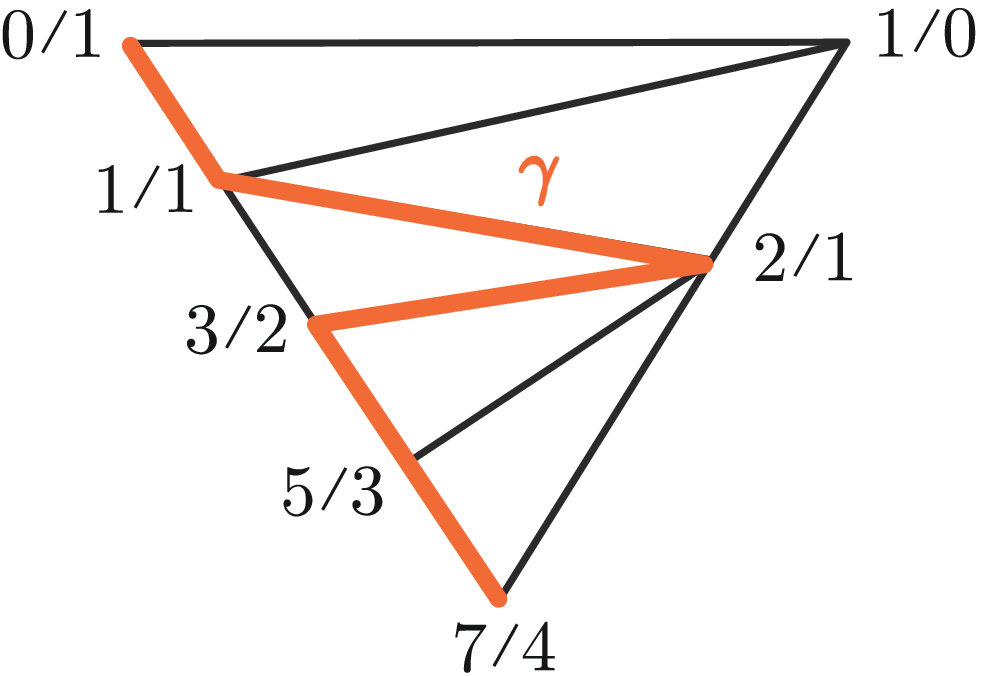}
\caption{}\label{figw7}
\end{figure}

In this case 
$r(\alpha )=2,\ l(\alpha )=4$, 
and for the path  $\gamma $ depicted in the right-hand side in Figure~\ref{figw7}, 
$w_R(\gamma )=4 ,\ w_L(\gamma )=1$. 
\end{exam}

\par 
\begin{thm}[\kern-0.1em {\bf S.Yamada \cite[Theorem 2]{Yamada-Proceeding}}]\label{w1-8}
The map 
$\phi : \mathbb{Q}_+\cup \{ \infty \} \longrightarrow \Lambda ^2$ satisfies 
\begin{equation}\label{eqw1-14}
\phi (0)=[0],\quad 
\phi (\infty )=A^6[\infty ], 
\end{equation}
and the recursive equation 
\begin{equation}\label{eqw1-15}
\phi (x)=-A^4\phi (y)-A^{-4}\phi (z),
\end{equation}
where $x, y, z\in \mathbb{Q}_+\cup \{ \infty \}$, and $(y,z)$ is the pair of parents of $x$. 
\end{thm} 
\begin{proof}
\par 
Write as $y=\frac{p}{q}$ and $z=\frac{r}{s}$ as irreducible fractions, and set $x=\frac{p+r}{q+s}$.  
Let us show \eqref{eqw1-15}. 
For each $\alpha \in \mathbb{Q}_+\cup \{ \infty \}$, let $\phi _{\infty }(\alpha )$ and $\phi _0(\alpha )$ be the coefficients of $[\infty ]$ and $[0]$ in \eqref{eqw1-13}, respectively. 
Note that $rq-ps=1$ since $y$ and $z$ are Farey neighbors. 
\par 
First, let us consider the case of $q-s>0$. 
Then $p-r\geq 0$, and 
$q^{\prime}:=q-s,\ p^{\prime}:=p-r$ satisfy 
$rq^{\prime}-p^{\prime}s=rq-ps=1$. 
Thus, the fraction $\frac{p^{\prime}}{q^{\prime}}$ is 
irreducible, and 
$\frac{p^{\prime}}{q^{\prime}}<\frac{r}{s},\ \frac{p}{q}=\frac{p^{\prime}}{q^{\prime}}\sharp \frac{r}{s}$. 
Therefore, the  Yamada's ancestor triangle of $x=\frac{p+r}{q+s}$ has a shape as the left-hand side in Figure~\ref{figw8}. 

\begin{figure}[htbp]
\centering 
\includegraphics[height=4cm]{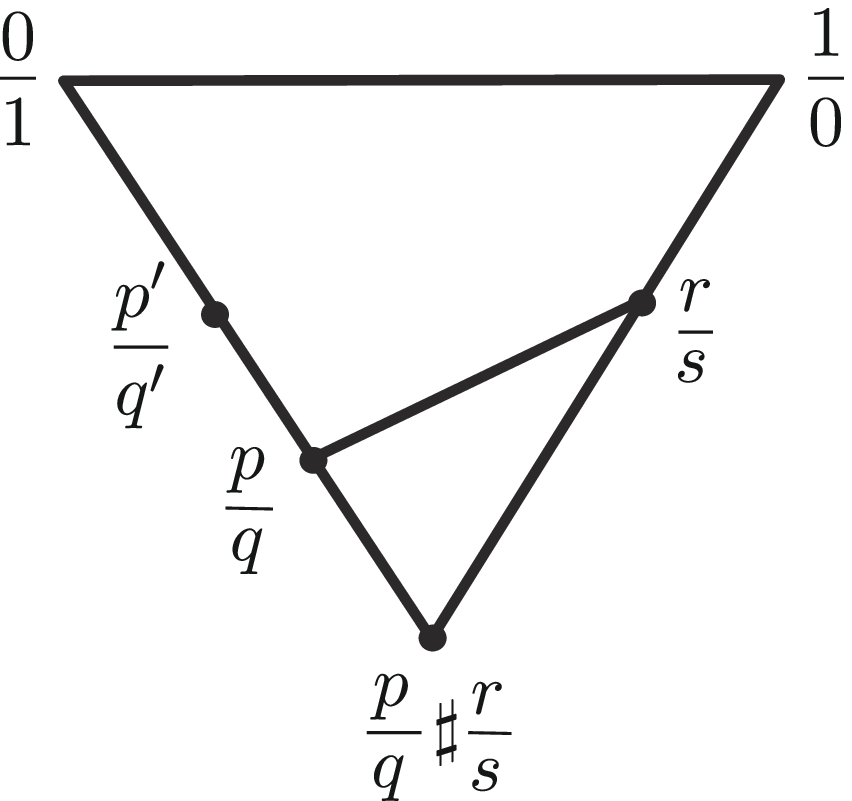} 
\hspace{2cm}
\includegraphics[height=4cm]{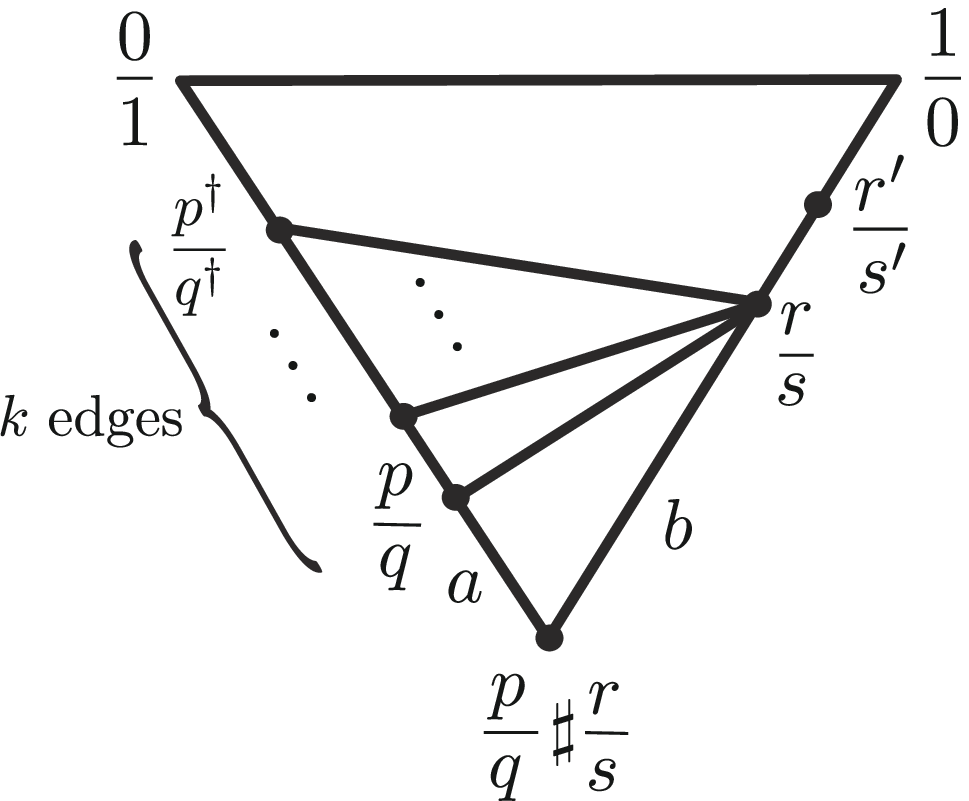}
\caption{}\label{figw8}
\end{figure}

Now, let $\frac{p^{\dag}}{q^{\dag}}, \frac{r^{\prime}}{s^{\prime}}$ be Farey neighbors and $\frac{r}{s}=\frac{p^{\dag}}{q^{\dag}}\sharp \frac{r^{\prime}}{s^{\prime}}$. 
Let $k\ (\geq 1)$ be the number of edges on the left oblique side between 
$\frac{p^{\dag}}{q^{\dag}}$ and $\frac{p}{q}$ in the  Yamada's ancestor triangle of $x=\frac{p+r}{q+s}$, and 
$a, b$ be the edges connecting $\frac{p}{q}, \frac{r}{s}$ to $x$, respectively.  
Then, 
\begin{align*}
r\Bigl(\frac{p}{q}\sharp \frac{r}{s}\Bigr) &=r\Bigl(\frac{p}{q}\Bigr) =r\Bigl(\frac{r}{s}\Bigr) +1,\\[0.2cm]  
l\Bigl(\frac{p}{q}\sharp \frac{r}{s}\Bigr) &=l\Bigl(\frac{p}{q}\Bigr) +1=l\Bigl(\frac{r}{s}\Bigr) +k , \displaybreak[0]\\[0.2cm]  
P_{\infty }\Bigl(\frac{p}{q}\sharp \frac{r}{s}\Bigr) &=\Bigl\{ \ \gamma \ast a \ \Bigr| \ \gamma \in P_{\infty}\Bigl(\frac{p}{q}\Bigr) \ \Bigr\} 
\cup \Bigl\{ \ \gamma \ast b \ \Bigr| \ \gamma \in P_{\infty}\Bigl(\frac{r}{s}\Bigr) \ \Bigr\} ,\\[0.2cm]  
P_0\Bigl(\frac{p}{q}\sharp \frac{r}{s}\Bigr) &= \Bigl\{ \ \gamma \ast a \ \Bigr| \ \gamma \in P_0\Bigl(\frac{p}{q}\Bigr) \ \Bigr\} 
\cup \Bigl\{ \ \gamma \ast b \ \Bigr| \ \gamma \in P_0\Bigl(\frac{r}{s}\Bigr) \ \Bigr\} , 
\end{align*}
and 
\begin{align*}
w_R(\gamma \ast a)&=w_R(\gamma )+1 \quad \text{for $\gamma \in P_{\infty}\Bigl(\frac{p}{q}\Bigr)$},\\[0.1cm]
w_R(\gamma \ast b)&=w_R(\gamma )\quad \text{for $\gamma \in P_{\infty}\Bigl(\frac{r}{s}\Bigr)$}, \\[0.1cm]
w_L(\gamma \ast a)&=w_L(\gamma )\quad \text{for $\gamma \in P_0\Bigl(\frac{p}{q}\Bigr)$},\\[0.1cm]  
w_L(\gamma \ast b)&=w_L(\gamma )+k+1\quad \text{for $\gamma \in P_0\Bigl(\frac{r}{s}\Bigr)$ }  
\quad (\text{see the right-hand side in Figure~\ref{figw8}}). 
\end{align*}
Hence, 
$\phi _{\ast }(\frac{p}{q}\sharp \frac{r}{s}) 
=-A^4\phi _{\ast }(\frac{p}{q}) -A^{-4}\phi _{\ast }(\frac{r}{s})$  for $\ast =\infty , 0$. 
\par 
In the case of $q-s<0$ we see that \eqref{eqw1-15} may also be shown by the same argument above. 
\end{proof}

\par \medskip 
Note that by \eqref{eqw1-4} and \eqref{eqw1-15} 
\begin{equation}\label{eqw1-16}
v(\phi (x))=-A^4v(\phi (y))-A^{-4}v(\phi (z))
\end{equation}
for $x, y, z\in \mathbb{Q}_+\cup \{ \infty \}$ so that $(y,z)$ is the pair of parents of $x$, and 
\begin{equation}
\begin{aligned}
v(\phi (\alpha ))&=(1+A^4)(-A^4)^{-r(\alpha )+1}\sum\limits_{\gamma \in P_{\infty}(\alpha )} (-A^4)^{w_R(\gamma )} \\ 
&\qquad +(-A^4)^{l(\alpha )}\sum\limits_{\gamma \in P_{0}(\alpha )} (-A^4)^{-w_L(\gamma )}
\end{aligned}
\end{equation}
for $\alpha \in \mathbb{Q}_+-\{ 0\} $. So, it is a Laurent polynomial in variable $-A^4$.  
From the above formulation we have: 

\par \medskip 
\begin{cor}\label{w1-9}
Define a function $\tilde{\phi }: \mathbb{Q}_+\cup \{ \infty \} \longrightarrow \Lambda^2$ by 
$\tilde{\phi }(\infty )=[\infty ],\ \tilde{\phi }(0)=[0]$ and 
\begin{equation}\label{eqw1-18}
\tilde{\phi }(\alpha )=\Bigl( (-A^4)^{-r(\alpha )}\sum\limits_{\gamma \in P_{\infty}(\alpha )} (-A^4)^{w_R(\gamma )}\Bigr) [\infty ]
+\Bigl( (-A^4)^{l(\alpha )}\sum\limits_{\gamma \in P_{0}(\alpha )} (-A^4)^{-w_L(\gamma )}\Bigr) [0]. 
\end{equation}
Then 
\begin{equation}
\tilde{\phi }(x)=-A^4\tilde{\phi }(y)-A^{-4}\tilde{\phi }(z), 
\end{equation}
where $x, y, z\in \mathbb{Q}_+\cup \{ \infty \}$ and $(y,z)$ is the pair of parents of $x$, 
and for a positive integer $\alpha \in \mathbb{Q}_+$ 
\begin{equation}\label{eqw1-19}
v(\phi (\alpha ))=\text{tr}(\tilde{\phi }(\alpha )), 
\end{equation}
where $\text{tr}$ is a map from $\Lambda ^2$ to $\Lambda$ defined by
$$\text{tr}(a[\infty ]+b[0])=a(-A^4)(A^4+1)+b\qquad (a, b\in \Lambda ).$$
\par \vspace{-0.85cm}
\ \qed 
\end{cor} 

Yamada \cite[Theorem 1]{Yamada-Proceeding} found that the value $v(\phi (\alpha ))$ coincides with the Kauffman bracket of $D(T(\alpha ))$ up to powers of $A^{\pm 1}$ (see also Theorem~\ref{w1-10} below). 
This means that the computation of Jones polynomial of rational knots and links is essentially reduced to that of the function $\phi$. 
\par 
The following theorem is an explicit version of \cite[Theorem 1]{Yamada-Proceeding}.  

\bigskip \noindent 
\begin{thm}\label{w1-10}
Let $a_1, \ldots , a_n$ be positive integers, and $a_0$ be a non-negative integer. 
\par 
If $n$ is even, then 
$$\phi ([a_0; a_1, \ldots , a_{n-1}, a_n])=(-A^3)^{a_0-a_1+\cdots -a_{n-1}+a_n}\langle T([a_0; a_1, \ldots , a_{n-1} ,a_n])\rangle , $$
and if $n$ is odd, then
$$\phi ([a_0; a_1, \ldots ,a_n])=(-A^3)^{a_0-a_1+\cdots -a_n+2}\langle T([a_0; a_1, \ldots ,a_n])\rangle .$$
\end{thm} 

\par \medskip
Before verifying the above theorem, let us observe its statement is true for the case where $n$ is small. 

\par \smallskip 
\begin{exam}\label{w1-11}
$(1)$ Let $n$ be a positive integer. 
Since 
$\frac{1}{n}=\frac{0}{1}\sharp \frac{1}{n-1},\ \frac{1}{n-1}=\frac{0}{1}\sharp \frac{1}{n-2},\cdots , \frac{1}{1}=\frac{0}{1}\sharp \frac{1}{0}$, 
the  Yamada's ancestor triangle of $\alpha =\frac{1}{n}$ is as in Figure~\ref{figw9}. 

\begin{figure}[htbp]
\centering
\includegraphics[height=4cm]{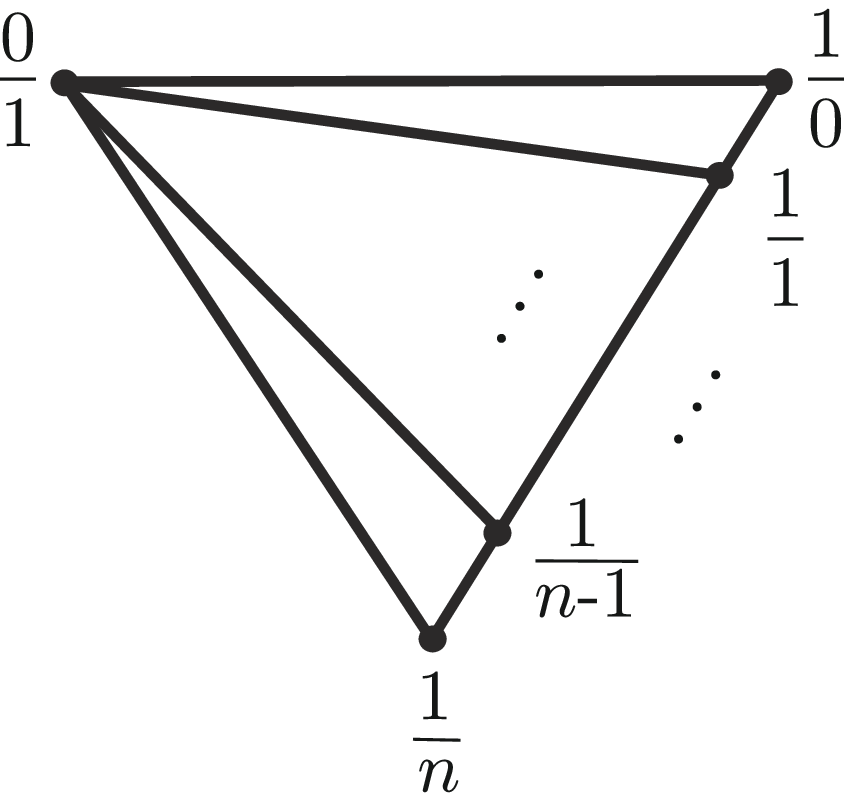}
\caption{}\label{figw9}
\end{figure}

Thus, 
$r(\frac{1}{n}) =n,\  
l(\frac{1}{n}) =1,\  
P_{\infty }(\frac{1}{n}) =\{ \gamma \} ,\  
P_0(\frac{1}{n}) = \{ \gamma _1,\dots , \gamma_n\} $, 
where 
$\gamma : \frac{1}{0}\to \frac{1}{1}\to \cdots \to \frac{1}{n}$,  
$\gamma _k : \frac{0}{1}\to \frac{1}{k}\to \frac{1}{k+1}\to \cdots \to \frac{1}{n}$ for $k=1,2,\ldots , n-1$, 
and $\gamma _n : \frac{0}{1}\to \frac{1}{n}$ are descending paths. 
Since $w_R(\gamma )=0, w_L(\gamma _k)=n-k$, 
we have 
$\phi ( \frac{1}{n})
=(-A^2)(-A^4)^{-n+1}\{ [\infty ]+A^2\sum_{k=0}^{n-1}(-A^4)^{k}[0]\}$. 
Thus, by Example \ref{w1-4} 
$$\Bigl\langle T\Bigl( \frac{1}{n}\Bigr)\Bigr\rangle \ 
=(-A^3)^{n-2}\phi \Bigl( \frac{1}{n}\Bigr). $$
\indent 
$(2)$ By a similar argument, it can be shown that 
$\langle T( \frac{n}{1})\rangle \ =(-A^{-3})^n\phi ( \frac{n}{1})$. 
\end{exam} 

The result in Example~\ref{w1-11} (1) is extended as follows: 

\par \bigskip \noindent 
\begin{exam}\label{w1-12}
$(1)$ Let $m, n$ be positive integers. 
Since $m+\frac{1}{n}=\frac{m}{1}\sharp \frac{mn-m+1}{n-1}$, 
we see that 
$\phi ( m+\frac{1}{n})
=-A^4\phi (m)-A^{-4}\phi (m+\frac{1}{n-1}) $. 
By repeating this process and using Example \ref{w1-4}(2) we have 
\begin{align*}
\phi \Bigl( m+\frac{1}{n}\Bigr)
&=\Bigl(-A^4+\sum\limits_{k=0}^{n-2}(-A^{-4})^k\Bigr)\phi (m)+(-A^{-4})^n\phi (\infty ) \\ 
&=(-A^3)^{m-n+2}\biggl(  d\Bigl( \frac{1}{n}\Bigr) \langle [m]\rangle +A^{-n}(-A^3)^{-m}[\infty ]\biggr)\displaybreak[0]\\ 
&=(-A^3)^{m-n+2}\Bigl\langle T\Bigl( m+\frac{1}{n}\Bigr)\Bigr\rangle .
\end{align*}
\indent 
$(2)$ Let $l,m, n$ be positive integers,\ $n\geq 2$, and $\alpha =l+\frac{1}{m+\frac{1}{n}}$. 
Then $\frac{lmn+l+n-lm-1}{mn-m+1}, \frac{lm+1}{m}$ are Farey neighbors, and 
$\alpha 
=(l+\frac{1}{m+\frac{1}{n-1}}) \sharp ( l+\frac{1}{m})$. 
Therefore, $[l; m,n]=[l ; m,n-1]\sharp [l ; m]$, 
and $\phi ([l ; m,n])=-A^4\phi ([l ; m,n-1])-A^{-4}\phi ([l ; m])$. 
It follows from Part (1) that $\phi ([l ; m])=(-A^3)^{l-m+2}\langle T([l ; m])\rangle $. 
\par 
Similarly, one can show that  
\begin{equation}\label{eqw1-21}
\phi ([l ; m,n])=(-A^3)^{l-m+n}\langle T([l ; m,n])\rangle 
\end{equation}
for positive integers $l,m, n$ by using 
$\phi ([l ; m,1])=\phi ([l ; m+1])
=(-A^3)^{l-m+1}\langle T([l ; m, 1])\rangle $ and 
$\langle T([l ; m,n])\rangle =A\langle T([l ; m,n-1])\rangle +A^{-1}(-A^{-3})^{n-1}\langle T([l ; m])\rangle $ for $n\geq 2$. 
\qed 
\end{exam}

\par \medskip \noindent 
{\bf Proof of Theorem~\ref{w1-10}.} 
\par 
For $n=0, 1, 2$ the theorem has already verified in Examples \ref{w1-11} and \ref{w1-12}. 
Let $n\geq 3$, and assume that the statement is true for $n-1$. 
\par 
Let $n\geq 3$ be odd. 
 If $a_n=1$, then 
\begin{align*}
\phi ([a_0 ; a_1,  \ldots ,a_n])
&=\phi ([a_0; a_1, \ldots ,a_{n-1},1]) 
=\phi ([a_0; a_1, \ldots ,a_{n-1}+1]) \\ 
&=(-A^3)^{a_0-a_1+\cdots -a_{n-2}+a_{n-1}+1}\langle T([a_0; a_1,  \ldots , a_{n-1}+1])\rangle \\ 
&=(-A^3)^{a_0-a_1+\cdots -a_{n-2}+a_{n-1}-a_n+2}\langle T([a_0; a_1,  \ldots , a_{n-1}, 1])\rangle . 
\end{align*}
\par 
Let $k$ be a positive integer, and assume that the statement is true for the case of $a_n=k$. 
If $a_n=k+1$, then $[a_0; a_1, \ldots , a_{n-1}],\ [a_0; a_1, \ldots , a_n-1]$ are Farey neighbors, 
and $[a_0; a_1, \ldots , a_n]=[a_0; a_1, \ldots , a_{n-1}]\sharp [a_0; a_1, \ldots , a_n-1]$. 
Thus, by Theorem~\ref{w1-8} we have 
\begin{align*}
\phi ([a_0; a_1, \ldots ,a_n])
&=\phi ([a_0; a_1, \ldots ,a_{n-1},k+1])\\ 
&=-A^{-4}\phi ([a_0; a_1, \ldots ,a_{n-1}, k])-A^4\phi ([a_0; a_1, \ldots ,a_{n-1}])  \displaybreak[0]\\ 
&=-A^{-4}(-A^3)^{a_0-a_1+\cdots -a_{n-2}+a_{n-1}-k+2}\langle T([a_0; a_1, \ldots , a_{n-1},k])\rangle \\ 
&\quad -A^4(-A^3)^{a_0-a_1+\cdots -a_{n-2}+a_{n-1}}\langle T([a_0; a_1, \ldots , a_{n-1}])\rangle \displaybreak[0]\\ 
&=(-A^3)^{a_0-a_1+\cdots -a_{n-2}+a_{n-1}-k+1}\langle T([a_0; a_1, \ldots , a_{n-1}, k+1])\rangle . 
\end{align*}
\par 
By a similar argument we see that the statement is true for the case where $n\geq 4$ is even. 
\qed 

\medskip 
By combining Theorems~\ref{w1-8} and \ref{w1-10}, we obtain a formula for the Kauffman bracket polynomials of rational tangle diagrams to compute it based on Farey neighbors. 

\noindent 
\begin{exam}
For positive integers $p, n$
$$\phi \Bigl( \frac{p}{pn+1}\Bigr) =(-A^3)^{p-n}\Bigl\langle T\Bigl( \frac{p}{pn+1}\Bigr) \Bigr\rangle $$
by Theorem~\ref{w1-10} since $\frac{p}{pn+1}=[0; n, p]$.  
\end{exam}

\section{The Kauffman bracket polynomials of Conway-Coxeter friezes of zigzag-type}
\par 
Some Conway-Coxeter friezes that we call of zigzag-type are closely related to words consisting of $L$ and $R$. 
An $LR$ word corresponds to a  rational number in the open interval $(0,1)$. 
\par 
In this section first we introduce two operations $i(w)$ and $r(w)$ for $LR$ words $w$, and examine these properties related to rational numbers. 
Next, we explain the definition of Conway-Coxeter friezes of zigzag-type, and show that the Kauffman bracket polynomial can be also defined for them by composing the functions $v$ and $\phi $ defined in \eqref{eqw1-4} and \eqref{eqw1-13}. 
In the final subsection we derive various properties of the Kauffman bracket polynomial for Conway-Coxeter friezes. 

\subsection{Irreducible fractions and $LR$ words}
\par 
Any downward path in the Stern-Brocot tree can be represented by an $LR$ word, which is obtained by recording $L$ or $R$ depending on a leftward or rightward walking. 
If we fix a starting number, then the above $LR$ word is uniquely determined.  
In this paper, for a rational number $\alpha $ in the open interval $(0,1)$ we choose $\frac{1}{2}$ as the starting number of downward paths.  
The reason why we adopt this definition comes from a relationship with Conway-Coxeter friezes. 
For the same reason we consider the inverse ordering of an $LR$ word, and 
denote the resulting word by $w(\alpha )$, and call it the $LR$ word associated to $\alpha$. 

\par \smallskip 
\begin{exam}\label{w2-1}
$w\Bigl(\dfrac{1}{2}\Bigr)=\emptyset$, 
$w\Bigl(\dfrac{1}{5}\Bigr)=LLL=L^3$, 
$w\Bigl(\dfrac{2}{7}\Bigr)=RLL=RL^2$, 
$w\Bigl(\dfrac{3}{8}\Bigr)=LRL$, 
$w\Bigl(\dfrac{3}{7}\Bigr)=RRL=R^2L$. 
\end{exam}

\par \medskip 
$LR$ words of Farey sums can be known by those of Farey neighbors. 

\par \medskip 
\begin{lem}\label{w2-2}
Let $\frac{p}{q}, \frac{r}{s}\in \mathbb{Q}\cap (0,1)$ be Farey neighbors, and $w, w^{\prime}$ are the corresponding $LR$ words of them, respectively. 
Then
\begin{equation}\label{eqw2-1}
w\Bigl( \dfrac{p}{q}\sharp \dfrac{r}{s}\Bigr) =
\begin{cases} 
Lw^{\prime} & \text{if $\ell (w)<\ell (w^{\prime})$},\\ 
Rw & \text{if $\ell (w)>\ell (w^{\prime})$,}
\end{cases}
\end{equation}
where $\ell (-)$ denotes the length function on the words. 
We denote the right-hand side of \eqref{eqw2-1} by $w\vee w^{\prime}$. 
\qed 
\end{lem} 

\begin{rem}\label{w2-3}
For positive integers $a_2, \ldots , a_m$, 
$\ell \bigl( w([0; a_2, \ldots , a_m])\bigr)>\ell \bigl( w([0; a_2, \ldots , a_{m-1}])\bigr)$. 
This inequality can be shown by double induction argument on $m\ (\geq 2)$ and $a_m$.
\end{rem}

\par \medskip 
\begin{exam}
\begin{enumerate}
\item[$(1)$] $w\bigl( \frac{5}{9}\bigr) =w\bigl( \frac{1}{2}\sharp \frac{4}{7}\bigr) =L^3R$ since  
$w\bigl( \frac{1}{2}\bigr) =\emptyset $ and $w\bigl( \frac{4}{7}\bigr) =L^2R$. 
\item[$(2)$] $w\bigl( \frac{5}{12}\bigr) =w\bigl( \frac{2}{5}\sharp \frac{3}{7}\bigr) =LR^2L$
since $w\bigl( \frac{2}{5}\bigr) =RL$ and $w\bigl( \frac{3}{7}\bigr) =R^2L$. 
\end{enumerate}
\end{exam} 

\par \medskip 
For an $LR$ word $w$ we denote by $i(w)$ the word obtained by changing $L$ with $R$, and 
by $r(w)$ the word obtained by reversing the order. 

\par \medskip 
\begin{lem}\label{w2-5}
\begin{enumerate}
\item[$(1)$] Let  $\frac{p}{q}\in \mathbb{Q}\cap (0,1)$ be an irreducible fraction, and set 
$w=w\bigl( \frac{p}{q}\bigr)$. Then 
$$i(w)= w\Bigl( \dfrac{q-p}{q}\Bigr).$$
\item[$(2)$] Let  $w, w^{\prime}$ be $LR$ words such that their corresponding irreducible fractions are Farey neighbors. 
Then 
$$i(w\vee w^{\prime})=i(w^{\prime})\vee i(w).$$
\end{enumerate}
\end{lem}
\begin{proof}
\par 
Parts (1) and (2) are proved by induction on the lengths of $LR$ words at the same time. 
\par 
First we show that the lemma is true for the $LR$ words with length $0$ or $1$. 
\par 
If an $LR$ word has length $0$, then it is $w_1:=w\bigl(\frac{1}{2}\bigr) =\emptyset$, and 
if it has length $1$, then it is $w_2:=w\bigl(\frac{1}{3}\bigr) =L$ or $w_3:=w\bigl(\frac{2}{3}\bigr) =R$. 
Since 
\begin{align*}
i(w_1)&=\emptyset =w\Bigl(\dfrac{1}{2}\Bigr) =w\Bigl(\dfrac{2-1}{2}\Bigr) ,\\ 
i(w_2)&=R=w\Bigl(\dfrac{2}{3}\Bigr) =w\Bigl(\dfrac{3-1}{3}\Bigr) ,\\ 
i(w_3)&=L=w\Bigl(\dfrac{1}{3}\Bigr) =w\Bigl(\dfrac{3-2}{3}\Bigr) , 
\end{align*}
we see that Part (1) holds for the $LR$ words with length $0$ or $1$. 
Furthermore, $i(w_2\vee w_1)=i(L)=R=i(w_1)\vee i(w_2)$ and $i(w_1\vee w_3)=i(R)=L=i(w_3)\vee i(w_1)$. 
Thus, Part (2) holds for the all pairs  $(w, w^{\prime})$ of $LR$ words with length less than or equal to $1$. 
\par 
Assume that (1) and (2) hold for the all $LR$ words with length  less than $d$.  
Let $w$ be an $LR$ word with length $d$, and represented by $w=w\bigl( \frac{x}{y}\bigr)$. 
Let $\bigl( \frac{p}{q}, \frac{r}{s}\bigr)$ be the pair of parents of $\frac{x}{y}$. 
If we write as $w_1=w\bigl( \frac{p}{q}\bigr)$ and $w_2=w\bigl( \frac{r}{s}\bigr)$, then by induction hypothesis $i(w_1)=w\bigl( \frac{q-p}{q}\bigr) ,\ i(w_2)=w\bigl( \frac{s-r}{s}\bigr)$. 
Since $s(q-p)-q(s-r)=-sp+rq=1$, 
the rational numbers $\frac{s-r}{s}, \frac{q-p}{q}$ are Farey neighbors. 
By the definitions of $\vee $ and $i$ we have 
$$i(w)=i(w_1\vee w_2)=\begin{cases}
Ri(w_2)\  & \text{if $\ell (w_1)<\ell (w_2)$},\\ 
Li(w_1)\  & \text{if $\ell (w_1)>\ell (w_2)$},
\end{cases}$$
and 
$$i(w_2)\vee i(w_1)=\begin{cases}
Li(w_1)\  & \text{if $\ell \bigl( i(w_2)\bigr) <\ell \bigl( i(w_1)\bigr)$},\\ 
Ri(w_2)\  & \text{if $\ell \bigl( i(w_2)\bigr) >\ell \bigl( i(w_1)\bigr)$}. 
\end{cases}$$
Since 
$\ell \bigl( i(w_1)\bigr) =\ell (w_1),\ \ell \bigl( i(w_2)\bigr) =\ell (w_2)$, we also have 
$i(w_2)\vee i(w_1)=i(w_1\vee w_2)$. 
Furthermore, 
$$i(w)=i(w_1\vee w_2)=i(w_2)\vee i(w_1)=w\Bigl( \dfrac{s-r}{s}\sharp \dfrac{q-p}{q}\Bigr) 
=w\Bigl( \dfrac{s+q-r-p}{s+q}\Bigr) =w\Bigl(\dfrac{y-x}{y}\Bigr).$$
Thus Part (1) holds for the $LR$ words with length $d$. 
By the same argument we see that 
Part (2) also holds for any $LR$ words $w_1, w_2$ with length less than or equal to $d$ such that the corresponding fractions are  Farey neighbors. 
\end{proof} 

\par \medskip 
\begin{lem}\label{w2-6}
Let  $\frac{p}{q}\in \mathbb{Q}\cap (0,1)$ be an irreducible fraction, and set 
$w=w\bigl( \frac{p}{q}\bigr)$. 
Then 
$$wL= w\Bigl( \frac{p}{p+q}\Bigr),\qquad wR= w\Bigl( \frac{p+(q-p)}{q+(q-p)}\Bigr).$$
\end{lem}
\begin{proof}
\par 
This lemma is proved by induction on the length of $LR$ words as the proof of Lemma~\ref{w2-5}. 
\par 
The irreducible fractions which correspond to $\emptyset , L, R$ are $\frac{1}{2}, \frac{1}{3}, \frac{2}{3}$, respectively. 
It can be easily verified that the formulae in the lemma are satisfied for these $LR$ words. 
\par 
Assume that the lemma holds for the all $LR$ words with length less than $d$, and let $w$ be an $LR$ word with length $d$. 
We write as in the form
$w=w\bigl( \frac{p^{\prime}}{q^{\prime}}\sharp \frac{r^{\prime}}{s^{\prime}}\bigr) $, 
and set $w_1:=w\bigl( \frac{p^{\prime}}{q^{\prime}}\bigr) , 
w_2:=w\bigl( \frac{r^{\prime}}{s^{\prime}}\bigr) $. 
\par 
Let us consider the case $\ell (w_2)>\ell (w_1)$. 
By Lemma~\ref{w2-2} we have $w=Lw_2$, and by induction hypothesis 
$$w_2L=w\Bigl( \dfrac{r^{\prime}}{r^{\prime}+s^{\prime}}\Bigr) ,\qquad 
w_2R=w\Bigl( \dfrac{r^{\prime}+(s^{\prime}-r^{\prime})}{s^{\prime}+(s^{\prime}-r^{\prime})}\Bigr) .$$
Therefore
\begin{align*}
wL
&=Lw_2L 
=Lw\Bigl( \dfrac{r^{\prime}}{r^{\prime}+s^{\prime}}\Bigr) 
\overset{(\ast 1)}{=} w\Bigl( \dfrac{p^{\prime}}{p^{\prime}+q^{\prime}}\sharp \dfrac{r^{\prime}}{r^{\prime}+s^{\prime}}\Bigr) \\ 
&=w\Bigl( \dfrac{p^{\prime}+r^{\prime}}{p^{\prime}+q^{\prime}+r^{\prime}+s^{\prime}}\Bigr) ,\displaybreak[0]\\ 
wR
&=Lw_2R 
=Lw\Bigl( \dfrac{r^{\prime}+(s^{\prime}-r^{\prime})}{s^{\prime}+(s^{\prime}-r^{\prime})}\Bigr) 
\overset{(\ast 2)}{=} w\Bigl( \dfrac{p^{\prime}+(q^{\prime}-p^{\prime})}{q^{\prime}+(q^{\prime}-p^{\prime})}
\sharp \dfrac{r^{\prime}+(s^{\prime}-r^{\prime})}{s^{\prime}+(s^{\prime}-r^{\prime})}\Bigr) \\ 
&=w\Bigl( \dfrac{(p^{\prime}+r^{\prime})+(q^{\prime}+s^{\prime}-p^{\prime}-r^{\prime})}{(q^{\prime}+s^{\prime})+(q^{\prime}+s^{\prime}-p^{\prime}-r^{\prime})}\Bigr) .
\end{align*}
Here at $(\ast 1)$ we use $\ell \bigl( w\bigl( \frac{r^{\prime}}{r^{\prime}+s^{\prime}}\bigr) \bigr) =\ell (w_2L)>\ell (w_1L)=\ell \bigl( w\bigl( \frac{p^{\prime}}{p^{\prime}+q^{\prime}}\bigr) \bigr) $ and Lemma~\ref{w2-2}. 
We also use at $(\ast 2)$ a similar inequality and Lemma~\ref{w2-2}. 

\par 
Next, we consider the case $\ell (w_1)>\ell (w_2)$. 
Then, $w=Rw_1$, 
$w_1L=w\bigl( \frac{p^{\prime}}{p^{\prime}+q^{\prime}}\bigr) $ and 
$w_1R=w\bigl( \frac{p^{\prime}+(q^{\prime}-p^{\prime})}{q^{\prime}+(q^{\prime}-p^{\prime})}\bigr) $. 
By the same argument in the case $\ell (w_2)>\ell (w_1)$, one can show that 
$wL=w\bigl( \frac{p^{\prime}+r^{\prime}}{p^{\prime}+q^{\prime}+r^{\prime}+s^{\prime}}\bigr) $ and 
$wR
=w\bigl( \frac{(p^{\prime}+r^{\prime})+(q^{\prime}+s^{\prime}-p^{\prime}-r^{\prime})}{(q^{\prime}+s^{\prime})+(q^{\prime}+s^{\prime}-p^{\prime}-r^{\prime})}\bigr) $. 
This shows that the lemma holds for the $LR$ words with length $d$. 
\end{proof}

\par \medskip 
\begin{cor}\label{w2-7}
Let $w$ be an $LR$ word. If $i(w)=w([0; a_2, \ldots , a_m])$ for some even integer $m$, then 
\begin{enumerate}
\item[$(1)$] $wR=w([0; 1, a_2, \ldots , a_m])$, 
\item[$(2)$] $wL=w([0; 2, a_2-1, a_3, \ldots , a_m])$,
\item[$(3)$] $w=w([0; 1, a_2-1, a_3, \ldots , a_m])$.  
\end{enumerate}
\end{cor}
\begin{proof}
(1) By Lemma~\ref{w2-6} we have 
$wR=w\bigl( \frac{p+(q-p)}{q+(q-p)}\bigr) =w\bigl( \frac{q}{q+(q-p)}\bigr) $. 
Since 
$$\dfrac{q}{q+(q-p)}
=0+\dfrac{1}{1+\dfrac{q-p}{q}}$$
and $i(w)=w\bigl( \frac{q-p}{q}\bigr)$ by Lemma~\ref{w2-5}(1), we have 
$\frac{q-p}{q}=[0; a_2, \ldots , a_m]$. Thus, 
$\frac{q}{q+(q-p)}=[0; 1, a_2, \ldots , a_m]$.
\par 
(2) By Lemma~\ref{w2-6} we have 
$wL=w\bigl( \frac{p}{p+q}\bigr) $. 
On the other hand, 
$$\dfrac{p}{p+q}
=0+\dfrac{1}{2+\dfrac{q-p}{p}}.$$
So, it can be written as $\frac{q-p}{p}=[0; b_2, \ldots , b_n]$ for some even integer $n$. 
By a similar argument in (1), we see that $n=m$ and $b_2=a_2-1, b_3=a_3, \ldots , b_m=a_m$. 
Thus $\frac{p}{p+q}=[0; 2, b_2, \ldots , b_m]=[0; 2, a_2-1, a_3, \ldots , a_m]$. 
\par 
(3) Let us write as $\frac{q-p}{p}=[0; b_2, \ldots , b_n]$ for some even integer $n$. 
Then $\frac{p}{q-p}=[b_2, \ldots , b_n]$. 
Thus 
$$\dfrac{q-p}{q}
=0+\dfrac{1}{1+\dfrac{p}{q-p}}
=[0; b_2+1, b_3, \ldots , b_n].$$
Since the continued fraction is equal to $[0; a_2, \ldots , a_m]$, it follows that 
$n=m$ and $b_2=a_2-1,\ b_3=a_3, \ldots , b_m=a_m$. 
Therefore $\frac{p}{q}=
[0; 1, b_2, \ldots , b_n]=[0; 1, a_2-1, a_3, \ldots , a_m]$. 
\end{proof}

\par \medskip 
For an $LR$ word $w$ we denote by $(ir)(w)$ the $LR$ word obtained by the composition of $i$ and $r$. 

\par \medskip 
\begin{lem}\label{w2-8}
Let $\alpha =\frac{x}{y}\in \mathbb{Q}\cap (0,1)$ be an irreducible fraction, and $\bigl( \frac{p}{q}, \frac{r}{s}\bigr)$ be the pair of parents of $\alpha $. 
If $w=w\bigl( \frac{x}{y}\bigr)$, 
then 
$$r(w)= w\Bigl( \dfrac{q}{y}\Bigr),\qquad (ir)(w)= w\Bigl( \dfrac{s}{y}\Bigr).$$
\end{lem}
\begin{proof}
\par 
This lemma is shown by induction on the length of $LR$ words.
\par 
For the $LR$ words $\emptyset ,\ L, R$ the equations in the lemma are satisfied as follows: 
\par 
$\bullet$ If $w=\emptyset $, then $r(w)=\emptyset = w\bigl( \frac{1}{2}\bigr),\  
(ir)(w)=\emptyset = w\bigl( \frac{1}{2}\bigr) $, and the pair of parents of $\frac{1}{2}$ is $\bigl( \frac{0}{1}, \frac{1}{1}\bigr)$. 
\par 
$\bullet$ If $w=L$, then $r(w)=L= w\bigl( \frac{1}{3}\bigr),\  
(ir)(w)=R= w\bigl( \frac{2}{3}\bigr)$, and the pair of parents of $\frac{1}{3}$ is $\bigl( \frac{0}{1}, \frac{1}{2}\bigr)$. 
\par 
$\bullet$ If $w=R$, then $r(w)=R= w\bigl( \frac{2}{3}\bigr),\ 
(ir)(w)=L= w\bigl( \frac{1}{3}\bigr)$, and the pair of parents of $\frac{2}{3}$ is $\bigl( \frac{1}{2}, \frac{1}{1}\bigr)$. 
\par 
Next, assume that the lemma holds for all $LR$ words with length less than $d$. 
Let $w$ be an $LR$ word with length $d$, and set 
$w_1=w\bigl(\frac{p}{q}\bigr) $ and $w_2=w\bigl(\frac{r}{s}\bigr)$. 
\par 
 In the case of $\ell (w_1)>\ell (w_2)$ we see that $w=Rw_1$ and $p\geq r,\ q\geq s$, and $\frac{p-r}{q-s}, \frac{r}{s}$ are Farey neighbors satisfying 
$\frac{p}{q}=\frac{p-r}{q-s}\sharp \frac{r}{s}$. 
Since 
$\ell (w_1)=d-1<d$, by induction hypothesis $r(w_1)= w\bigl( \frac{q-s}{q}\bigr) ,\ (ir)(w_1)= w\bigl( \frac{s}{q}\bigr)$. 
Hence by Lemma~\ref{w2-6}
\begin{align*}
r(w)&=r(w_1)R
=w\Bigl( \dfrac{q-s}{q}\Bigr)R
=w\Bigl(\dfrac{q}{q+s}\Bigr) 
=w\Bigl( \dfrac{q}{y}\Bigr) ,\\ 
(ir)(w)&=(ir)(w_1)L
=w\Bigl( \dfrac{s}{q}\Bigr)L
=w\Bigl( \dfrac{s}{s+q}\Bigr) 
=w\Bigl( \dfrac{s}{y}\Bigr) . 
\end{align*}
\par 
 In the case of $\ell (w_1)<\ell (w_2)$ we see that $w=Lw_2$ and $p\leq r,\ q\leq s$, and $\frac{p}{q}, \frac{r-p}{s-q}$ are Farey neighbors satisfying 
$\frac{r}{s}=\frac{p}{q}\sharp \frac{r-p}{s-q}$. 
Since $r(w_2)= w\bigl( \frac{q}{s}\bigr),\ (ir)(w_1)= w\bigl( \frac{s-q}{s}\bigr)$ by induction hypothesis, 
as the same argument above we have the desired equations. 
\end{proof}

\par \medskip 
\begin{exam}
$(1)$ If $w=LLLR$, then 
$$i(w)=RRRL,\quad r(w)=RLLL,\quad (ir)(w)=LRRR,$$
and the word $w$ corresponds to 
$\frac{5}{9}=\frac{1}{2}\sharp \frac{4}{7}$. Therefore 
we have the following table:
$$\renewcommand{\arraystretch}{1.7}
\begin{array}{c|cccc}
& w & i(w) & r(w) & (ir)(w) \\ \hline 
\text{$LR$ words} & LLLR & RRRL  & RLLL & LRRR \\ \hline 
\text{irreducible fractions} & \frac{5}{9} & \frac{4}{9} & \frac{2}{9} & \frac{7}{9} \\ \hline 
\end{array}$$
\par \bigskip 
$(2)$ If $w=LLRR$, then 
$$i(w)=RRLL,\quad r(w)=RRLL,\quad (ir)(w)=LLRR,$$
and the word $w$ corresponds to 
$\frac{7}{10}=\frac{2}{3}\sharp \frac{5}{7}$. Thus
we have the following table:
$$\renewcommand{\arraystretch}{1.7}
\begin{array}{c|cccc}
& w & i(w) & r(w) & (ir)(w) \\ \hline 
\text{$LR$words} & LLRR & RRLL & RRLL & LLRR \\ \hline 
\text{irreducible fractions} & \frac{7}{10} & \frac{3}{10} & \frac{3}{10} & \frac{7}{10} \\ \hline 
\end{array}$$
\end{exam} 

\subsection{Conway-Coxeter friezes of zigzag-type}\label{subsection2-2}
\par 
A Conway-Coxeter frieze (abbreviated by CCF)  is an array of natural numbers, displayed on shifted lines such that the top and bottom lines are composed only of 1s and each unit diamond 
$\begin{array}{ccc}
& b & \\
a & & d \\
& c & 
\end{array}$  
satisfies the determinant condition $ ad -bc =1$, namely 
$\begin{pmatrix}
a & b \\
c & d 
\end{pmatrix} \in \text{SL}(2,{\Bbb Z})$ and $a,b,c,d >0$.

\medskip 
Figure~\ref{fig12} is an example of a Conway-Coxeter frieze. 
This CCF is constructed from the initial condition given in Figure~\ref{fig10}, which consists of a $1$-zigzag line of type $L^2R^2L$ (see Figure~\ref{fig11}). 

\medskip

\begin{figure}[htbp]
\centering 
$\begin{array}{cccccccccccccccccccccccccc}
1 & & 1 & &1  & & 1 & & 1 & & 1 & & 1 & & 1 & &1  & & 1 &&1&&1&&1& \\\hline 
&  & &  & &  & &  & & 1 & &  & &  & &  & &  & &  &&&&&& \\ 
&&  & &  & &  & & 1 & &   & &  & &  & &  & &  & &  &&&&& \\ 
&&&  & &  & & 1 & &  & &  & &  & &  & &  & &  & &  &&&& \\ 
&&&&  & &  & & 1 & &  & &  & &  & &  & &  & &  & &  &&& \\ 
&&&&&  & &  & & 1 & &  & &  & &  & &  & &  & &  & &  && \\ 
&&&&&&  & & 1 & &  & &  & &  & &  & &  & &  & &  & &  & \\\hline 
&1&&1&&1&&1 & & 1 & & 1 & & 1 & & 1 & & 1 & & 1 & & 1 & & 1 & & 1 \\
\end{array}$
\caption{}\label{fig10}
\end{figure}

\begin{figure}[htbp]
\centering
\includegraphics[width=16cm]{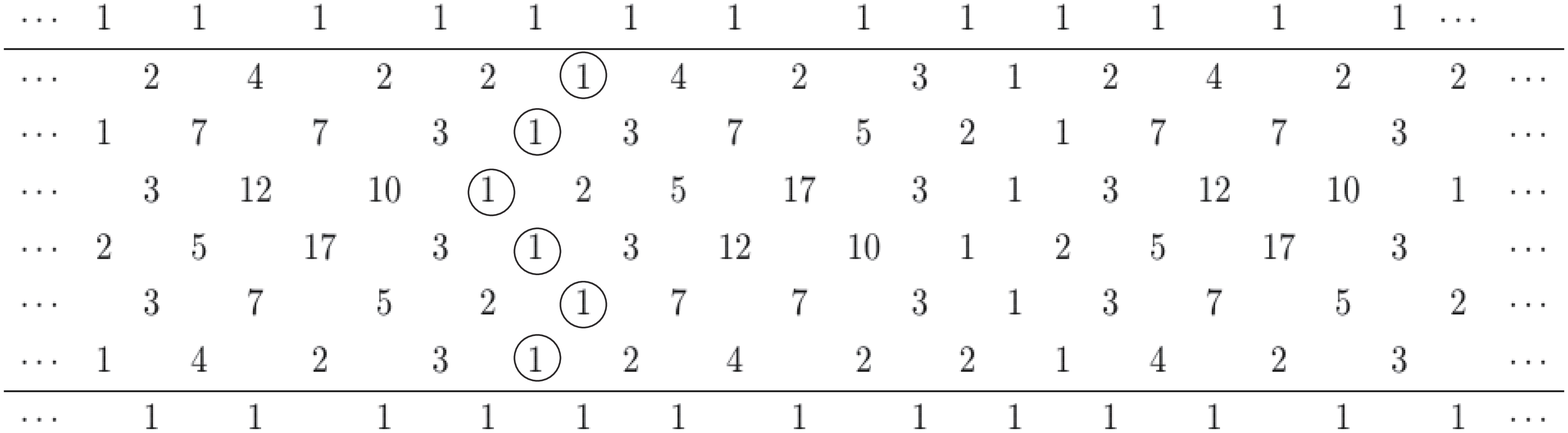}
\caption{}\label{fig11}
\end{figure}

\medskip

\begin{figure}[htbp]
\centering
\scalebox{0.9}[1.05]{
$\begin{array}{c@{\kern0.8em}c@{\kern0.8em}c@{\kern0.8em}c@{\kern0.8em}c@{\kern0.8em}c@{\kern0.8em}c@{\kern0.8em}c@{\kern0.8em}c@{\kern0.8em}c@{\kern0.8em}c@{\kern0.8em}c@{\kern0.8em}c@{\kern0.8em}c@{\kern0.8em}c@{\kern0.8em}c@{\kern0.8em}c@{\kern0.8em}c@{\kern0.8em}c@{\kern0.8em}c@{\kern0.8em}c@{\kern0.8em}c@{\kern0.8em}c@{\kern0.8em}c@{\kern0.8em}c@{\kern0.8em}c@{\kern0.8em}c@{\kern0.8em}c}
\cdots & 1 & & 1 & & 1 & & 1 & & 1 & & 1 & & 1 & & 1 & & 1 & & 1 && 1 && 1 &&1 & \cdots  \\\hline 
\cdots && 2 & & 4 & & 2 & & 2 & & 1 & & 4 & & 2 & & 3 & & 1 & & 2 && 4 && 2 && 2 & \cdots  \\ 
\cdots & 1 && 7 & & 7 & & 3 & & 1 & & 3  & & 7 & & 5 & & 2 & & 1 & & 7 && 7 && 3 && \cdots  \\ 
\cdots && 3 && 12 & & 10 & & 1 & & 2 & & 5 & & 17 & & 3 & & 1 & & 3 & & 12 && 10 && 1 & \cdots  \\ 
\cdots & 2 && 5 && 17 & & 3 & & 1 & & 3 & & 12 & & 10 & & 1 & & 2 & & 5 & & 17 && 3 && \cdots  \\ 
\cdots && 3 &&7 && 5 & & 2 & & 1 & & 7 & & 7 & & 3 & & 1 & & 3 & & 7 & & 5 && 2 & \cdots \\ 
\cdots & 1 && 4 && 2 && 3 & & 1 & & 2 & & 4 & & 2 & & 2 & & 1 & & 4 & & 2 & & 3 && \cdots  \\\hline 
\cdots &&1 &&1 &&1 &&1 & & 1 & & 1 & & 1 & & 1 & & 1 & & 1 & & 1 & & 1 & & 1& \cdots \\
\end{array}$}
\caption{}\label{fig12}
\end{figure}

\medskip 

Does a $1$-zigzag line connecting the ceiling and the floor appear in the CCF?
This is not always true. 
For example, in the CCF given in Figure~\ref{fig13}, a $1$-zigzag line connecting the ceiling and the floor does not  appear.

\begin{figure}[htbp]
\centering
$\begin{array}{cccccccccccccccccccccccccccc}
\cdots & 1 & & 1 & & 1 & & 1 & & 1 & & 1 & & 1 & & 1 & & 1 & & 1 && 1 && 1 &&1 & \cdots  \\\hline 
\cdots&& 1 & & 4 & & 1 & & 2 & & 4 & & 1 & & 2 & & 3 & & 1 & & 4 && 1 && 2 && 4 & \cdots  \\ 
\cdots &&& 3 && 3 & & 1 & & 7 & & 3 & & 1  & & 5 & & 2 & & 3 & & 3 & & 1 && 7 &&  \cdots  \\ 
\cdots && 5 && 2 & & 2 & & 3 & & 5 & & 2 & & 2 & & 3 & & 5 & & 2 & & 2 && 3 && 5 & \cdots  \\ 
\cdots  &&& 3 && 1 && 5 & & 2 & & 3 & & 3 & & 1 & & 7 & & 3 & & 1 & & 5 & & 2 &&  \cdots  \\ 
\cdots && 4 && 1 && 2 && 3 & & 1 & & 4 & & 1 & & 2 & & 4 & & 1 & & 2 & & 3 & & 1 & \cdots  \\\hline 
\cdots &1 &&1 &&1 &&1 & & 1 & & 1 & & 1 & & 1 & & 1 & & 1 & & 1 & & 1 & & 1&& \cdots \\
\end{array}$
\caption{}\label{fig13}
\end{figure}

From now, a CCF where $1$-zigzag appears will be called a {\it CCF of zigzag-type}. 
\medskip

\begin{rem}\label{w2-01}
A CCF is related to a polygon triangulation, and the number of the first row of the CCF corresponds to the number allocated to each vertex of the polygon. 
Where the number of vertices of a polygon is a number obtained by adding $1$ to the number of diagonal lines coming out from the vertex. 
As it can be seen in the following examples (a) and (b), $1$-zigzag does not appear when a triangle with only a diagonal line appears in triangulation of a polygon. 

\begin{enumerate} 
\item[$(a)$] $1$-zigzag appears in CCF.

\includegraphics[width=10cm]{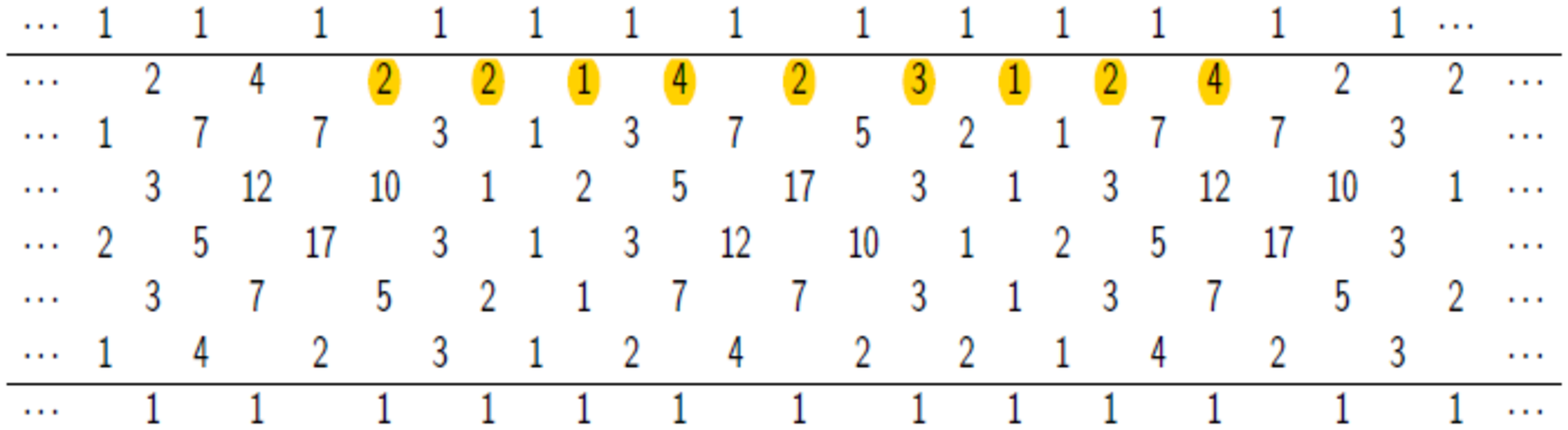} \raisebox{1.5cm}{$~~~\Leftrightarrow~~~$}
\hspace{-0.1cm} \raisebox{-0.1cm}{\includegraphics[width=3.2cm]{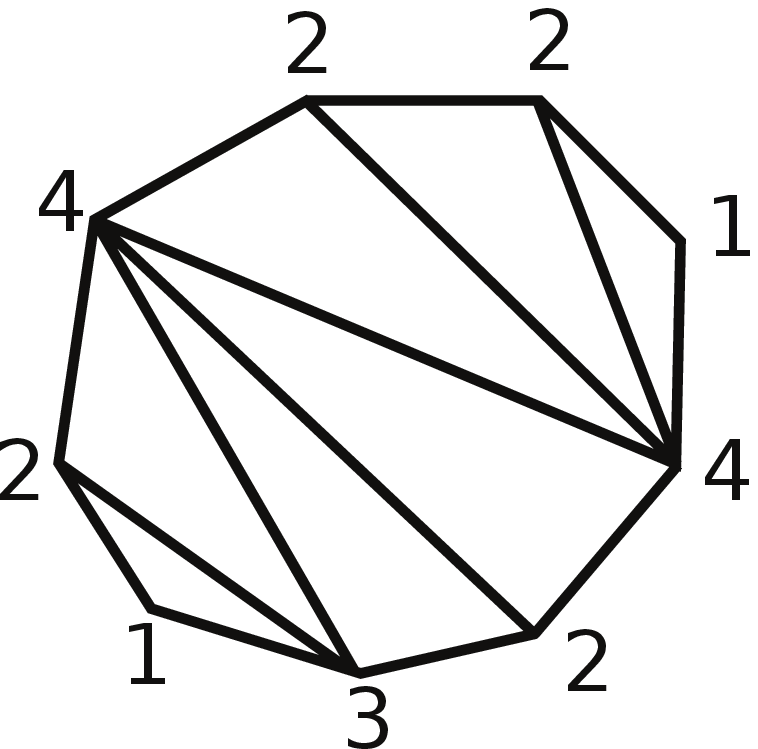}} 

\item[$(b)$] $1$-zigzag does not appear in CCF.

\includegraphics[width=11cm]{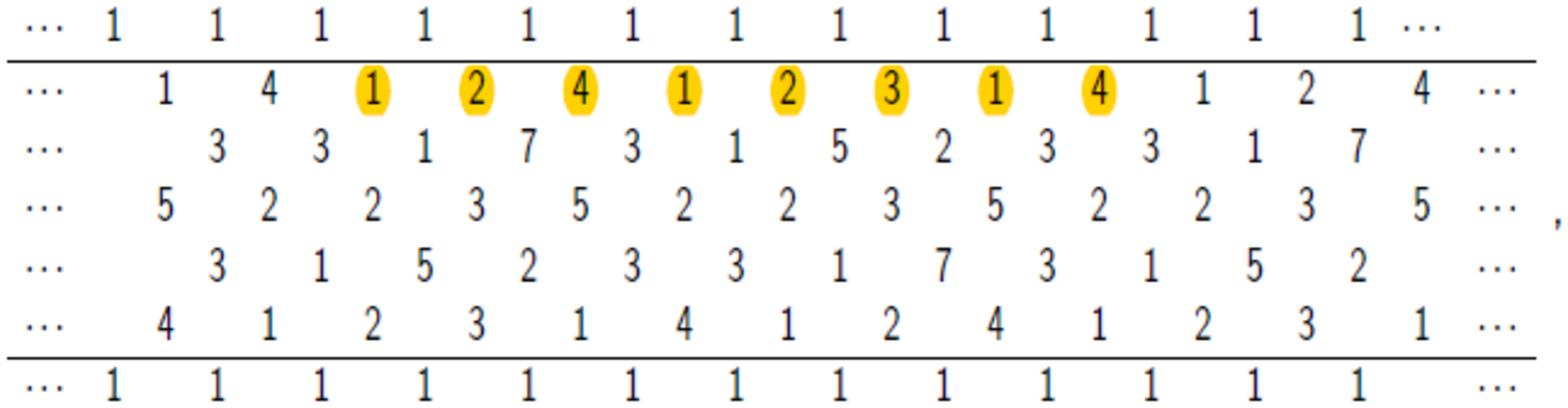} \raisebox{1.3cm}{$\Leftrightarrow~$}
\hspace{-0.1cm} \raisebox{-0.1cm}{\includegraphics[width=3.2cm]{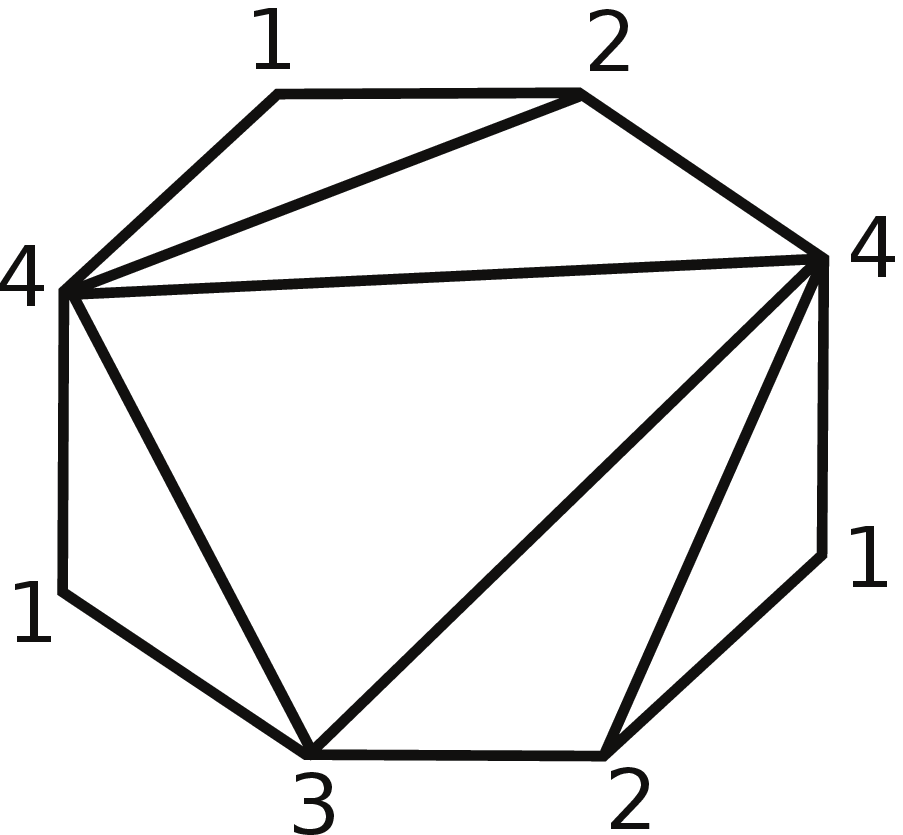}} 

\end{enumerate}
\end{rem}

\subsection{The Kauffman bracket polynomials of Conway-Coxeter friezes}\label{subsection2-3}
\par 
As viewed in the previous subsection any $LR$ word $w$ determines a Conway-Coxeter frieze, which is denoted by $\text{CCF}(w)$ or $\Gamma (w)$. 
However, the converse is not true since it can be also represented as $\text{CCF}(ir(w))$. 
In fact, it is known by Coxeter \cite{Coxeter} that a Conway-Coxeter frieze has a symmetry of glide-refection. 
In other words, it has a fundamental domain under horizontal translations and the reflection with respect to the middle horizontal line (see also \cite{M-G} for this statement).  
\par 
In this subsection we show that the Kauffman bracket polynomial of Conway-Coxeter friezes can be defined up to replacing $A$ with $A^{-1}$, and it has some interesting properties. 
\par 
Any rational number $\alpha $ in the open interval $(0,1)$ has a continued fraction expansion such as 
$\alpha =[0; a_1, a_2, \ldots , a_{n-1}, a_n]$ with an even integer $n$. Then 
we set 
$$ir(\alpha ):=[0; a_n, a_{n-1}, \ldots ,a_1].$$
This operation corresponds just to the operation $ir$ for $LR$ words as shown below.

\par \medskip 
\begin{lem}\label{w2-10}\ 
For a rational number $\alpha \in \mathbb{Q}\cap (0,1)$ we have 
$(ir)\bigl( w(\alpha )\bigr)=w\bigl( (ir)(\alpha )\bigr) $. 
\end{lem} 
\begin{proof}
The equation in the lemma holds for $\alpha =\frac{1}{2}$ since $(ir)\bigl( w(\alpha )\bigr)=\emptyset =w\bigl(\frac{1}{2}\bigr) =w\bigl( (ir)(\alpha )\bigr)$. 
\par 
Assume that the equation in the lemma holds for an $LR$ word $w=w\bigl( \frac{p}{q}\bigr)$. 
\par 
First let us consider $wR$. 
We write as 
$i(w)=w([0; a_1, \ldots , a_m])$ for some odd integer $m$. 
By Corollary~\ref{w2-7} then 
$wR=w([0; 1, a_1, \ldots , a_m])$, and $w=w([0; 1, a_1-1, a_2, \ldots , a_m])$. 
Therefore 
$$(ir)(wR)=L(ir)(w)
=Lw([0; a_m, \ldots , a_2, a_1-1, 1])
=Lw([0; a_m, \ldots , a_2, a_1]).$$
Since $m$ is even, 
$$[0; a_m, \ldots , a_1,1]=[0; a_m, \ldots , a_2, a_1+1]
=[0; a_m, \ldots , a_2]\sharp [0; a_m, \ldots , a_1].$$
Since 
$\ell \bigl( w([0; a_m, \ldots , a_2])\bigr)<\ell \bigl( w([0; a_m, \ldots , a_1])\bigr)$ by Remark~\ref{w2-3}, 
it follows from Lemma~\ref{w2-2} that $Lw([0; a_m, \ldots , a_2, a_1])=w([0; a_m, \ldots , a_1, 1])$. 
Thus $(ir)(wR)=w([0; a_m, \ldots , a_1, 1])$. 
\par 
Next, let us consider $wL$, and write 
$i(w)=w([0; a_1, \ldots , a_m])$ for some odd integer $m$. 
By Corollary~\ref{w2-7} we have 
$wL=w([0; 2, a_1-1, a_2, \ldots , a_m])$, and 
$(ir)(wL)=R(ir)(w)
=Rw([0; a_m, \ldots , a_2, a_1-1, 1])$. 
Since $m$ is odd, 
$$[0; a_m, \ldots , a_1-1, 2]
=[0; a_m, \ldots , a_1-1, 1]\sharp [0; a_m, \ldots , a_1-1].$$
It follows from Lemma~\ref{w2-2} and Remark~\ref{w2-3} that 
$Rw([0; a_m, \ldots , a_2, a_1-1, 1])=w([0; a_m, \ldots , a_1-1,2])$. 
Thus $(ir)(wL)=w([0; a_m, \ldots , a_1-1,2])$. 
\end{proof}

\par \medskip 
\begin{thm}\label{w2-11}
\par 
For a rational number $\alpha \in \mathbb{Q}\cap (0,1)$ 
$$v(\phi (ir(\alpha )))=\overline{v(\phi (\alpha ))}.$$
\end{thm} 
\begin{proof}
Write in the form $\alpha =\frac{p}{q}$ as an irreducible fraction, and expand it as a continued fraction: 
$\frac{p}{q}=[0; a_1, \ldots , a_{n-1}, a_n]$, where $n$ is even, and 
$a_1, \ldots , a_n$ are positive integers. 
By Theorem~\ref{w1-10} and \eqref{eqw1-7} then 
$$v\bigl(\phi (\alpha )\bigr) =(-A^3)^{-a_1+\cdots -a_{n-1}+a_n}\Bigl\langle D\Bigl( T\Bigl(\dfrac{p}{q}\Bigr) \Bigr)\Bigr\rangle .$$
On the other hand, setting 
$ir(\alpha )=[0; a_n, a_{n-1}, \ldots ,a_1]=\frac{p^{\prime}}{q^{\prime}}$, we have 
$$v\bigl(\phi (ir(\alpha ))\bigr) =(-A^3)^{-a_n+\cdots -a_2+a_1}\Bigl\langle D\Bigl( T\Bigl(\dfrac{p^{\prime}}{q^{\prime}}\Bigr) \Bigr)\Bigr\rangle .$$
Since $n$ is even and 
$\frac{p^{\prime}}{q^{\prime}}=[0; a_n, a_{n-1}, \ldots ,a_1]$, we have  
$$T\Bigl( \dfrac{p^{\prime}}{q^{\prime}}\Bigr) =\ \raisebox{-1.5cm}{\includegraphics[height=3cm]{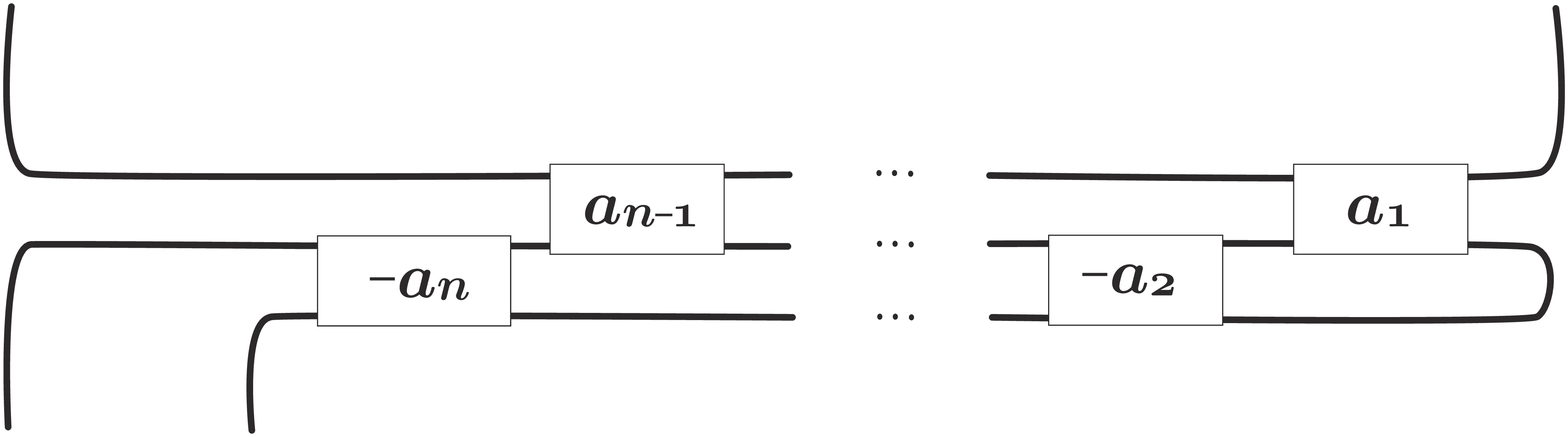}}\ .$$
Hence 
\begin{align*}
D\Bigl( T\Bigl(\dfrac{p^{\prime}}{q^{\prime}}\Bigr) \Bigr)
&=\ \ \raisebox{-0.5cm}{\includegraphics[height=1.5cm]{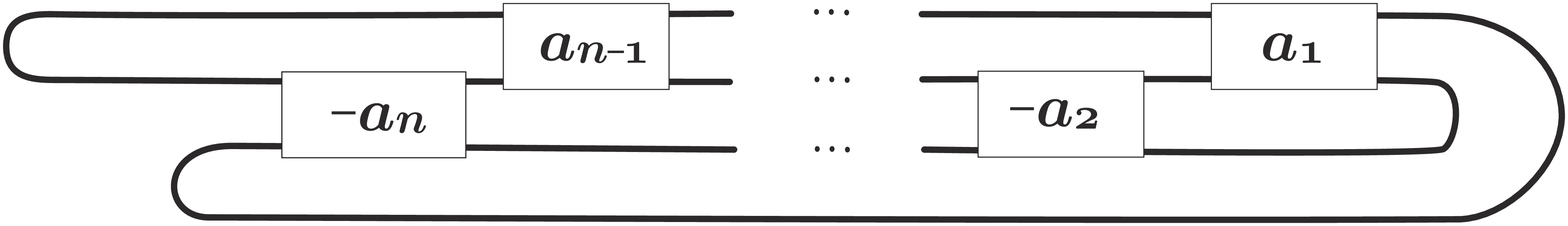}} \\[0.7cm]  
&\sim \ \ \raisebox{-0.7cm}{\includegraphics[height=1.5cm]{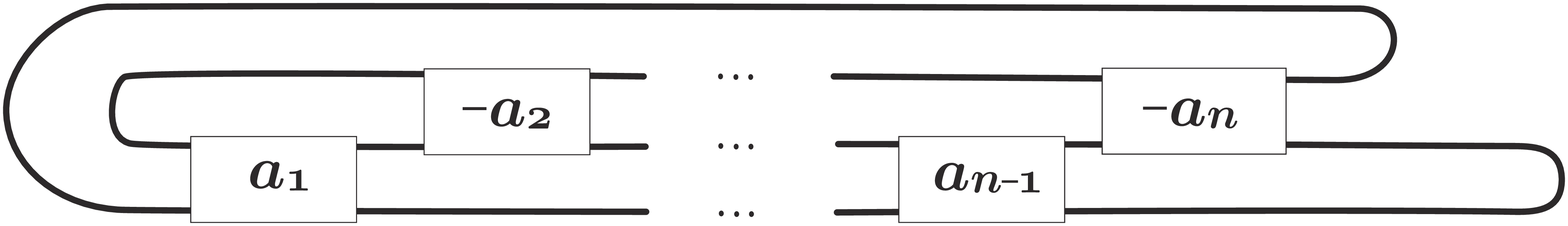}} \\[0.3cm]   
&\sim \ \raisebox{-1.0cm}{\includegraphics[height=2.0cm]{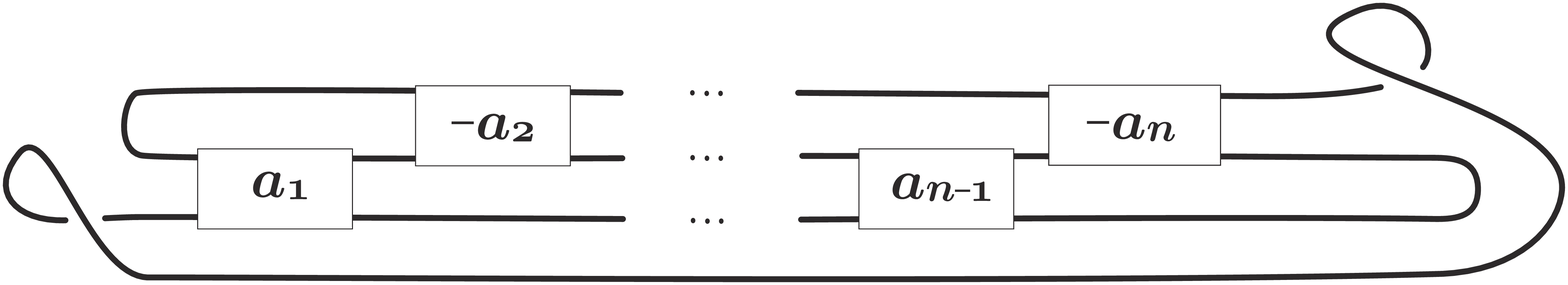}} . 
\end{align*}
Here $\sim$ means what they are regular isotopic. 
Therefore,  
\begin{align*}
\Bigl\langle D\Bigl( T\Bigl(\dfrac{p^{\prime}}{q^{\prime}}\Bigr) \Bigr)\Bigr\rangle 
&=\ \Biggl\langle \ \raisebox{-0.6cm}{\includegraphics[height=1.5cm]{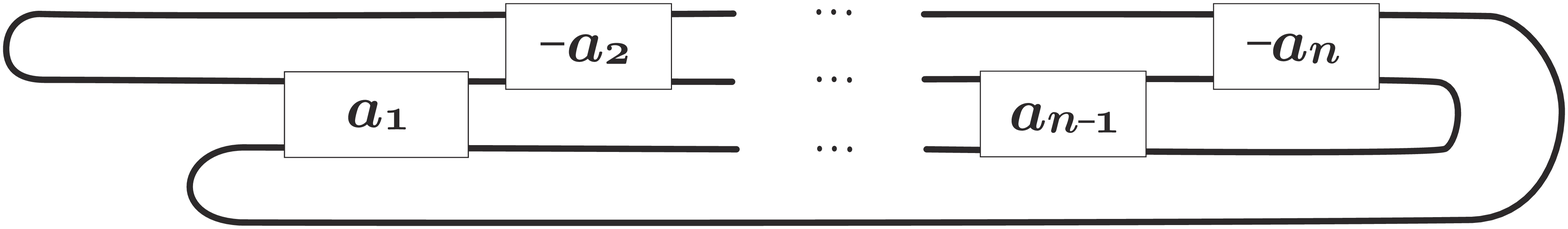}}\ \ \Biggr\rangle \displaybreak[0]\\[0.3cm]  
&=\Bigl\langle D\Bigl( -T\Bigl(\dfrac{p}{q}\Bigr) \Bigr)\Bigr\rangle \displaybreak[0]\\[0.3cm]  
&=\overline{\Bigl\langle D\Bigl( T\Bigl(\dfrac{p}{q}\Bigr) \Bigr)\Bigr\rangle } .
\end{align*}
This shows that 
$v\bigl(\phi (ir(\alpha ))\bigr) =\overline{(-A^3)^{a_n-\cdots +a_2-a_1}\bigl\langle D\bigl( T(\frac{p}{q}) \bigr)\bigr\rangle }
=\overline{v\bigl(\phi (\alpha )\bigr) }$. 
\end{proof}

\par \smallskip 
\begin{defn}
For a Conway-Coxeter frieze $\Gamma $ of zigzag-type we write $\Gamma =\text{CCF}(w)$ by using an $LR$ word $w$. 
Let $\alpha $ be the corresponding rational number of $w$, and define 
\begin{equation}
\langle \Gamma \rangle :=v(\phi (\alpha )).
\end{equation}
We call it the Kauffman bracket polynomial of the Conway-Coxeter frieze $\Gamma $. 
By Theorem~\ref{w2-11} it is well-defined up to replacing $A$ with $A^{-1}$. 
\end{defn} 

\par \medskip 
One can inductively compute the Kauffman bracket polynomial  $\langle \Gamma (w)\rangle $ of a Conway-Coxeter frieze $\Gamma (w)$ with a given word $w$ as the following theorem. 

\par \medskip 
\begin{thm}\label{w2-13}
$$\langle \Gamma (\emptyset )\rangle =-A^4-A^{-4},\quad 
\langle \Gamma (L)\rangle =-A^4+1+A^{-8},\quad 
\langle \Gamma (R)\rangle =A^8+1-A^{-4},$$ 
and for $LR$ words $w, w^{\prime}$ such that these corresponding rational numbers are Farey neighbors, 
\begin{equation}\label{eqw2-3}
\langle \Gamma (w\vee w^{\prime})\rangle =-A^4\langle \Gamma (w)\rangle -A^{-4}\langle \Gamma (w^{\prime})\rangle . 
\end{equation}
In particular, for a positive integer $k$ and an $LR$ word $w$ 
\begin{equation}\label{eqw2-4}
\begin{aligned}
\langle \Gamma (R^kLw)\rangle &=-A^4\langle \Gamma (R^{k-1}Lw)\rangle -A^{-4}\langle \Gamma (w)\rangle ,\\ 
\langle \Gamma (L^kRw)\rangle &=-A^4\langle \Gamma (w)\rangle -A^{-4}\langle \Gamma (L^{k-1}Rw)\rangle .
\end{aligned}
\end{equation}
\end{thm} 
\begin{proof}
The equation \eqref{eqw2-3} is a direct consequence of Lemma~\ref{w2-2} and \eqref{eqw1-16}. 
By the method of construction of the Stern-Brocot tree 
the rational numbers corresponding to $R^{k-1}Lw$ and $w$ are Farey neighbors, and their Farey sum corresponds to $R^kLw$. 
So, by applying \eqref{eqw2-3} the first equation of \eqref{eqw2-4} is obtained. 
The second equation can be verified similarly. 
\end{proof}

\par 
Given a Conway-Coxeter frieze  $\Gamma $, we have a new Conway-Coxeter frieze by applying a vertical reflection. 
The resulting Conway-Coxeter frieze is denoted by $\bot (\Gamma )$. 
For example, if $w=L^3R$, then $\bot (\Gamma )$ is the Conway-Coxeter frieze depicted as in the following: 
$$\raisebox{1.7cm}{$\bot (\Gamma )$\ \ : }\ \ \includegraphics[width=13.5cm]{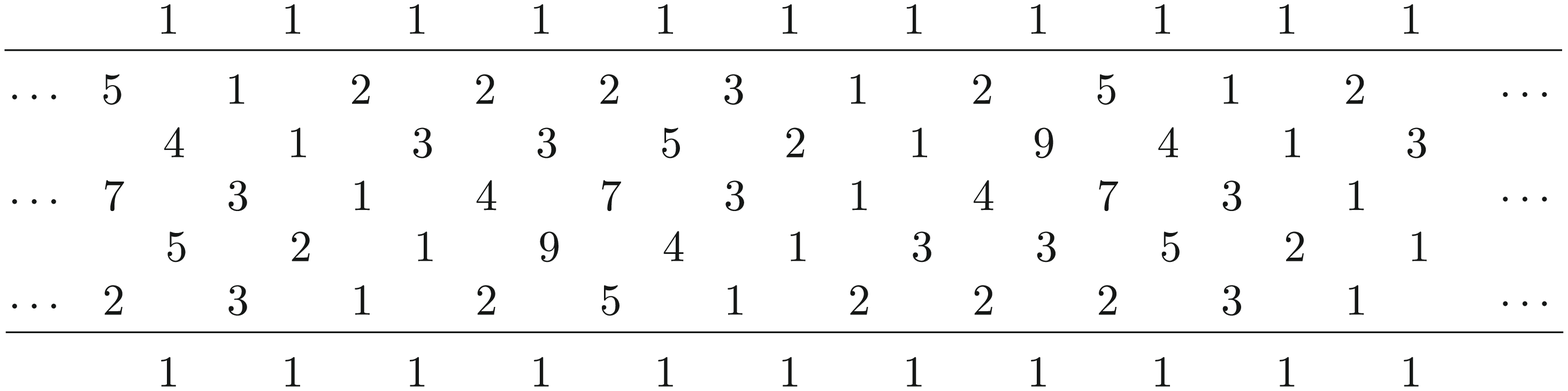}$$

We note that the Conway-Coxeter frieze $\bot (\Gamma )$ coincides with $\text{CCF}(i(w))$ if $\Gamma =\text{CCF}(w)$ for some $LR$ word $w$. 

\par \medskip 
\begin{thm}\label{w2-15}
For a Conway-Coxeter frieze $\Gamma $
$$\langle \textrm{$\bot $}(\Gamma )\rangle \dot{=}\langle \Gamma \rangle ,$$
where $\dot{=}$ means what is equal up to replacing $A$ with $A^{-1}$. 
\end{thm} 
\begin{proof}
It is enough to show that for an $LR$ word $w$ 
\begin{equation}\label{eqw2-5}
v\bigl(\phi \bigl( i(w)\bigr) \bigr) =\overline{v(\phi (w))}
\end{equation}
under the identification of the set of $LR$ words and $\mathbb{Q}\cap (0,1)$. 
We show it by induction on the length of $LR$ words. 
\par 
It can be easily shown that \eqref{eqw2-5} holds for any $LR$ word with length less than or equal to $1$. 
\par 
Assume that \eqref{eqw2-5} holds for all $LR$ words with length less than $d$. 
Let $w$ be an $LR$ word with length $d$. 
Take $LR$ words $w_1, w_2$ such that $w=w_1\vee w_2$, and write as 
$w_1=w\bigl( \frac{p}{q}\bigr) ,\ w_2=w\bigl( \frac{r}{s}\bigr)$. 
Then $\frac{p}{q}, \frac{r}{s}$ are Farey neighbors, and 
by Lemma~\ref{w2-5} 
$i(w_2)=w\bigl( \frac{q-p}{q}\bigr),\ i(w_1)=w\bigl( \frac{s-r}{s}\bigr)$,  
$i(w)=i(w_2)\vee i(w_1)$. 
It follows from Theorem~\ref{w2-13} and induction hypothesis that 
\begin{align*}
v\bigl(\phi \bigl( i(w)\bigr) \bigr) 
&=-A^4v\bigl(\phi \bigl( i(w_2)\bigr) \bigr) -A^{-4}v\bigl(\phi \bigl( i(w_1)\bigr) \bigr)  \\ 
&=-A^4\overline{v(\phi (w_2))}-A^{-4}\overline{v(\phi (w_1))} \displaybreak[0]\\ 
&=\overline{-A^{-4}v(\phi (w_2))-A^4v(\phi (w_1))} \displaybreak[0]\\ 
&=\overline{v(\phi (w))}. 
\end{align*}
Thus \eqref{eqw2-5} holds for $w$. 
\end{proof}

\par \noindent 
\begin{rem}
Let $w$ be an $LR$ word, and $\alpha , \beta$ be the corresponding rational numbers to $w$, $r(w)$, respectively. 
Then by Theorems~\ref{w2-11} and \ref{w2-15} we have 
$v(\phi (\alpha ))=v(\phi (\beta ))$. 
\end{rem} 

\par \noindent 
Combining \eqref{eqw2-5} with Lemma~\ref{w2-5} (1) we have: 

\begin{thm}\label{property_vphi}
Let $\frac{p}{q}$ be an irreducible fraction with $1\leq p\leq q$. 
Then, $\overline{v(\phi (\frac{p}{q}))}=v(\phi (\frac{q-p}{q}))$. 
\end{thm} 

\par \smallskip 
\begin{exam}
For $w=L^3R$ we set $\Gamma :=\text{CCF}(w)$. 
As it has previously observed, then $\bot (\Gamma )$ is the Conway-Coxeter frieze associated to $i(w)=R^3L$ and $\langle \bot (\Gamma )\rangle \dot{=}\langle \Gamma \rangle $ by Theorem~\ref{w2-15}. 
This equation is induced from $v\bigl(\phi \bigl( \frac{4}{9}\bigr)\bigr)=\overline{v\bigl(\phi \bigl( \frac{5}{9}\bigr)\bigr)}$ 
since $w$ and  $i(w)$ correspond to the rational numbers $\frac{5}{9}$ and $\frac{4}{9}$, respectively.
\end{exam} 

\begin{thm}\label{w2-18}
Let $w$ be an $LR$ word, and denote by $|w|_{L}$ and $|w|_{R}$ the numbers of $L$ and $R$ in the word $w$, respectively. 
If $\langle \Gamma (w)\rangle$ is regarded as $v(\phi (w))$, then 
\begin{enumerate}
\item[$(1)$] the minimum and maximum degrees of $\langle \Gamma (w)\rangle$ are $-4(|w|_L+1)$ and $4(|w|_R+1)$, respectively, 
\item[$(2)$] if the coefficient of  $A^{4k}$ in $\langle \Gamma (w)\rangle$ is not $0$, then its sign is $(-1)^k$ for all $k\in \mathbb{Z}$. 
\end{enumerate}
\end{thm} 
\begin{proof}
In this proof $t$ stands for $A^4$, and the minimum and the maximum degrees of $\langle \Gamma (w)\rangle$ are counted as quarters. 
\par 
This theorem can be proved by induction on $\ell (w)$. 
As the first step of induction we will check that the statement is true for $w=L^a$, $R^a$, $R^aL^b$ and $L^aR^b$ with some positive integers $a, b$. 
\par 
It is easily verified for $w=L$. 
For $w=L^a$ with $a\geq 2$, we have
$\langle \Gamma (L^a)\rangle =-t-t^{-1}\langle \Gamma (L^{a-1})\rangle $ since $w\bigl( \frac{1}{a+2}\bigr) =L^a$ and $\frac{1}{a+2}=\frac{0}{1}\sharp \frac{1}{a+1}$. 
Therefore by applying an induction argument, it can be shown that the statement is true for $w=L^a$ for all positive integers $a$. 
If $w=R^aL^b$, then $w=w\bigl( \frac{a+1}{a(b+1)+b+2}\bigr)$ and $\frac{a+1}{a(b+1)+b+2}=\frac{a}{ab+a+1}\sharp \frac{1}{b+1}$. 
Hence, we have $\langle \Gamma (R^aL^b)\rangle =-t\langle \Gamma (R^{a-1}L^b)\rangle -t^{-1}\langle \Gamma (L^{b-1})\rangle $. 
By an induction argument again it follows that the statement in the theorem is true for $w=R^aL^b$ for all positive integers $a, b$. 
\par 
Applying the operation $i$ to $R^a$ and $L^aR^b$ and using the equation \eqref{eqw2-5} we see that 
the statement is true for such all words.  
\par 
Next, let $l$ be a positive integer, and assume that the statement is true for all words $w_1$ with $\ell (w_1)\leq l$. 
It is sufficient to show that it is also true for $Lw_1$ and $Rw_1$ in the situation that $w_1$ is written in the form $w_1=L^aRw_2$ or $w_1=R^aLw_2$ for some positive integer $a$ and some $LR$ word $w_2$. 
Let $\frac{r}{s}$ be the rational number corresponding to $w_1$. 
\par 
Let us consider the case $w_1=L^aRw_2$. 
Then, by setting $w_2=w\bigl( \frac{p}{q}\bigr)$ we have $Lw_1=w\bigl( \frac{p}{q}\sharp \frac{r}{s}\bigr)$, and 
hence $\langle \Gamma (Lw_1)\rangle =-t\langle \Gamma (w_2)\rangle -t^{-1}\langle \Gamma (w_1)\rangle$. 
By induction hypothesis the minimum degrees of $\langle \Gamma (w_1)\rangle $ and $\langle \Gamma (w_2)\rangle $ are 
$-(|w_1|_L+1)$ and $-(|w_2|_L+1)$, respectively, and the maximum degrees of $\langle \Gamma (w_1)\rangle $ and $\langle \Gamma (w_2)\rangle $ are $|w_1|_R+1$ and $|w_2|_R+1$, respectively. 
Since $|w_1|_L=|w_2|_L+a$ and $|w_1|_R=|w_2|_R+1$, it follows that the minimum and the maximum degrees of $\langle \Gamma (Lw_1)\rangle $ are $-(|w_1|_L+2)=-(|Lw_1|_L+1)$ and $|w_2|_R+2=|Lw_1|_R+1$, respectively. 
Part (2) follows from the recursive formula of $\langle \Gamma (Lw_1)\rangle$, immediately. 
Similarly, by setting $L^{a-1}Rw_2=w\bigl( \frac{x}{y}\bigr)$, we have $Rw_1=w\bigl( \frac{r}{s}\sharp \frac{x}{y}\bigr)$ and 
 $\langle \Gamma (Rw_1)\rangle =-t\langle \Gamma (w_1)\rangle -t^{-1}\langle \Gamma (L^{a-1}Rw_2)\rangle$. 
By induction hypothesis and this equation it can be shown that the minimum and the maximum degrees of 
$\langle \Gamma (Rw_1)\rangle $ are $-(|w_2|_L+a+1)=-(|Rw_1|_L+1)$ and $|w_2|_R+3=|Rw_1|_R+1$, respectively. 
\par 
The same argument is applicable to the case $w_1=R^aLw_2$. 
In this case, 
if we set $R^{a-1}Lw_2=w\bigl( \frac{p}{q}\bigr)$, then $Lw_1=w\bigl( \frac{p}{q}\sharp \frac{r}{s}\bigr)$ and $\langle \Gamma (Lw_1)\rangle =-t\langle \Gamma (R^{a-1}Lw_2)\rangle -t^{-1}\langle \Gamma (w_1)\rangle$. 
Similarly, if we set $w_2=w\bigl( \frac{x}{y}\bigr)$,  then $Rw_1=w\bigl( \frac{r}{s}\sharp \frac{x}{y}\bigr)$ and 
 $\langle \Gamma (Rw_1)\rangle =-t\langle \Gamma (w_1)\rangle -t^{-1}\langle \Gamma (w_2)\rangle$. 
By using these equations one can easily show that the statement in the theorem is true for $\langle \Gamma (Lw_1)\rangle $ and $\langle \Gamma (Rw_1)\rangle $ under the condition $w_1=R^aLw_2$. 
\end{proof}

\section{Complete invariants on Conway-Coxeter friezes of zigzag-type}
\par 
Let us recall the computation method of the Kauffman bracket of the rational tangle diagram $T(\alpha )$ associated with a rational number $\alpha $ using the Stern-Brocot tree described in Subsection~\ref{subsection1-3}. 
Instead of the step (iv) in that process, we take the product of all  $-t^{\pm 1}$ on the path from $\alpha $ to  $\frac{1}{2}$, and take the sum of them running over such all paths. 
We denote the resulting Laurent polynomial by $Q_{\alpha }(t)$. 
Similarly take the product of all  $-t^{\pm 1}$ on the path from $\alpha $ to  $\frac{1}{1}$ (or equivalently $\frac{0}{1}$), and take the sum of them running over such all paths. 
We denote the resulting Laurent polynomial by $R_{\alpha }(t)$.  
Then the Kauffman bracket polynomial of the Conway-Coxeter frieze $\Gamma (w(\alpha ))$ can be computed by the formula: 
\begin{equation}\label{eqw3-1}
\langle \Gamma (w(\alpha ))\rangle =(-t-t^{-1})Q_{\alpha }(t)+R_{\alpha }(t)
\end{equation}
for substituting $t=A^4$. 
To derive the equation we define 
\begin{align*}
\Bigl[ \dfrac{1}{2}\Bigr] &:=(1-t)[0]+t^{-2}[\infty ]\in \Lambda ^2,\\ 
[1] &:=-t[0]-t^{-1}[\infty ]\in \Lambda ^2.
\end{align*}

\begin{lem}\label{w3-1}
\begin{enumerate}
\item[$(1)$] $Q_{\frac{1}{2}}(t)=1,\ R_{\frac{1}{2}}(t)=0$, $Q_{\frac{0}{1}}(t)=Q_{\frac{1}{1}}(t)=0,\ R_{\frac{0}{1}}(t)=R_{\frac{1}{1}}(t)=1$, 
$Q_{\frac{1}{3}}(t)=-t^{-1},\ R_{\frac{1}{3}}(t)=-t$, 
$Q_{\frac{2}{3}}(t)=-t,\ R_{\frac{2}{3}}(t)=-t^{-1}$, and if $(y, z)$ is the pair of parents of 
$x\in \mathbb{Q}\cap (0,1)$, then 
\begin{align*}
Q_{x}(t)&=-tQ_{y}(t)-t^{-1}Q_{z}(t),\\ 
R_{x}(t)&=-tR_{y}(t)-t^{-1}R_{z}(t).
\end{align*}
\item[$(2)$] For $\alpha =\frac{p}{q}\in \mathbb{Q}\cap (0,1)$, 
\begin{enumerate}
\item[$(i)$] if $2p>q$, then $\tilde{\phi }(\alpha )=Q_{\alpha }(t)\bigl[ \frac{1}{2}\bigr] +R_{\alpha}(t)[1]$, 
\item[$(ii)$] if $2p<q$, then $\tilde{\phi }(\alpha )=Q_{\alpha }(t)\bigl[ \frac{1}{2}\bigr] +R_{\alpha}(t)[0]$. 
\end{enumerate}
\item[$(3)$] For any $\alpha \in \mathbb{Q}\cap (0,1)$
$$\langle \Gamma (w(\alpha ))\rangle =(-t-t^{-1})Q_{\alpha }(t)+R_{\alpha }(t).$$
\end{enumerate}
\end{lem} 
\begin{proof}
Parts (1) and (2) are direct consequences of the definitions of $Q_{\alpha}(t)$ and $R_{\alpha }(t)$. 
Let us show Part (3). 
We set $\alpha =\frac{p}{q}$ as an irreducible fraction. 
\par 
\noindent 
If $2p>q$, then by Part (2) 
\begin{align*}
\tilde{\phi }(\alpha )
&=Q_{\alpha }(t)\Bigl[ \dfrac{1}{2}\Bigr] +R_{\alpha}(t)[1] \\ 
&=\bigl( (1-t)Q_{\alpha }(t)-tR_{\alpha }(t)\bigr) [0]+\bigl( t^{-2}Q_{\alpha }(t)-t^{-1}R_{\alpha }(t)\bigr) [\infty ]. 
\end{align*}
It follows that 
\begin{align*}
v(\phi (\alpha ))
&=\text{tr}\bigl( \tilde{\phi }(\alpha )\bigr) \\ 
&=\bigl( t^{-2}Q_{\alpha }(t)-t^{-1}R_{\alpha }(t)\bigr) (-t(t+1))+\bigl( (1-t)Q_{\alpha }(t)-tR_{\alpha }(t)\bigr) \displaybreak[0]\\ 
&=(-t-t^{-1})Q_{\alpha }(t)+R_{\alpha }(t). 
\end{align*}
Similarly the same formula can be obtained in the case of $2p<q$. 
\end{proof}

\par \medskip 
\begin{prop}\label{w3-2}
Let $\alpha =\frac{p}{q}$ be an rational number in $(0,1)$. 
\begin{enumerate}
\item[$(1)$] If $2p>q$, then 
\begin{enumerate}
\item[$(i)$] $Q_{\frac{p}{q}}(-1)+R_{\frac{p}{q}}(-1)=p$, 
\item[$(ii)$] $Q_{\frac{p}{q}}(-1)=q-p$.
\end{enumerate}
\item[$(2)$] If $2p<q$, then 
\begin{enumerate}
\item[$(i)$] $Q_{\frac{p}{q}}(-1)=p$, 
\item[$(ii)$] $Q_{\frac{p}{q}}(-1)+R_{\frac{p}{q}}(-1)=q-p$.
\end{enumerate}
\end{enumerate}
\end{prop} 
\begin{proof}
If $\frac{p}{q}$ is a vertex in the branch consisting of the lower side from $\frac{1}{3}$ in the Stern-Brocot tree, then $\frac{p}{q}<\frac{1}{2}$, or equivalently $2p<q$. 
If $\frac{p}{q}$ is in the branch consisting of  the lower side from $\frac{2}{3}$, then $2p>q$. 
Since all rational numbers in $(0,1)$ can be appeared in the Stern-Brocot tree,  by the definition of Farey sum and the construction of the tree, for all $\frac{x}{y}\in \mathbb{Q}\cap (0,1)$ except for 
$\frac{1}{2}, \frac{1}{k+2}, \frac{k+1}{k+2}\ (k=1,2,\ldots )$, 
the pair of parents $\bigl( \frac{p}{q}, \frac{r}{s}\bigr)$ of $\frac{x}{y}$ satisfy
\begin{align*}
2x<y\ \ & \Longleftrightarrow \ \ 2p<q\ \ \text{and}\ \ 2r\leq s,\\ 
2x>y\ \ & \Longleftrightarrow \ \ 2p\geq q\ \ \text{and}\ \ 2r>s.
\end{align*}
(If $\frac{1}{2}$ is contained as a parent, then an equality sign holds in the above inequalities.) 
Since 
$Q_{\frac{1}{2}}(-1)+R_{\frac{1}{2}}(-1)=Q_{\frac{1}{2}}(-1)=1$, the equations in Parts (1) and (2) hold for $\frac{p}{q}=\frac{1}{2}$. 
\par 
(1) Let $\frac{p}{q}, \frac{r}{s}\in \mathbb{Q}\cap (0,1)$ be Farey neighbors, and set $\frac{x}{y}=\frac{p}{q}\sharp \frac{r}{s}$. 
\par 
(i) Suppose that $2p\geq q$, $2r>s$, and 
$Q_{\frac{p}{q}}(-1)+R_{\frac{p}{q}}(-1)=p,\ Q_{\frac{r}{s}}(-1)+R_{\frac{r}{s}}(-1)=r$. 
Then 
\begin{align*}
Q_{\frac{x}{y}}(-1)+R_{\frac{x}{y}}(-1)
&=-(-1)Q_{\frac{p}{q}}(-1)-(-1)^{-1}Q_{\frac{r}{s}}(-1)-(-1)R_{\frac{p}{q}}(-1)-(-1)^{-1}R_{\frac{r}{s}}(-1) \\ 
&=p+r=x. 
\end{align*}

The pair of parents of $\frac{k+1}{k+2}$ is $\bigl( \frac{k}{k+1}, \frac{1}{1}\bigr)$ and 
$Q_{\frac{1}{1}}(-1)+R_{\frac{1}{1}}(-1)=0+1=1$. 
It follows that 
(1)(i) holds for $\frac{p}{q}=\frac{1}{1}$. 
Thus by induction argument we have 
\begin{align*}
Q_{\frac{k+1}{k+2}}(-1)+R_{\frac{k+1}{k+2}}(-1)
&=-(-1)Q_{\frac{k}{k+1}}(-1)-(-1)^{-1}Q_{\frac{1}{1}}(-1)-(-1)R_{\frac{k}{k+1}}(-1)-(-1)^{-1}R_{\frac{1}{1}}(-1) \\ 
&=Q_{\frac{k}{k+1}}(-1)+R_{\frac{k}{k+1}}(-1)+1\\ 
&=\cdots =Q_{\frac{1}{2}}(-1)+R_{\frac{1}{2}}(-1)+k=k+1.
\end{align*}
\par 
(ii) Suppose that $2p\geq q$, $2r>s$, and 
$Q_{\frac{p}{q}}(-1)=q-p,\ Q_{\frac{r}{s}}(-1)=s-r$. 
Then 
\begin{align*}
Q_{\frac{x}{y}}(-1)
&=-(-1)Q_{\frac{p}{q}}(-1)-(-1)^{-1}Q_{\frac{r}{s}}(-1) \\ 
&=q-p+s-r=y-x.
\end{align*}

Since the pair of parents of $\frac{k+1}{k+2}$ is $\bigl( \frac{k}{k+1}, \frac{1}{1}\bigr)$, and 
$Q_{\frac{1}{1}}(-1)=0=1-1$
we see that (1)(ii) holds for $\frac{p}{q}=\frac{1}{1}$. Thus 
\begin{align*}
Q_{\frac{k+1}{k+2}}(-1)
&=-(-1)Q_{\frac{k}{k+1}}(-1)-(-1)^{-1}Q_{\frac{1}{1}}(-1) \\ 
&=Q_{\frac{k}{k+1}}(-1)\\ 
&=\cdots =Q_{\frac{1}{2}}(-1)=1=(k+2)-(k+1).
\end{align*}
\par 
Part (2) can be shown as the same manner of the proof of Part (1) by using the fact that 
 the pair of parents of $\frac{1}{k+2}$ is $\bigl( \frac{0}{1}, \frac{1}{k+1}\bigr)$. 
\end{proof}

\par \medskip 
From Proposition~\ref{w3-2} we get immediately the following result, which is announced in \cite{KW1}.  

\par \smallskip 
\begin{cor}\label{w3-3}
Let $\frac{p}{q}$ be an irreducible fraction with $1\leq p\leq q$. 
Let $\omega \in \mathbb{C}$ be a primitive $8$th root of unity, and 
 $v(\phi (\frac{p}{q}))(\omega )$ be the resulting value obtained by substituting $A=\omega $ to $v(\phi (\frac{p}{q}))$. 
Then, $v(\phi (\frac{p}{q}))(\omega )=q$. 
\end{cor} 
\begin{proof}
Set $\alpha =\frac{p}{q}$. Then 
by the equation \eqref{eqw3-1}
$$v(\phi (\alpha ))|_{A=\omega }=2Q_{\alpha }(-1)+R_{\alpha }(-1).$$
The right-hand side of the above is equal to $q$ by Proposition~\ref{w3-2}. 
This implies that 
$v(\phi (\alpha ))|_{A=\omega}$\newline $=q$. 
\end{proof}

\par 
\begin{rem}
For a link diagram $D$, $\langle D \rangle_{A=\omega} $  is the complex number obtained by substituting $A= \omega$  into the Kauffman bracket polynomial $\langle D \rangle$. 
Then, it is known that the absolute value $|\langle D\rangle _{A=\omega }|$ is a non-negative integer, and it only depends on ambient isotopy class of $D$. 
So, for a link $L$ the value $\text{det}(L):=|\langle D\rangle _{A=\omega }|$ is well-defined with independence of the choice of diagrams $D$ of $L$. 
If $L$ is ambient isotopic to the rational link $N(T(\frac{p}{q}))$ for an irreducible fraction $\frac{p}{q}$, then $\text{det}(L)=|p|$. 
Since $\bigl( T(\frac{p}{q})\bigr) ^{\text{in}}$ is isotopic to $T(\frac{q}{p})$, by \eqref{eqw1-6} and \eqref{eqw1-7} we have 
$\text{det}(D(T(\frac{p}{q})))=|q|$. 
The Part (2) in the above theorem confirms this result. 
\end{rem} 

\par \medskip 
As another application of Proposition~\ref{w3-2} we have: 

\par \medskip 
\begin{thm}\label{w3-5}
For an $LR$ word $w$, we set 
\par \smallskip \centerline{$C_w:=\bigl\{ \frac{Q_w(-1)}{2Q_w(-1)+R_w(-1)}, \frac{Q_w(-1)+R_w(-1)}{2Q_w(-1)+R_w(-1)}, 
\frac{Q_{r(w)}(-1)}{2Q_{r(w)}(-1)+R_{r(w)}(-1)}, \frac{Q_{r(w)}(-1)+R_{r(w)}(-1)}{2Q_{r(w)}(-1)+R_{r(w)}(-1)} \bigr\} \subset \mathbb{Q}\cap (0,1).$}
\par \smallskip \noindent  
If $w=w\bigl( \frac{p}{q}\bigr) ,\ r(w)=w\bigl( \frac{p^{\prime}}{q}\bigr) $, then 
$C_w=\{ \frac{p}{q}, \frac{q-p}{q}, \frac{p^{\prime}}{q}, \frac{q-p^{\prime}}{q}\}$. 
Therefore, $C_w$ is a complete invariant of the Conway-Coxeter frieze $\Gamma (w)$. 
\end{thm} 
\begin{proof}
To show that $C_w$ is a complete invariant of the Conway-Coxeter frieze $\Gamma (w)$, 
suppose that $\Gamma (w)=\Gamma (w^{\prime})$ for another $LR$ word $w^{\prime}$. 
Then $w^{\prime}=w$ or $w^{\prime}=i(w)$. 
If $w$ is expressed as $w=w\bigl( \frac{p}{q}\bigr)$ by some irreducible fraction $\frac{p}{q}$, 
then $r(w)=w\bigl( \frac{p^{\prime}}{q}\bigr)$ and $(ir)(w)=w\bigl( \frac{q-p^{\prime}}{q}\bigr)$ 
for some positive integer $p^{\prime}$ by Lemma~\ref{w2-8}. 
Thus, for a Conway-Coxeter frieze $\Gamma$ of zigzag-type the set $C_w$ does not depend on the choice of $w$ such that $\Gamma =\Gamma (w)$. 
This means that $C_w$ gives a complete invariant of Conway-Coxeter friezes of zigzag-type. 
\end{proof}

\section{Recipe of making Kauffman bracket polynomials by using CCF's}
In this section, we explain that one can recognize Theorem~\ref{w2-13} and Theorem~\ref{w2-18} in view point from deleting ``trigonometric-curves'' in a CCF. 
For simplicity we write $D(\frac{p}{q})$ for the denominator of the rational tangle diagram $T(\frac{p}{q})$. 
\par 
For irreducible fractions $\frac{s}{M}, \frac{u}{M}, \frac{v}{M}, \frac{t}{M}$ that are related as in Figure~\ref{figw14},   
the associated Kauffman bracket polynomials  
$\langle D(\frac{s}{M}) \rangle , \langle D(\frac{u}{M}) \rangle , \langle D(\frac{v}{M}) \rangle , \langle D(\frac{t}{M}) \rangle$ are determined by using a ``$(s \rightarrow M \rightarrow u)$-curve" and a ``$(t \rightarrow M \rightarrow v)$-curve" in a CCF  as Figure. 14 whose maximum value is $M$ as follows:

\noindent 
For Conway-Coxeter frieze $\textrm{CCF}(\frac{s}{M})$ like Figure. 14, the $(s \rightarrow M \rightarrow u)$-curve is a kind of trigonometric curve in $\textrm{CCF}(\frac{s}{M})$ defined to pass determined by specifying the route $s \rightarrow M \rightarrow u$.
 Similarly  the $(t \rightarrow M \rightarrow v)$-curve is defined. These trigonometric curves are determined uniquely in the CCF of zigzag type.

In order to compute $\langle D(\frac{s}{M}) \rangle$ one may take the following steps: 

(i) First of all, make such a CCF as follows. 
Place a natural number in the diamond type surrounding $ M $. 
Since each small diamond is an element of $\text{SL}(2, {\Bbb Z}) $ and  $M$ is the maximum value, all other natural numbers in the CCF are determined from $ M $. 
We write this CCF by $\Gamma $. 
\medskip

\begin{figure}[htbp]
\centering 
$$\begin{array}{cccccccccccccccccccccccc} 
1 &&1 &&1 &&1 &&1 && 1 &&1 &&1 && 1 &&1 &&1 &&1 &   \\\hline 
&&&&&&&&&&&&&&&&&&&&&&& \\
&&&&&&&&&&&&&&&&&&&&&&& \\
&&&&&&&&& \ddots&& \rotatebox[origin=c]{90}{$\ddots$} & &&&&&&&&&&& \\
&&&&&&&&\ddots  && a &&\rotatebox[origin=c]{90}{$\ddots$} &&&&&&&&&&& \\
&&&&&&& \ddots&&  s && v &&\rotatebox[origin=c]{90}{$\ddots$} &&&&&&&&&& \\ 
&&&&&&&& b && M &&  d &&&&&&&&&&& \\
&&&&&&& \rotatebox[origin=c]{90}{$\ddots$}&& t && u &&\ddots&&&&&&&&&& \\ 
&&&&&&&& \rotatebox[origin=c]{90}{$\ddots$}&& c &&\ddots&&&&&&&&&&& \\
&&&&&&&&& \rotatebox[origin=c]{90}{$\ddots$}&& \ddots&&&&&&&&&&&& \\
&&&&&&&&&&&&&&&&&&&&&&& \\
&&&&&&&&&&&&&&&&&&&&&&& \\\hline  
1 &&1 &&1 &&1 &&1 && 1 &&1 &&1 && 1 &&1 &&1 && 1 &
\end{array}  $$
\caption{}\label{figw14}
\end{figure}

\medskip

(ii) Next, consider the  $(t \rightarrow M \rightarrow v)$-curve  (magenta line) and the  $(s \rightarrow M \rightarrow u)$-curve (blue line) passing through the maximum value ``$ M $" in the CCF $\Gamma$.

\medskip 

\begin{figure}[htbp]
\centering
{\mathversion{bold}
\scalebox{0.9}[0.9]{
$\begin{array}{ccccccccccccccccccccccccccccc} 
1 && 1 &&\textcolor{blue}{\bf 1} &&1 &&1 && 1 &&1 &&1 && \textcolor{magenta}{\bf 1} && 1 &&1 &&1 &&1 &&1 &&1   \\\hline 
&&&&&\textcolor{blue}{\bf \ddots}&&&&&&&&&& \textcolor{magenta}{\bf\rotatebox[origin=c]{90}{$\ddots$}} &&&&&&&&&&&& \\
&&&&&&\textcolor{blue}{\bf \ddots}&&&&&&&& \textcolor{magenta}{\bf\rotatebox[origin=c]{90}{$\ddots$}}&&&&&&&&&&&&&& \\
&&&&&&&\textcolor{blue}{\bf \ddots} && \ddots&&  \rotatebox[origin=c]{90}{$\ddots$} && \textcolor{magenta}{\bf\rotatebox[origin=c]{90}{$\ddots$}}&&&&&&&&&&&&&&& \\
&&&&&&&&\textcolor{blue}{\bf \ddots}  && a &&\textcolor{magenta}{\rotatebox[origin=c]{90}{$\ddots$} }&&&&&&&&&&&&&&&& \\
&&&&&&& \ddots&&  \text{\textcolor{blue}{$s$}} && \text{\textcolor{magenta}{$v$}} &&\rotatebox[origin=c]{90}{$\ddots$} &&&&&&&&&&&&&&& \\ 
&&&&&&&& b && \text{\textcolor{red}{$M$}} &&  d &&&&&&&&&&&&&&&& \\
&&&&&&& \rotatebox[origin=c]{90}{$\ddots$}&& \text{\textcolor{magenta}{$t$}} && \text{\textcolor{blue}{$u$}} &&\ddots&&&&&&&&&&&&&&& \\ 
&&&&&&&&\textcolor{magenta}{\bf \rotatebox[origin=c]{90}{$\ddots$}}&& c &&\textcolor{blue}{\bf \ddots}&&&&&&&&&&&&&&&& \\
&&&&&&& \textcolor{green}{\bf \rotatebox[origin=c]{90}{$\ddots$}}&& \rotatebox[origin=c]{90}{$\ddots$}&& \ddots&&\textcolor{blue}{\bf \ddots}&&&&&&&&&&&&&& \\
&&&&&&\textcolor{magenta}{\bf \rotatebox[origin=c]{90}{$\ddots$}}&&                                             &&           &&&&\textcolor{blue}
{\bf \ddots}&&&&&&&&&&&&&& \\
&&&&& \textcolor{magenta}{\bf \rotatebox[origin=c]{90}{$\ddots$}}&&                                              &&           &&&&&&
\textcolor{blue}{\bf \ddots}&&&&&&&&&\\\hline  
1 &&1 &&\textcolor{magenta}{\bf 1} &&1 &&1 && 1 &&1 &&1 && \textcolor{blue}{\bf 1} && 1 &&1 && 1&&1 &&1 &&1  
\end{array}  $
}
}
\caption{}\label{figw15}
\end{figure}

\medskip

(iii) In order to compute the Kauffman bracket polynomial of the CCF $\Gamma $ associated with the fraction $\frac{s}{M}$, pay attention to the magenta   $( t \rightarrow M \rightarrow v) $-curve  and look at the maximum value ``$ M $" like, the following {\bf (S)}.

\medskip

\noindent 
{\bf (S)} 
Picking up the ``sinusoidal" shape part of the magenta, bend it at the maximum value ``$M$":

\medskip

\begin{figure}[htbp]
\centering
{\mathversion{bold}
$$\begin{array}{ccccccccccccccccccccccccccc} 
&&&&&&&&&& \text{\textcolor{red}{$M$}} &&&&&&&&&&&&&&&& \\
&&&&&&&&& \text{\textcolor{magenta}{$t$}} && \text{\textcolor{magenta}{$v$}} &&&&&&&&&&&&&&& \\ 
&&&&&&&&\textcolor{magenta}{\bf \rotatebox[origin=c]{90}{$\ddots$}}&&   &&\textcolor{magenta}{\bf \ddots}&&&&&&&&&&&&&& \\
&&&&&&& \textcolor{magenta}{\bf\rotatebox[origin=c]{90}{$\ddots$}}&&&&&&\textcolor{magenta}{\bf \ddots}&&&&&&&&&&&& \\
&&&&&&\textcolor{magenta}{\bf \rotatebox[origin=c]{90}{$\ddots$}}&&&&&&&&\textcolor{magenta}{\bf \ddots}&&&&&&&&&&&& \\
&&&&&\textcolor{magenta}{\bf \rotatebox[origin=c]{90}{$\ddots$}}&&&&&&&&&&\textcolor{magenta}{\bf \ddots}&&&&&&&&&&& \\\hline  
1 &&1 &&\textcolor{magenta}{\bf 1} &&1 &&1 && 1 &&1 &&1 && \textcolor{magenta}{\bf 1}&& 1&&1&& 1&&1 &&1 
\end{array}  $$
}
\caption{}\label{figw16}
\end{figure}

\medskip

(iv) Connect the magenta numbers by a line with the following rules and add a sign to each segment:

\medskip

\begin{enumerate}
\item[(S1)] Put signature minus $-$ on segment from ``$ M $" to the left ``$1$" (southwest-direction) in the floor 
and  signature plus $+$ on segment from ``$ M $" to the right ``$1$" (southeast-direction) in the ceiling.
\item[(S2)] Connect three numbers in the above with signed lines with the following rule.
Draw a line segment so that the number on the vertex of the triangle is the two numbers on the bottom base and determine its signature. 
Extend the line segment so that positive and negative line segments are output one by one from the top vertex in accordance with the signs of the right and the left ends.
\item[(S3)] On each line segment, replace plus with weight $-A^4$ and minus with weight $-A^{-4}$.
\item[(S4)] Compute the product of weights on each path from ``$M$" to the left ``$1$" or to the right  ``$1$". 
\item[(S5)] Sum the product of weights on each path from ``$M$" to the left ``$1$" or to the right  ``$1$". 
\end{enumerate}

\noindent 
Put $\textcircled{1}$ as $1$ in the floor and $\fbox{1}$ as $1$ in the ceiling of the curve, and 
$$\textrm{path}(\Gamma ):=\textrm{the decreasing path  from maximal  to $\textcircled{1}$, or $\fbox{1}$}.$$
Then, the Laurent polynomial by the above recipe coincides with  $\langle \Gamma  \rangle$, that is, 
\begin{equation}\label{eq5.1}
\langle \Gamma  \rangle =\sum_{ \gamma \in \textrm{path}(\Gamma ) } (-1)^{p( \gamma) + q( \gamma)}~ A^{4(p(\gamma) -q(\gamma))}, 
\end{equation}
where $p(\gamma), q(\gamma)$ mean the number of $+$'s , $-$'s respectively in the path $\gamma$. 
(cf. Theorem~\ref{w2-13}.)
\par 
We also set  
$$\textrm{path}(\Gamma )_{\textrm{num}}:=\textrm{the decreasing path  from maximal $M$ to $\textcircled{1}$}$$
and 
\begin{equation}\label{eq5.2}
\langle \Gamma \rangle_{num} :=\sum_{ \gamma \in \textrm{path}(\Gamma )_{\textrm{num}} } (-1)^{p( \gamma) + q( \gamma)} ~A^{4(p(\gamma) -q(\gamma))}. 
\end{equation}

\par \medskip 
Note that 
if $\frac{s}{M}=[a_0;a_1,\dots , a_n]$, then the Kauffman bracket polynomial of the rational link diagram $D(\frac{s}{M})$ is given by 
$$
\Bigl\langle D\Bigl(\frac{s}{M}\Bigr) \Bigr\rangle =(-A^{3})^{\sum_{i=0}^n (-1)^{i+1} a_i} \langle \Gamma \rangle
$$
by Theorem~\ref{w1-10}.  So, we have a formula
\begin{equation}\label{eq5.3}
\Bigl\langle D\Bigl(\frac{s}{M}\Bigr) \Bigr\rangle =(-A^{3})^{\sum_{i=0}^n (-1)^{i+1} a_i} \sum_{ \gamma \in \textrm{path}(\Gamma ) } (-1)^{p( \gamma) + q( \gamma)}~ A^{4(p(\gamma) -q(\gamma))}. 
\end{equation}

\par 
Similarly the Kauffman bracket polynomial $\langle D( \frac{u}{M} ) \rangle$ can be computed, and by using the ``cosine-curve" in the CCF,  
$\langle D( \frac{v}{M} ) \rangle$ and $\langle D( \frac{t}{M} ) \rangle$ can be also computed.

\medskip 
\noindent 
\begin{exam}\label{5.1}
We consider  the case of $\frac{s}{M}=\frac{7}{19}$. 
We focus a diamond surrounding $M=19$ in a fundamental domain $\mathcal{D}_1$ as follows:

\noindent 
$$\begin{array}{ccccc}
&& s && \\
&  7&  & b & \\ 
t && 19 && v \\
& c && d & \\
&& u && 
\end{array}$$

Then this satisfies the following relations:
\begin{align*}
7+d & = b+c=19, \\ 
s+t & = 7, \\
u+t & = c, \\
v+u & = d , \\
s+v & = b , \displaybreak[0]\\ 
7b-17s & =1 , \\
17t-7c & =1, \\
17v-bd & =1, \\
cd-17u & =1, \\
tv-su & =1.
\end{align*} 

\noindent 
We can solve the equations from the maximality of $19$ and have the CCF corresponding to $\frac{7}{19}$ as Figure~\ref{figw17}.

\begin{figure}[htbp]
\centering
\includegraphics[width=15cm]{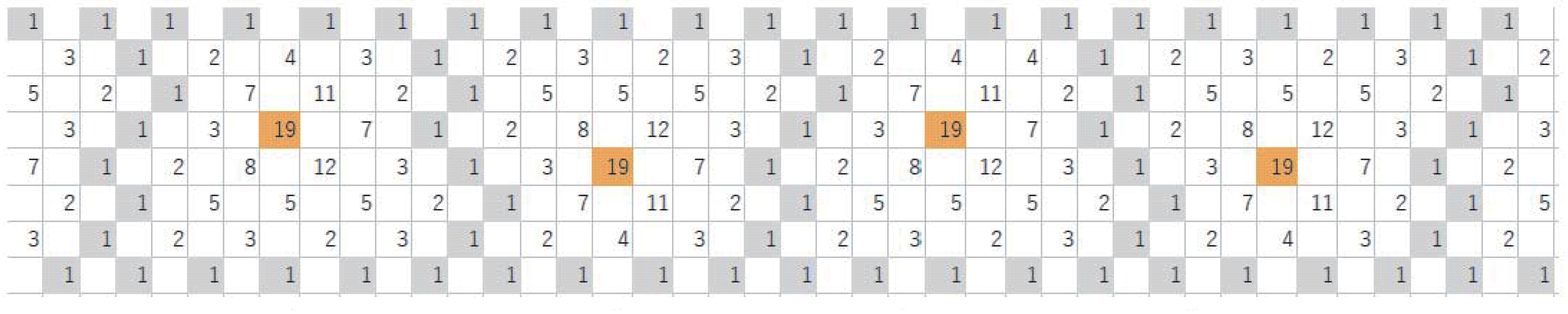} 
$~~~~~~~~\mathcal{D}(RL^2 RL)_1,~~~~~~~~\mathcal{D}(RL^2 RL)_2,~~~~~~~ \mathcal{D}(RL^2 RL)_3, ~~~~~~~~\mathcal{D}(RL^2 RL)_4$
\caption{}\label{figw17}
\end{figure}

\noindent 
Here, each domain surrounded by the gray block is a set such that it may be regarded as 
$\mathcal{D}(RL^2RL)_1=\mathcal{D}(RL^2RL)_3, ~\mathcal{D}(RL^2RL)_2$ $=\mathcal{D}(RL^2RL)_4$
from the left. 
At this time, the arrangement of numbers in each domain has the relationship as shown below.

$$\mathcal{D}(RL^2RL)_2= \rotatebox[origin=c]{180}{$\reflectbox{ $\mathcal{D}(RL^2RL)_1$ }$}$$
$$\mathcal{D}(RL^2RL)_4= \rotatebox[origin=c]{180}{$\reflectbox{ $\mathcal{D}(RL^2RL)_1$ }$}$$

Consider the yellow $(8 \rightarrow 19 \rightarrow 11) $- curve and the blue $( 7 \rightarrow 19 \rightarrow 12) $-curve passing through the maximum value ``$19$" 
in the $RL^2RL$-type Conway-Coxeter frieze. 
At the ``$19$" in $\mathcal{D}(RL^2RL)_1$ then 
the yellow $(8 \rightarrow 19 \rightarrow 11 )$-curve is minus until it passes through it, plus up to the next ``$19$", minus after exceeding it, 
while the blue $(7 \rightarrow 19 \rightarrow 12 )$-curve is minus up to ``$19$" in $\mathcal{D}(RL^2RL) _1$. 
Before exceeding it, it increases to the next ``$19$", and after exceeding it, it is minus, and so on (see Figure~\ref{figw19}).

\begin{figure}[htbp]
\centering
\includegraphics[width=12cm]{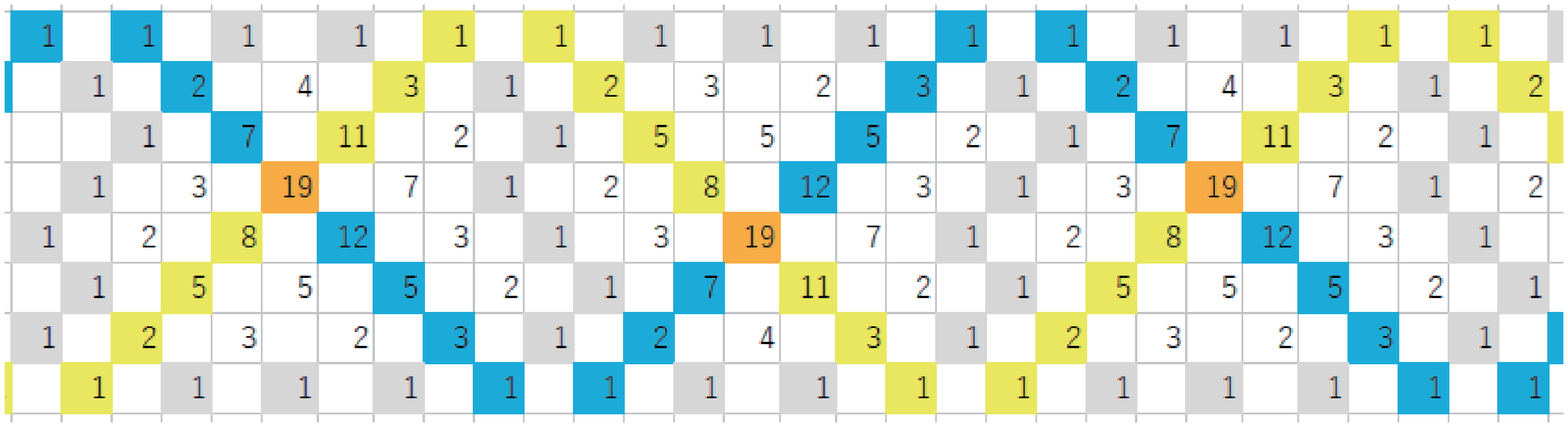} 
\caption{}\label{figw18}
\end{figure}

\medskip 
\begin{figure}[htbp]
\centering
\includegraphics[width=12cm]{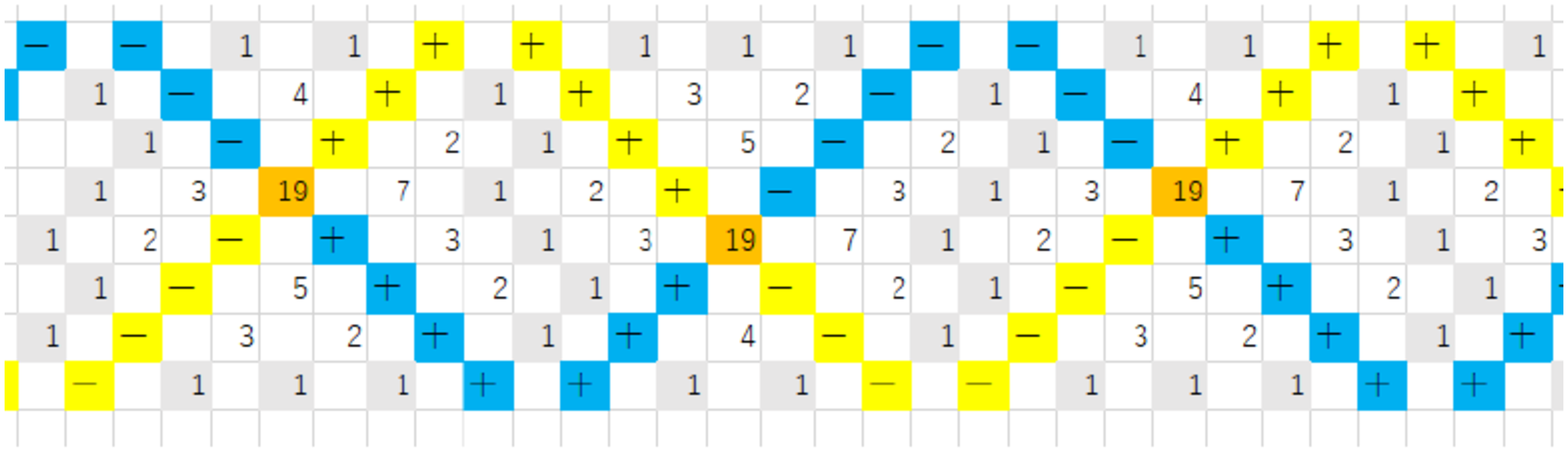} 
\caption{}\label{figw19}
\end{figure}

Let us focus on the yellow line. 
\begin{enumerate}
\item[$(i)$] At the maximum value ``$19$", this upward line is folded as shown in Figure~\ref{figw20}. 
\item[$(ii)$] Draw a line so that the number at the vertex of the triangle is the sum of the numbers at the bottom two vertices. 
\item[$(iii)$] Add signs so that a line segment with plus and minus appears from the vertex of the triangle.
(The signs on the left end of the large triangle are all minus, and the signs on the right end are all plus, so the signs of other line segments are determined naturally.)
\end{enumerate}

\begin{figure}[htbp]
\centering
\includegraphics[width=5.5cm]{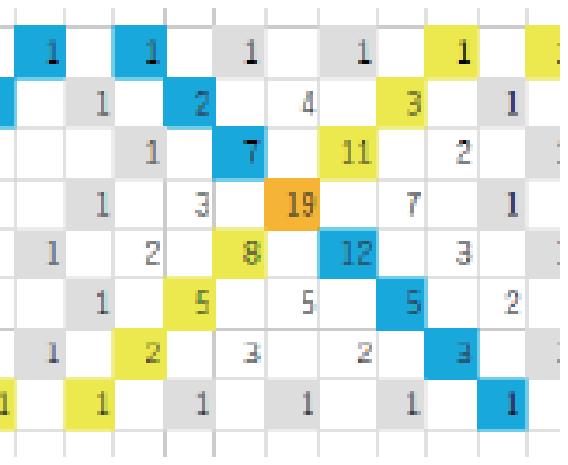} ~~
\includegraphics[width=7cm]{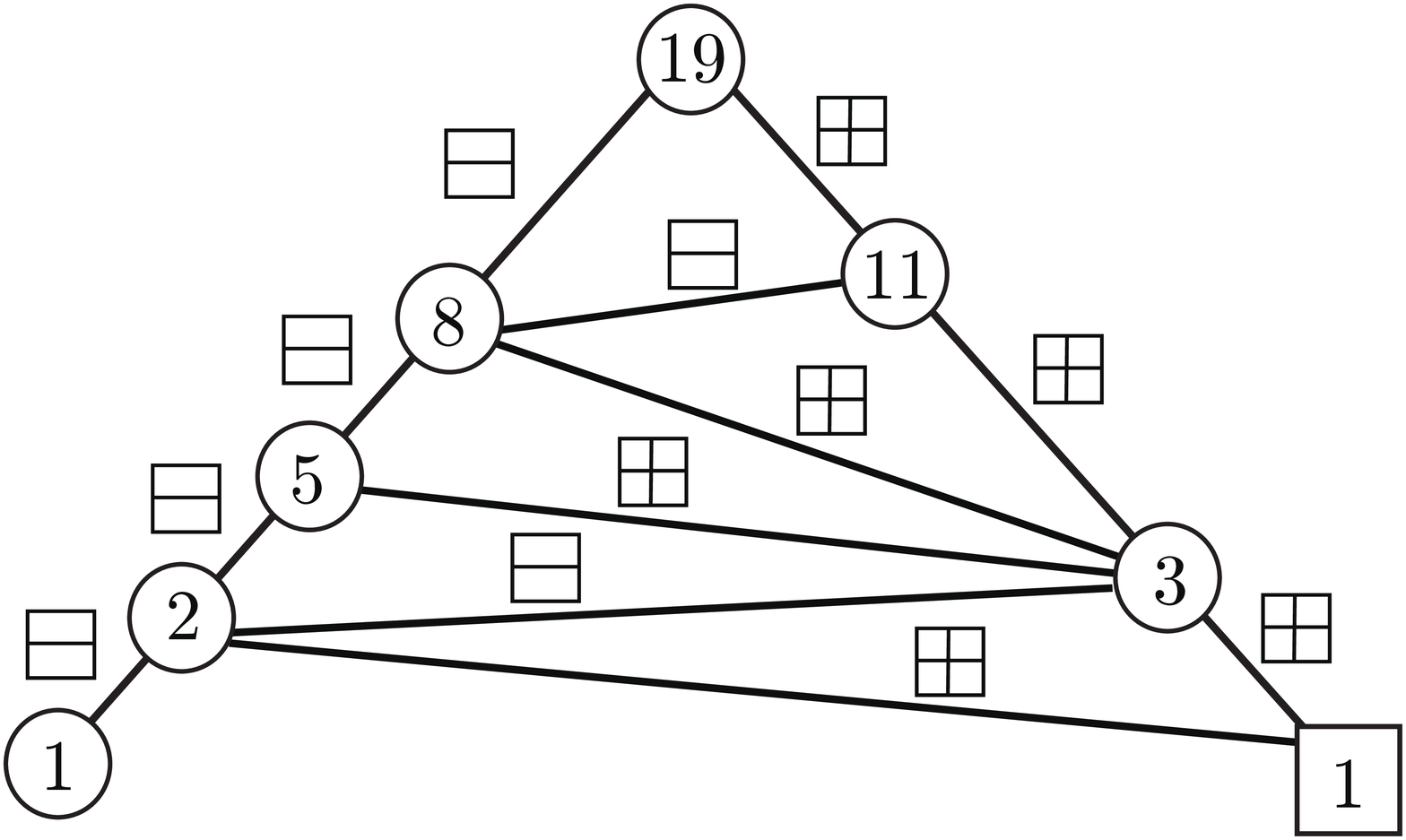} 
\caption{}\label{figw20}
\end{figure}

Next, consider the path going from ``$19$" to the lower left ``$1$"(=$\textcircled{1}$) and the lower right ``$1$" (=\fbox{$1$}), starting with ``$19$", followed by a path whose numbers decrease, which is shown graphically on the right side of Figure~\ref{figw21}.

\begin{figure}[htbp]
\centering
\includegraphics[width=6cm]{RL2RL-tri_arranged.eps}
\includegraphics[width=7cm]{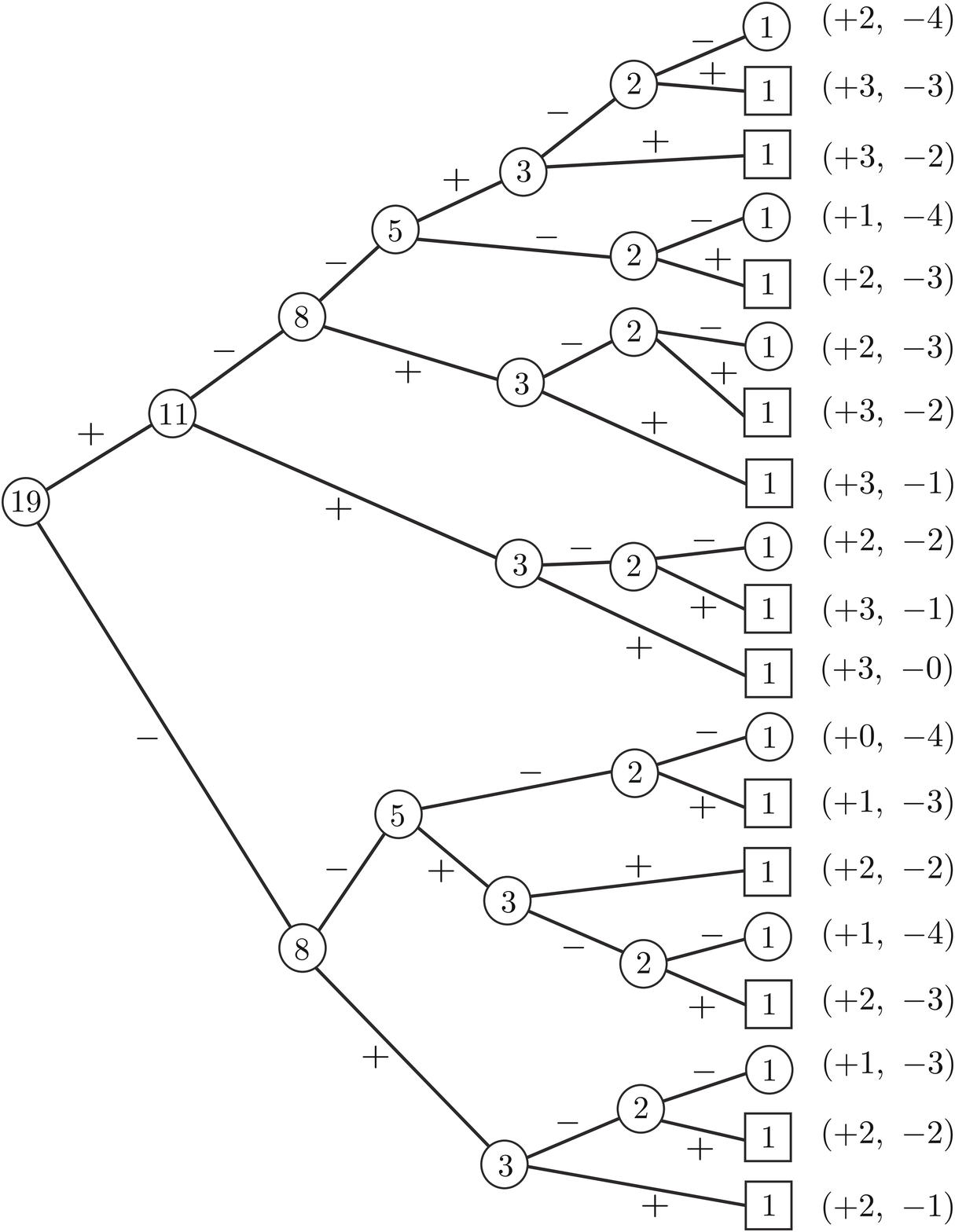} 
\caption{}\label{figw21}
\end{figure}

Next, for each path, substitute plus $+$ for weight $- A^{4}$ and minus $-$ for weight $-A^{-4} $, and compute the product of those weights.
\par 
Finally, sum their signed $ A $-monomials for all paths.

\medskip 
\begin{figure}[htbp]
\centering
\includegraphics[width=6cm]{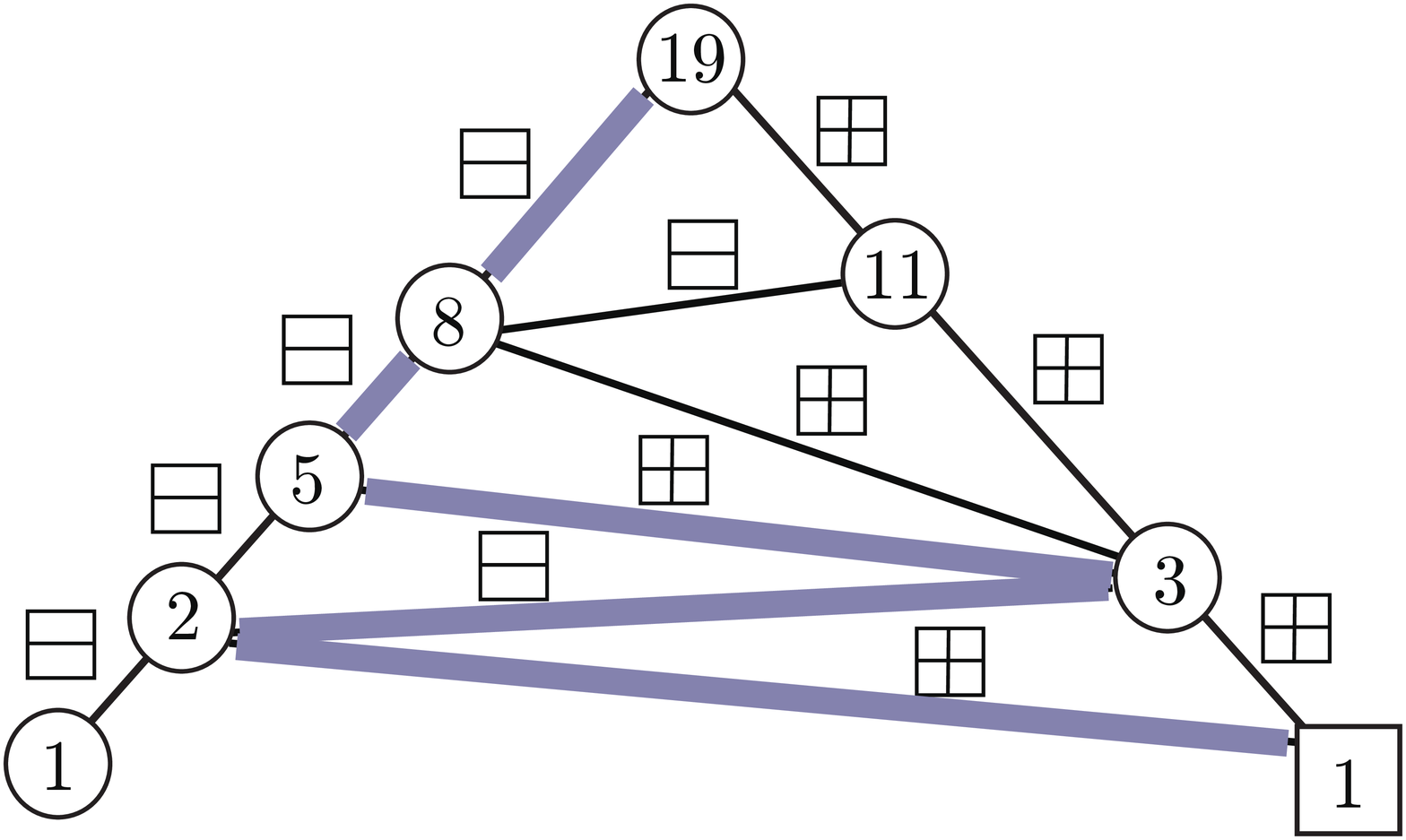}
\includegraphics[width=7cm]{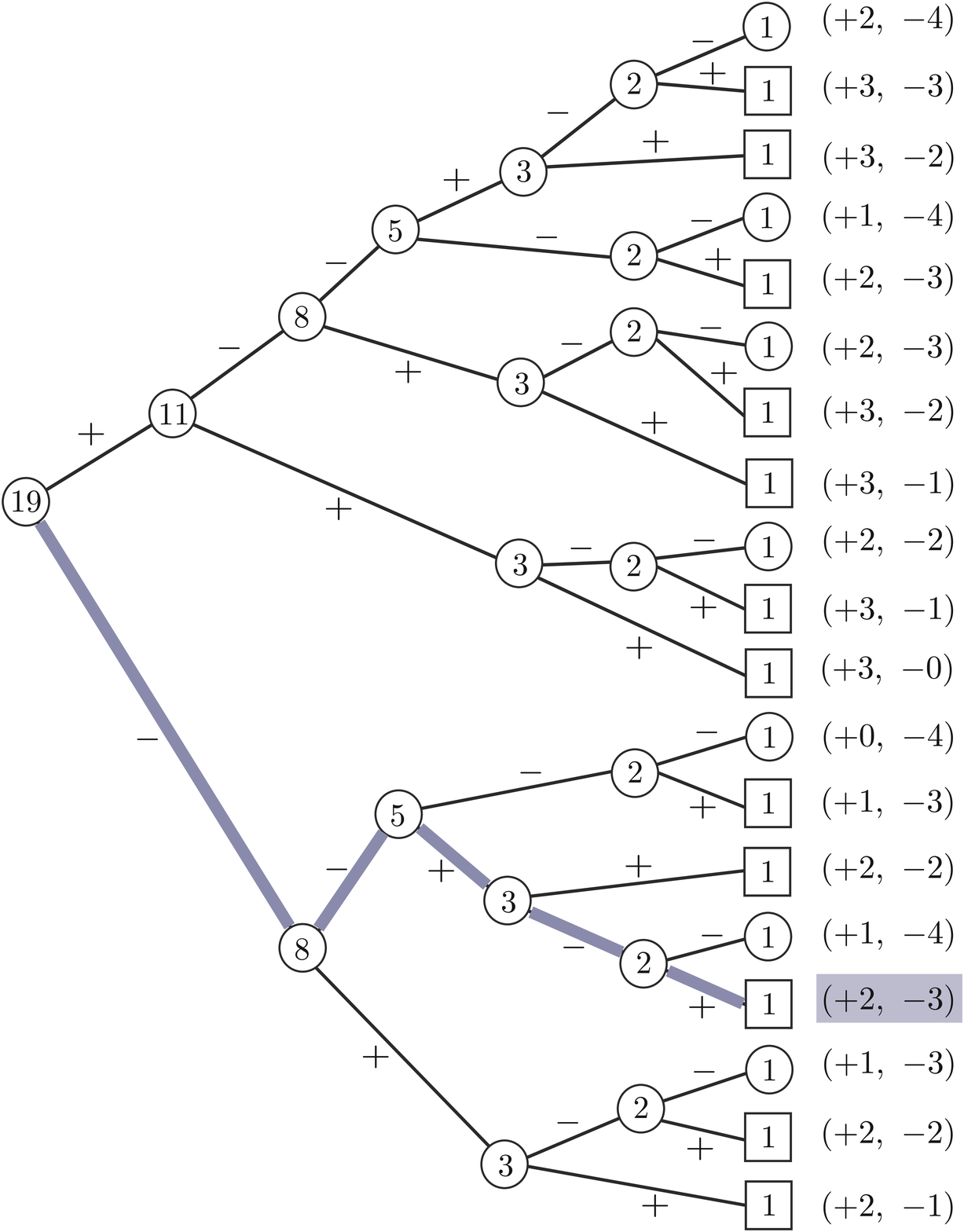} 
\caption{}\label{figw22}
\end{figure}

\noindent 
For example, 

\noindent 
Path $\gamma$:$~~19 \xrightarrow{-} 8 \xrightarrow{-} 5 \xrightarrow{+} 3 \xrightarrow{-}  2 \xrightarrow{+} \fbox{$1$} $ 

$\Rightarrow $

Path $\gamma$:$~~19 \xrightarrow{-A^{-4}} 8 \xrightarrow{-A^{-4}} 5 \xrightarrow{-A^4} 3 \xrightarrow{-A^{-4}}  2 \xrightarrow{-A^4} \fbox{$1$} $ 

\medskip

\noindent 
Path $\gamma$-monomial: $(-A^{-4}) \cdot (-A^{-4}) \cdot (-A^4) \cdot (-A^{-4}) \cdot (-A^4)=(-1)^{2+3} A^{4\cdot (2-3)}=-A^{-4}$

\medskip 

\noindent 
In general, for the corresponding signature $(+p,-q)$ on each path $\gamma$, associate a monomial 
$(-1)^{p+q} A^{4(p-q)}$. 

By \eqref{eq5.1}  we have 
\par \medskip 
\centerline{$\langle \Gamma ( RL^2RL) \rangle = -A^{12} +2A^8-3 A^4 +4 -3 A^{-4} +3 A^{-8} -2A^{-12} +A^{-16}$. }

\begin{figure}[htbp]
\centering
\includegraphics[width=7cm]{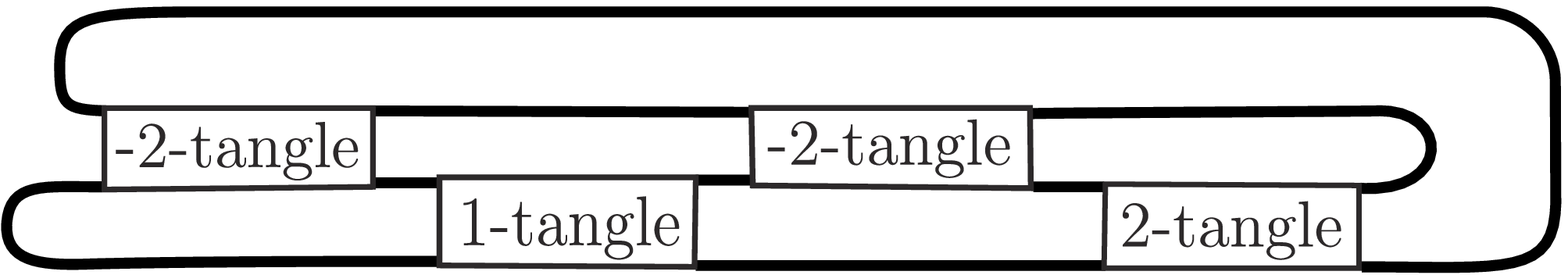} \ \ \raisebox{0.6cm}{$=$}\ \ 
\raisebox{-0.2cm}{\includegraphics[width=5cm]{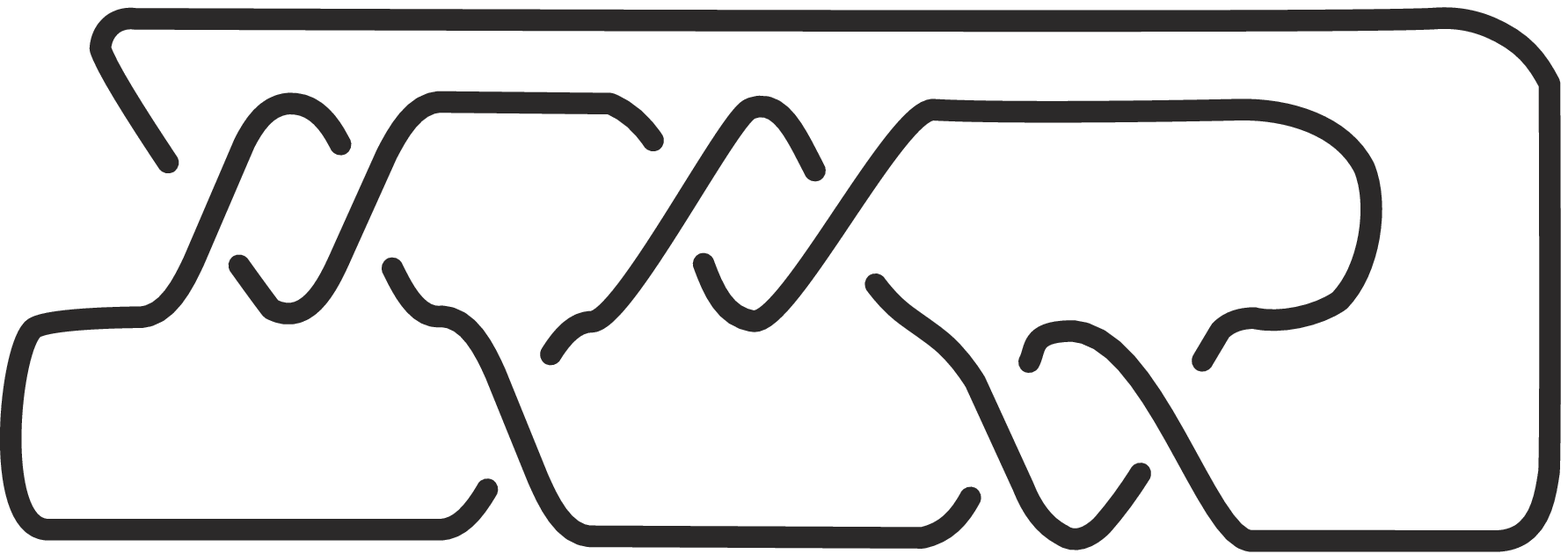}}
\caption{}\label{figw23}
\end{figure}

\par \medskip 
Remark that $A^{-16}= A^{-4(3+1)}=(-1)^{|RL^2 RL|_L +1}A^{-4 (|RL^2RL|_L +1)}$,  
$-A^{12}=(-1) A^{4 \times 3}=(-1)^{|RL^2 RL|_R +1}A^{4(|RL^2RL|_R +1)} $, and the 
sequence of coefficients $1,2,3,4,3,3,2,1$ is increasing and decreasing. 
From $\langle \Gamma (RL^2RL) \rangle$, by \eqref{eq5.3} we can get the  Kauffman bracket polynomial of the knot related to the  fraction $\frac{7}{19}=[0;2,1,2,2]$: 
\begin{align*}
\raisebox{0.6cm}{$\langle D(\frac{7}{19}) \rangle $}\ 
& \ \raisebox{0.6cm}{$=\Biggl\langle~ $}\includegraphics[width=4cm]{knot_7-19_01_arranged.eps}  ~\raisebox{0.6cm}{$\Biggr\rangle $} \\[-0.1cm]   
&=(-A^{-3})^{\sum_{i=0}^n (-1)^{i+1} a_i} \langle \Gamma (RL^2RL) \rangle \\ 
&=A^{15}-2A^{11}+3A^7-4A^{3}+3A^{-1}-3A^{-5}+2A^{-9}-A^{-13}.
\end{align*}

On the other hand 

\medskip 
\noindent 
$\left\{ \begin{array}{l} 
\langle \Gamma ( RL^2RL) \rangle =-A^{12} +2A^8-3 A^4 +4 -3 A^{-4} +3 A^{-8} -2A^{-12} +A^{-16}, \\
\langle \Gamma ( RL^2RL) \rangle_{\textrm{num}} =1-A^{-4}+2 A^{-8} -2 A^{-12} +A^{-16}.  
\end{array} \right. $

\medskip 
By substituting $A^4=-1$ we have
\medskip 

\noindent 
$\left\{ \begin{array}{ll} 
\langle \Gamma ( RL^2RL)  \rangle  &\xrightarrow{A^4 =-1} 19=\textrm{the denominator of the fraction $ \frac{7}{19}$},  \\
\langle \Gamma ( RL^2RL) \rangle_{\textrm{num}}  & \xrightarrow{A^4 =-1} 7=\textrm{the numerator of the fraction $ \frac{7}{19}$}, 
\end{array} \right. $

\medskip 
\noindent 
namely, 
$\frac{ \langle \Gamma ( RL^2RL) \rangle_{\textrm{num}} }{ \langle \Gamma ( RL^2RL) \rangle }
=\frac{1-A^{-4}+2 A^{-8} -2 A^{-12} +A^{-16}  }{-A^{12} +2A^8-3 A^4 +4 -3 A^{-4} +3 A^{-8} -2A^{-12} +A^{-16}}\ \xrightarrow{A^4=-1} \frac{7}{19}$. 
 
\medskip 
\begin{figure}[htbp]
\centering
\includegraphics[width=15cm]{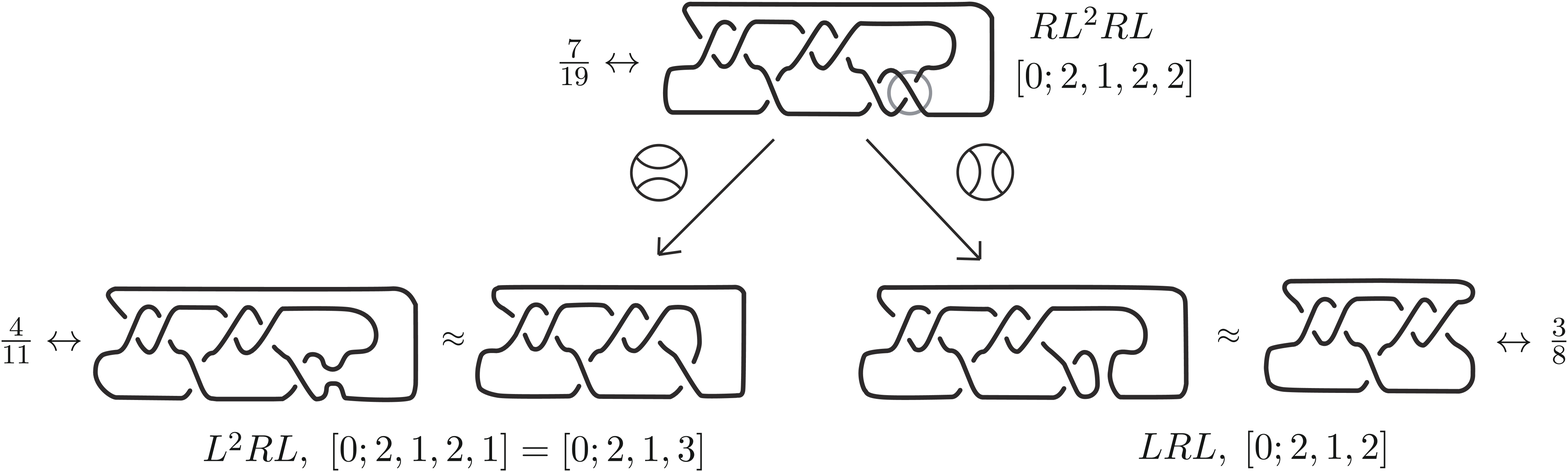} 
\caption{}\label{figw25}
\end{figure}

\medskip 
We note that the Farey sum $\frac{7}{19} = \frac{4}{11} \sharp \frac{3}{8}$ is related to the lateral removal and the vertical extension of the intersection of the rational entanglement as shown in Figure~\ref{figw25}, where  $~\approx~ $ means 
that the corresponding Kauffman bracket polynomials match except for a power multiplying factor of A except.
Here, the lateral removal of the three-dimensional intersection at the lower right of this $\frac{7}{19}$-knot eliminates the blue $(7 \rightarrow 19 \rightarrow 12 )$-curve passing through ``$19$" in the Conway-Coxeter frieze  $\Gamma (RL^{2}RL)$, compresses the empty space diagonally and converts one size small CCF.
The created CCF becomes the $L^{2}RL$ - type with ``$11$" as the maximum value, which corresponds to $\frac{4}{11}$. 
Similarly, considering the $(* \rightarrow  11 \rightarrow  * )$-curve and the $( * \rightarrow ,11, \rightarrow *)$-curve, the $(*\rightarrow  11 \rightarrow * )$-curve passing through ``$11$" is removed and compressed, and furthermore, it becomes the LRL-type CCF whose maximum value is one size small ``$8$". 
It corresponds to $\frac{3}{8}$. 


\begin{figure}[htbp]
\vspace*{-2cm}
\begin{align*}
\raisebox{0cm}{$RL^2RL$}  & \ \ \includegraphics[width=8cm]{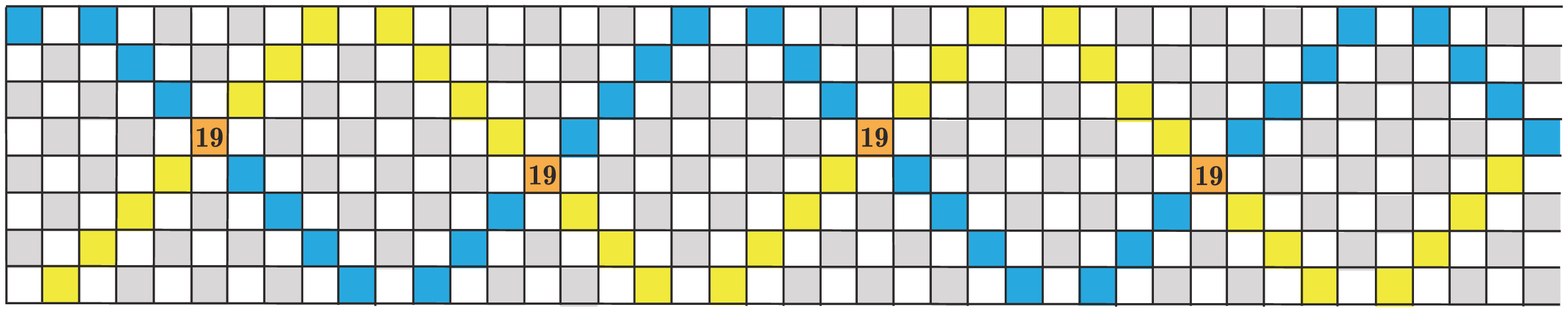} & \raisebox{-0.6cm}{\includegraphics[width=3cm]{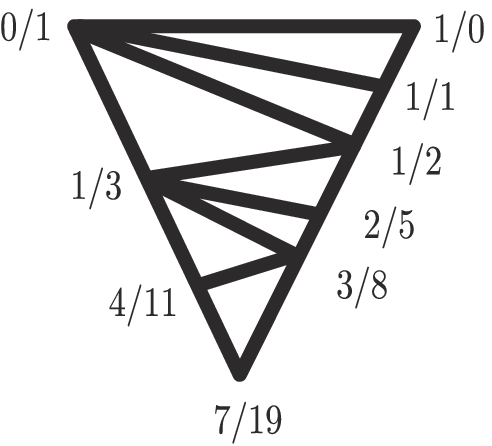}} & \  \ \raisebox{0cm}{$\dfrac{7}{19}$} \\  
\raisebox{0cm}{$RL^2RL$-\text{delete}}  & \ \  \includegraphics[width=8cm]{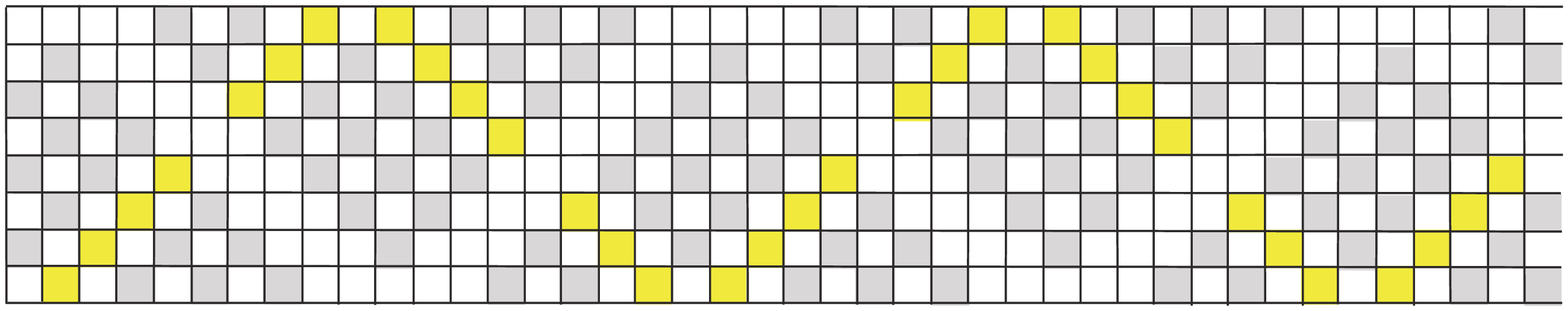} &  & \\[0.3cm]       
\raisebox{0cm}{$L^2RL$} & \ \ \includegraphics[width=8cm]{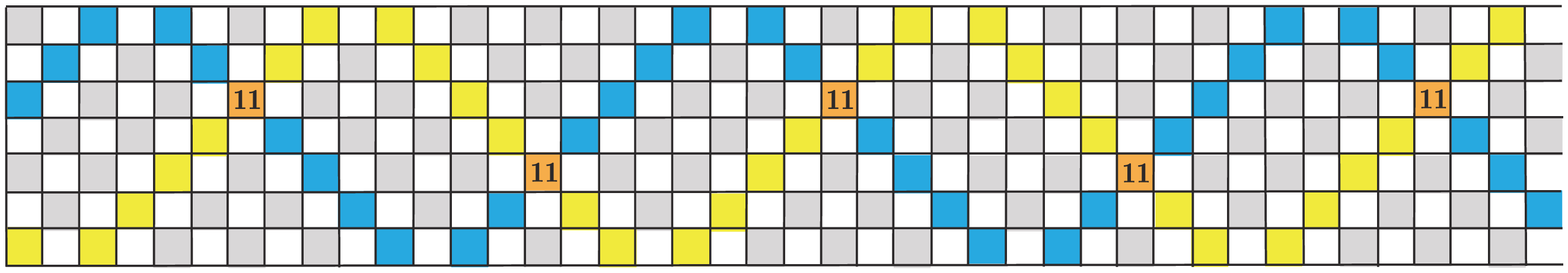} &  \  \  \raisebox{-0.5cm}{\includegraphics[width=2cm]{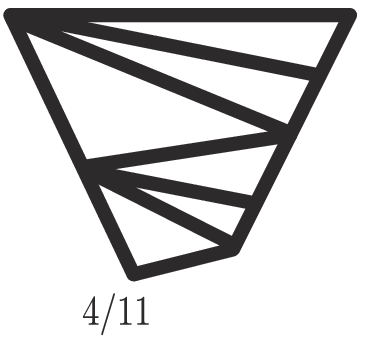}} & \ \ \raisebox{0cm}{$\dfrac{4}{11}$} \\   
\raisebox{0cm}{$L^2RL$-\text{delete}}  & \ \  \includegraphics[width=8cm]{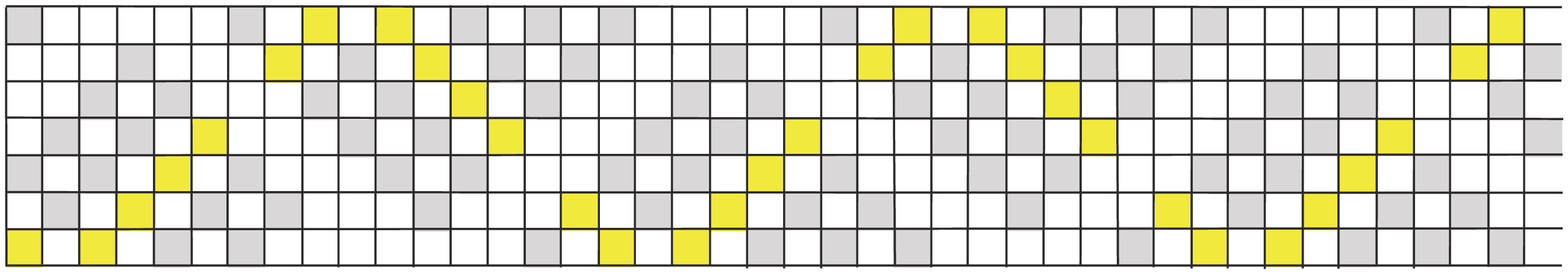} &  & \\[0.3cm]   
\raisebox{0cm}{$LRL$}  & \ \  \includegraphics[width=8cm]{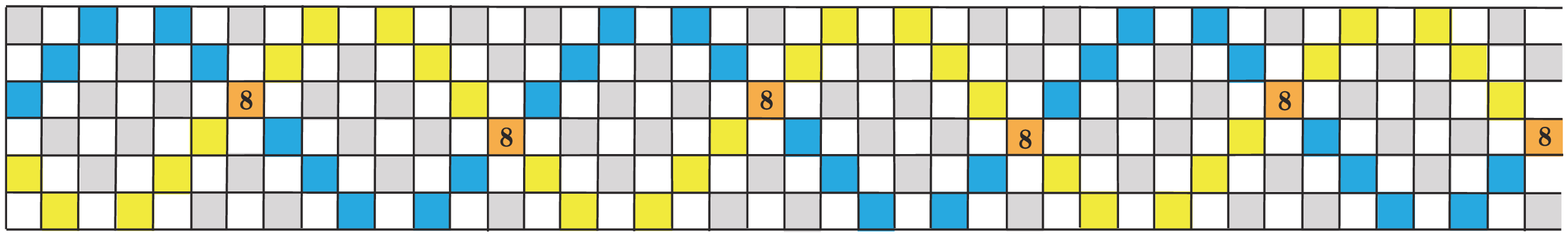} &  \raisebox{-0.4cm}{\includegraphics[width=2cm]{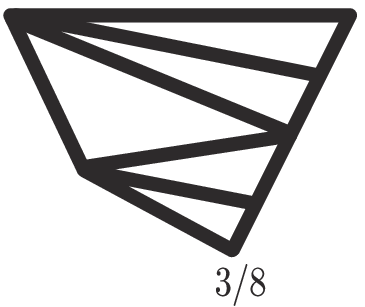}} & \  \  \raisebox{0.3cm}{$\dfrac{3}{8}$} \\[0.3cm]    
\raisebox{0cm}{$LRL$-\text{delete}}  & \ \  \includegraphics[width=8cm]{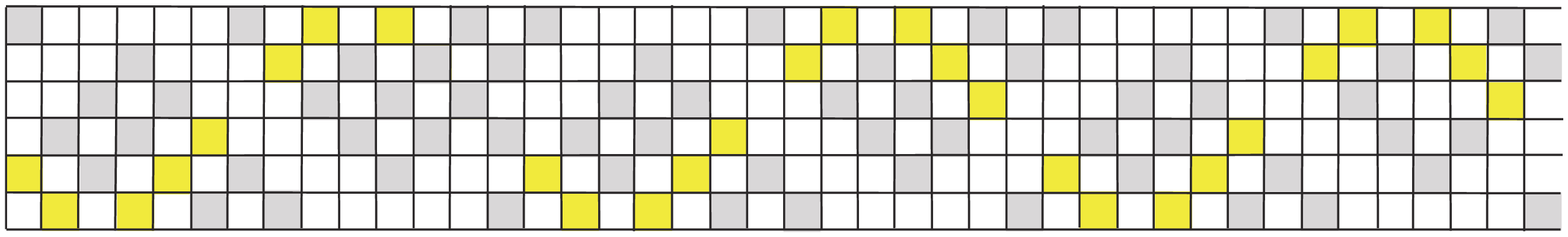} &  & \\[0.3cm]     
\raisebox{0cm}{$RL$} \ & \ \  \includegraphics[width=8cm]{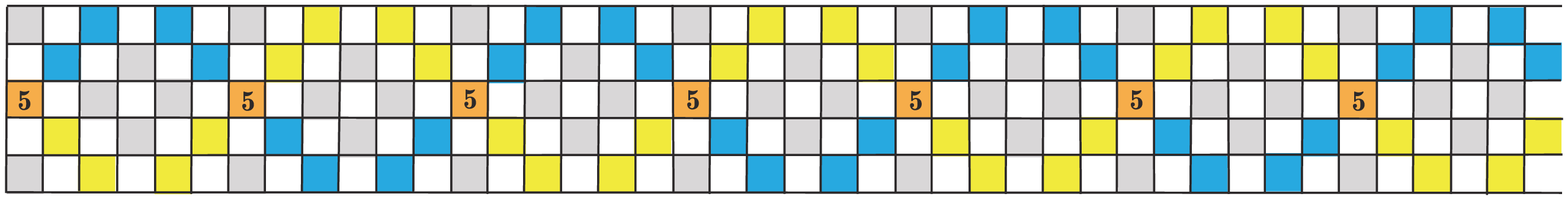} &  \raisebox{-0.4cm}{\includegraphics[width=2cm]{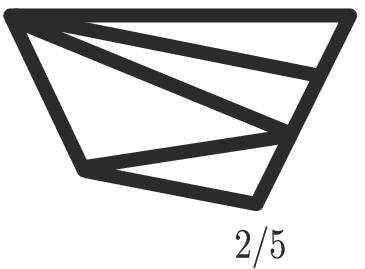}} & \  \  \raisebox{0.3cm}{$\dfrac{2}{5}$} \\[0.3cm]   
\raisebox{0cm}{$RL$-\text{delete}}  & \ \  \includegraphics[width=8cm]{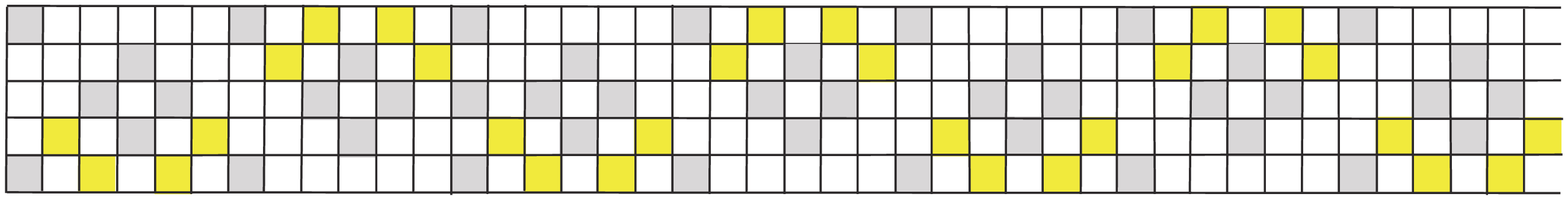} &  & \\[0.3cm]    
\raisebox{0cm}{$L$} \ & \ \  \includegraphics[width=8cm]{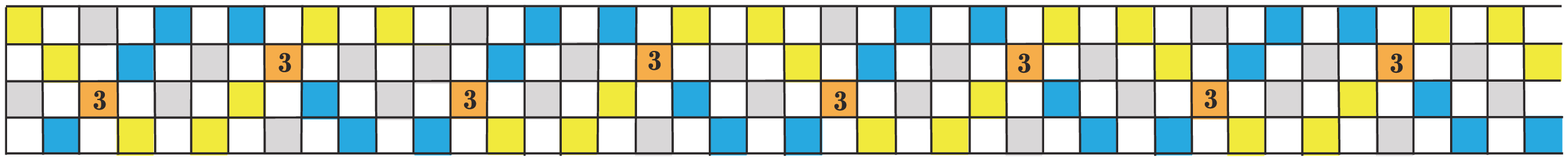} &  \raisebox{-0.4cm}{\includegraphics[width=2cm]{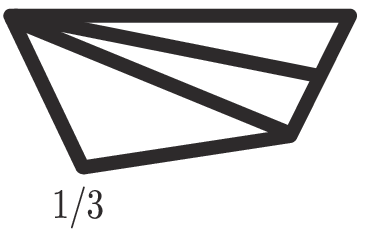}} & \  \ \raisebox{0.3cm}{$\dfrac{1}{3}$} \\[0.3cm]    
\raisebox{0cm}{$L$-\text{delete}}  & \ \  \includegraphics[width=8cm]{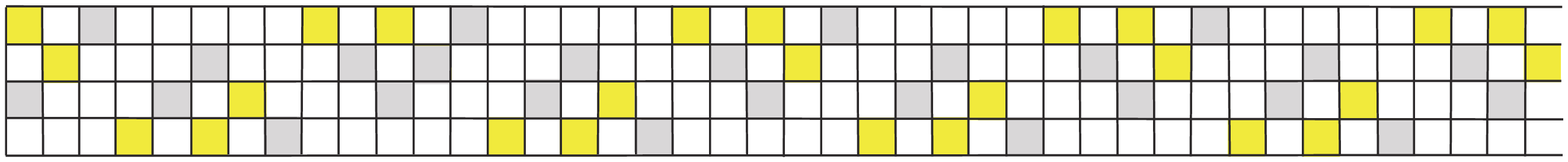} &  & \\[0.3cm]     
\raisebox{0cm}{$\emptyset$}  & \ \  \includegraphics[width=8cm]{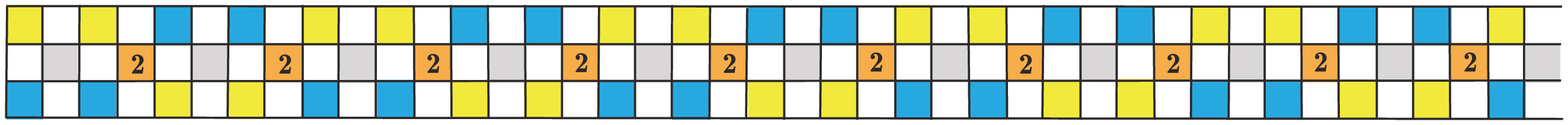} &  \raisebox{-0.4cm}{\includegraphics[width=2cm]{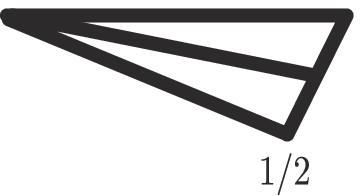}} & \  \  \raisebox{0cm}{$\dfrac{1}{2}$}
\end{align*}
\caption{}\label{figw26}
\end{figure}

Repeating this operations corresponds to repeating the operation of deleting ``cosine-curve" from the CCF of type $RL^2RL$ and of removing vertices from the associated Yamada's ancestor triangle of $\frac{7}{19}$ as Figure~\ref{figw26}. 
\end{exam}

\medskip 
The above process of deleting ``$* \rightarrow * \rightarrow *$-curves" yields a braid, which can be regarded as a link on a torus.  
We explain it by using the CCF of type $RL^2RL$ examined in Example~\ref{5.1}. 
In this case, the braid is given in Figure~\ref{figw27}.
It is created by clipping a part between two lines from top left to bottom right through $19$ in the CCF. 
\par 
Each strand labeled by an integer $k\ (k=19, 11, 8, 5, 3, 2, 1)$ in the braid corresponds to the ``$(* \rightarrow * \rightarrow *)$-curve" passing through $k$ in the CCF $\Gamma (RL^2RL)$. 
The crossings on the strand labeled by $19$ are determined as follows. 
From the left edge to $19$ in the middle all crossings are positive, and the rest are negative. 
For other crossings are inductively determined by the same manner. 
\par 
In the language of Yamada's ancestor triangle, 
removing the strand ``$19$" in the braid corresponds to deleting the bottom triangle from the Yamada's ancestor triangle of $\frac{7}{19}$. 
After the operation, we have a new braid. 
Then removing the strand ``$11$" from the new braid corresponds to deleting the bottom triangle from the Yamada's ancestor triangle of $\frac{4}{11}$, and so on. 
\par 
In this way, one have a braid from any CCF of zigzag-type. 
Such a braid is pure, and has the following property. 
For any two strands of the braid there are exactly two crossings with linking number $-2$. 
Detailed studies on these braids would be appeared elsewhere as a subsequent article. 

\begin{figure}[htbp]
\centering
\includegraphics[width=15cm]{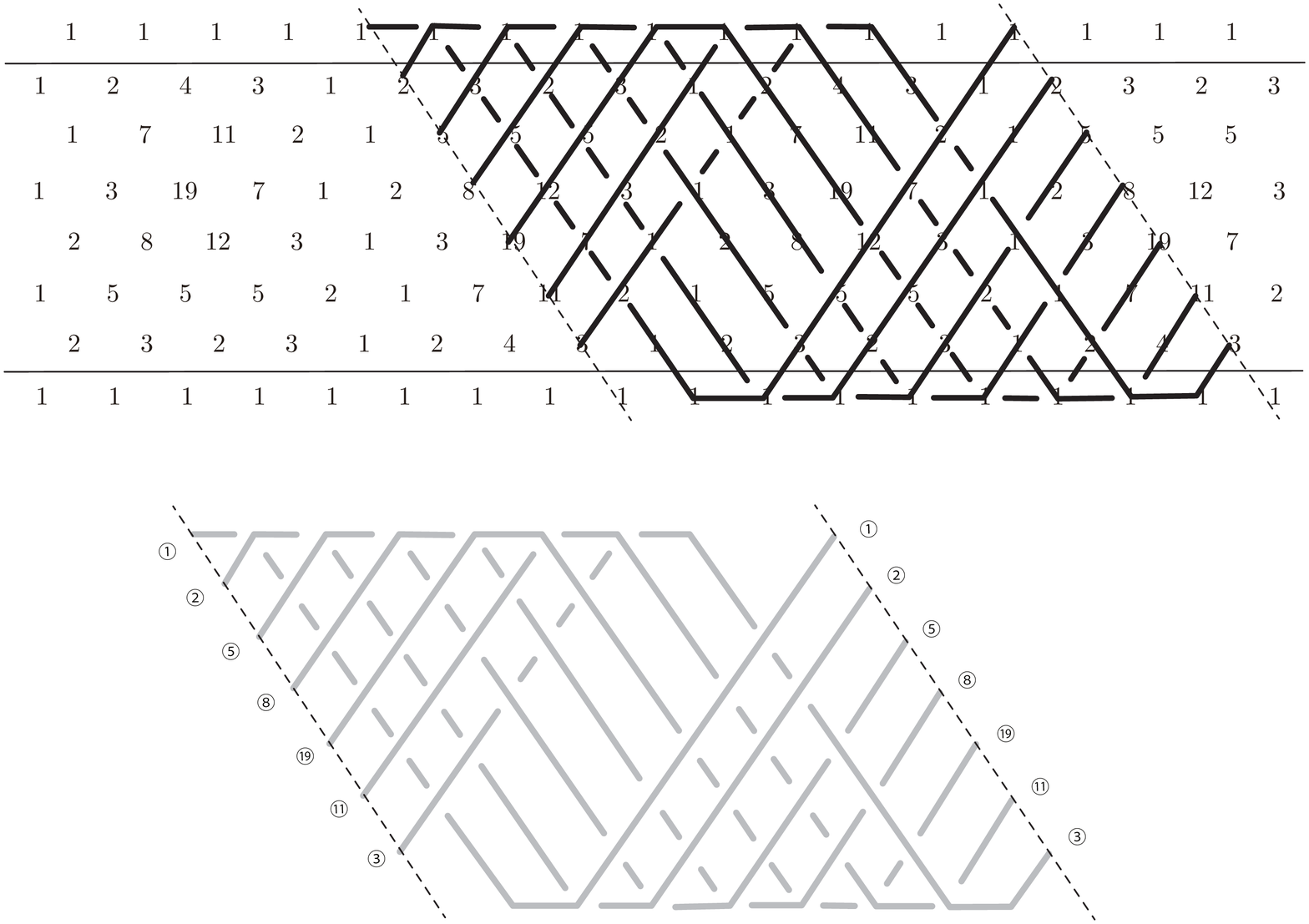} 
\caption{}\label{figw27}
\end{figure}

\par \bigskip \noindent 
{\bf Acknowledgments}.  
We would like to thank Professor Shuji Yamada for his kind message about his formula given in \cite{Yamada-Proceeding}. 
We would like to thank Professor Jun Murakami and Professor Yuji Terashima for useful comments on Yamada's ancestor triangles and rational knots. 
Thanks are also due to Professor Ralf Schiffler for informing us that our research is closely related to the work in the paper \cite{LS} with detailed explanation. 
We express appreciation to Dr. Shigeto Wada who created the calculation table of CCF by Excel.


\newpage 


\begin{thebibliography}{99}

\bibitem{Adams}
C.C. Adams,
\textit{The knot book},
W.H. Freeman \& Co., New York, 1994. 

\par \noindent 
\bibitem{Aigner}
M. Aigner, 
\textit{Markov's Theorem
and 100 Years of the Uniqueness Conjecture},
Springer, 2013. 

\par \noindent 
\bibitem{Brocot}
A. Brocot, 
\textit{Calcul des rouages par approximation,
Nouvelle m\'{e}thode}, 
Revue chronome\'{e}trique {\bf 3} (1861), 186--194.

\par \noindent 
\bibitem{Conway}
J.H. Conway,
\textit{An enumeration of knots and links, and some of their algebraic properties}, 
in Proceedings of the conference on computational problems in abstract algebra held at Oxford 1967, 
(J. Leech ed.),  Pergamon Press, 1970, 329--358. 

\par \noindent 
\bibitem{BS}
F. Bonahon and L.C. Siebenmann,
\textit{New geometric splittings of classical knots and the classification and symmetries of arborescent knots}, 
June 12, 2010. 

\bibitem{CoCo1}
J.H. Conway, H.S.M. Coxeter, 
\textit{Triangulated polygons and frieze patterns}, 
Math. Gaz. {\bf 57} (1973), no. 400, 87--94.

\bibitem{CoCo2}
J.H. Conway, H.S.M. Coxeter, 
\textit{Triangulated polygons and frieze patterns II}, 
Math. Gaz. {\bf 57} (1973), no. 401, 175--183. 

\par \noindent 
\bibitem{Coxeter}
H.S.M. Coxeter,
\textit{Frieze patterns},
Acta Arith. {\bf 18} (1971), 297--310. 

\par \noindent 
\bibitem{Cromwell}
P. Cromwell,
\textit{Knots and links}, 
Cambridge University Press, 2004.

\par  \noindent 
\bibitem{GK2}
J.R. Goldman and L.H. Kauffman, 
\textit{Rational tangles}, 
Adv. Appl. Math. {\bf 18} (1997), 300--332. 


\par  \noindent 
\bibitem{FWZ}
S. Fomin, L.Williams and A.Zelevinsky, 
\textit{Introduction to Cluster algebras Chapters 1--3, Chapter 4--5}, 
https:\slash/arxiv.org\slash abs\slash1608.05735 ,
https:\slash\slash arxiv.org \slash abs \slash 1707.07190.

\par \noindent 
\bibitem{FZ1}
S. Fomin and A. Zelevinsky, 
\textit{Cluster algebras. I. Foundations}, 
J. Amer.Math. Soc. {\bf 15} (2002), 497--529.

\par \noindent 
\bibitem{FZ2}
S. Fomin and A.  Zelevinsky, 
\textit{Cluster algebras. II. Finite type classification},
Invent. Math. {\bf 154} (2003), 63--121.

\par  \noindent 
\bibitem{Kanenobu}
T. Kanenobu, 
\textit{Jones and Q-polynomials for $2$-bridge knots and links}, 
Proc. Amer. Math. Soc. {\bf 110} (1990), 835--841. 

\par \noindent 
\bibitem{Kau-Topology}
L.H. Kauffman, 
\textit{State models and the Jones polynomial}, 
Topology {\bf 26} (1987), 395--407. 

\par  \noindent 
\bibitem{KL3}
L.H. Kauffman and S. Lambropoulou, 
\textit{On the classification of rational tangles}, 
Adv. Appl. Math. {\bf 33} (2004), 199--237. 

\par  \noindent 
\bibitem{KW1}
T. Kogiso and M. Wakui,
\textit{Kauffman bracket polynomials of Conway-Coxeter Friezes}, 
in \lq\lq Proceedings of meeting for study of number theory, Hopf algebras and related topics", 
edited by H. Yamane, T. Kogiso, Y. Koga and I. Kimura, Yokohama Publ., 2019, 51--79. 

\par  \noindent 
\bibitem{LLS}
E. Lee, S.Y. Lee and M. Seo, 
\textit{A recursive formula for the Jones polynomial of $2$-bridge links and applications}, 
J. Korean Math. Soc. {\bf 46} (2009), p.919--947. 

\par  \noindent 
\bibitem{LS}
K. Lee and R. Schiffler, 
\textit{Cluster algebras and Jones polynomials}, 
arXiv:1710.08063v2, [math.GT], 15 Nov. 2017. 

\par \noindent 
\bibitem{M-G}
S. Morier-Genoud, 
\textit{Coxeter's frieze patterns at the crossroads of algebra, geometry and combinatorics},
Bull. London Math. Soc. {\bf 47} (2015), 895--938. 

\par \noindent 
\bibitem{Nakabo1}
S. Nakabo, 
\textit{Formulas on the HOMFLY and Jones polynomials of 2-bridge knots and links}, 
Kobe J. Math. {\bf 17} (2000), 131--144. 

\par\noindent 
\bibitem{Nakabo2}
S. Nakabo, 
\textit{Explicit description of the HOMFLY polynomials for 2-bridge knots and links}, 
JKTR {\bf 11} (2002), 565--574. 

\par  \noindent 
\bibitem{NT}
W. Nagai and Y. Terashima, 
\textit{Cluster variables, ancestral triangles and Alexander polynomials}, 
arXiv:1812.02434v1, [math.GT], 6 Dec. 2018. 

\par  \noindent 
\bibitem{QY-AQ}
K. Qazaqzeh, M. Yasein and M. Abu-Qamar, 
\textit{The Jones polynomial of rational links}, 
Kodai Math. J. {\bf 39} (2016), 59--71.

\par \noindent 
\bibitem{Stern}
M.A. Stern, 
\textit{Ueber eine zahlentheoretische Funktion}, 
Journal f\"{u}r die reine und angewandte Mathematik {\bf 55} (1858), 193--220.

\par\noindent 
\bibitem{Yamada-Proceeding}
S. Yamada, 
\textit{Jones polynomial of two-bridge knots (Ni-hashi musubime no Jones takoushiki)}, in Japanese, 
in Proceedings of \lq\lq Musubime no shomondai to saikin no seika", 1996, 92--96. 

\end{thebibliography}
\end{document}